\documentclass{imsart}
\oddsidemargin
-0.0in \topmargin -0.5in
 \usepackage[utf8]{inputenc}

\usepackage{stackengine}

\usepackage[textsize=footnotesize]{todonotes}
\usepackage{xcolor}
\usepackage{hyperref}
\RequirePackage[numbers]{natbib}

\usepackage{mathrsfs}
\usepackage{graphicx}
\usepackage{amsmath}        
\usepackage{amssymb}
\usepackage{amsthm}
\usepackage{bm}
\usepackage[mathscr]{euscript}
   
        \usepackage[countmax]{subfloat}  
\usepackage{fourier}
\usepackage{xcolor}
\usepackage{mathtools}
\usepackage{enumerate}
\usepackage{enumitem}
\counterwithin{figure}{section}
\usepackage{natbib}
\usepackage{url}
\usepackage{verbatim}
\usepackage{booktabs}
\usepackage{placeins}
\usepackage{multirow}
\usepackage{subcaption}
\usepackage[cmtip,all]{xy}
\usepackage{aliascnt}
\usepackage[font=small]{caption}

\hypersetup
{
    pdfauthor={Anne van Delft},
    pdfkeywords={functional data, time series, spectral analysis,martingales}
}
\newcounter{alphcount}
{\begin{list}{{\upshape(}\alph{alphcount}\/{\upshape)\ }}%
             {\usecounter{alphcount}\labelwidth1.5em%
              \bigmargin2em\labelsep0.5em\topsep0.25em plus 0.5ex%
              \itemsep0.25em plus 0.5ex\parsep0em}}{\end{list}}
{\begin{list}{{\upshape(#1\arabic{alphcount})\hfill}}%
             {\usecounter{alphcount}\labelwidth2.5em%
              \bigmargin2.5em\labelsep0em\topsep0.25em plus 0.5ex%
              \itemsep0.25em plus 0.5ex\parsep0em}}{\end{list}}
\newcommand\tageq{\addtocounter{equation}{1}\tag{\theequation}}

\newcommand{\mynewtheorem}[2]{
  \newaliascnt{#1}{dummy}
  \newtheorem{#1}[#1]{#2}
  \aliascntresetthe{#1}
  \expandafter\def\csname #1autorefname\endcsname{#2}
}
 
\theoremstyle{plain}
  \mynewtheorem{thm}{Theorem}
  \mynewtheorem{Proposition}{Proposition}
  \mynewtheorem{Corollary}{Corollary}
  \mynewtheorem{assumption}{Assumption}
  \mynewtheorem{lemma}{Lemma}
\theoremstyle{definition}
  \mynewtheorem{definition}{Definition}
  \mynewtheorem{example}{Example}
\theoremstyle{remark}
  \mynewtheorem{Remark}{Remark}

\newcommand{\td}{\tilde{\partial}}
\newcommand{\mb}[1]{\mathbb{#1}}
\newcommand{\mc}[1]{\mathcal{#1}}

\newcommand{\znum}{\mathbb{Z}}

\newcommand{\rnum}{\mathbb{R}}
\newcommand{\nnum}{\mathbb{N}}
\newcommand{\E}{\mathbb{E}}
\newcommand{\pr}{       \mathbb{P}}
\newcommand{\Cov}{\text{Cov}}
\newcommand{\flo}[1]{\lfloor#1\rfloor}
\newcommand{\F}{\mathcal{F}}
\newcommand{\G}{\mathcal{G}}

\newcommand{\bor}{\mathcal{B}}
\newcommand{\si}{$\sigma$}
\newcommand{\PH}{PH}
\newcommand{\ceil}[1]{\lceil#1\rceil}
\newcommand{\VR}{\mathrm{VR}}

\newcommand{\xu}[1]{\tilde{\mb{X}}_{#1}(\frac{#1}{T})}
\newcommand{\xum}[1]{\tilde{\mb{X}}_{m,#1}(\frac{#1}{T})}

\newcommand{\bv}{\Big\vert}
\newcommand{\hti}{\mathfrak{h}}

\allowdisplaybreaks

\newcommand{\sh}{\Gamma}

\newcommand{\longsquiggly}{\xymatrix@C=1.5em{{}\ar@{~>}[r]&{}}}
\newcommand{\medsquiggly}{\xymatrix@C=1.2em{{}\ar@{~>}[r]&{}}}

\providecommand{\AMS}[1]{\textbf{\textit{AMS subject classification: }} #1}

\newcounter{relctr} 
\everydisplay\expandafter{\the\everydisplay\setcounter{relctr}{0}} 

\AtBeginDocument{}

\begin{document}

\begin{frontmatter}

\title{A statistical framework for analyzing shape in a time series of random geometric objects}
\begin{aug}
\author[A]{\fnms{Anne} \snm{van Delft}\ead[label=e1,mark]{anne.vandelft@columbia.edu}}
\and
\author[B]{\fnms{Andrew J.} \snm{Blumberg}\ead[label=e2]{andrew.blumberg@columbia.edu}}
\address[A]{Department of Statistics, Columbia University, 1255 Amsterdam Avenue, New York, NY 10027, USA.\\ \printead{e1}}

\address[B]{Irving Institute for Cancer Dynamics, Columbia University, 1190 Amsterdam Avenue, New York, NY, 10027, USA.\\ \printead{e2}}

\end{aug}

\begin{abstract}
We introduce a new framework to analyze shape descriptors that capture
the geometric features of an ensemble of point clouds. At the core of
our approach is the point of view that the data arises as sampled
recordings from a metric space-valued stochastic process, possibly of
nonstationary nature, thereby integrating geometric data analysis into
the realm of functional time series analysis.  Our framework allows
for natural incorporation of spatial-temporal dynamics, heterogeneous
sampling, and the study of convergence rates. Further, we derive
complete invariants for classes of metric space-valued stochastic
processes in the spirit of Gromov, and relate these invariants to
so-called ball volume processes. Under mild dependence conditions, a
weak invariance principle in $D([0,1]\times [0,\mathscr{R}])$ is
established for sequential empirical versions of the latter, assuming
the probabilistic structure possibly changes over time. Finally, we
use this result to introduce novel test statistics for topological
change, which are distribution-free in the limit under the hypothesis
of stationarity.  We explore these test statistics on time series of
single-cell mRNA expression data, using shape descriptors coming from
topological data analysis.
\end{abstract}

\AMS{Primary 62M99, 62R20, 62R40; secondary  60B05, 60F17, 62M10}

\end{frontmatter}


\section{Introduction}\label{sec1}

Geometric data analysis is concerned with rigorously quantifying and analyzing ``shape''. The data is usually presented as an ensemble of ``point clouds'', i.e., finite metric spaces $(\tilde{\mb{X}}^{n}_t)_{t=1}^T$ that represent a collection of observations where each $\tilde{\mb{X}}^{n}_t$ arises from sampling $n$ points from an underlying geometric object ${\mb{X}}_t$. The most familiar and widely used method in  geometric data analysis is clustering, yet this only captures coarse geometric information.  Thus, over the past 20 years there has been intensive work on more sophisticated ways of capturing shape information (e.g., via manifold learning~\cite{pol22,isomap, lle, meila24}, topological data analysis (TDA)~\cite{c14,c09,eh08, eh10,  flrwbs14, o15}, and optimal transport~\cite{villani09, peyre19}). However,  statistical methodology suitable for exploiting these sophisticated shape descriptors is still in its infancy, and available methods rely on the iid sampling paradigm. 

Yet the ability to model, analyze and predict the evolution over time of the geometric features of data is of paramount interest in many applications. For example, cell differentiation can be studied by analyzing time series of single-cell mRNA expression data (scRNA); a core problem here is to quantify changes in gene expression profiles for cells collected during the process of development.  Specifically, the data looks like a sample $(X_t)_{t=1}^T$ where $X_t \subset \mathbb{R}^N$ for $N$ in the tens of thousands. Cell type is captured in part by the cluster structure of the point clouds of expression vectors.  Changes in the shape of these point clouds reflect differentiation events such as the emergence of new cell types.  More precisely, so-called bifurcation events reflect when an ancestral cell type  changes into multiple lineages, and can be detected by change in shape.  Examples of studies of this kind include~\cite{sstc19}, which profiles several hundred thousand cells from  mouse embryonic fibroblasts and provides evidence that shape provides insight into developmental trajectories.  A toy example of a developmental process is given by modeling the genomic profiles of cells as generated by sampling from a superposition of two spherical Gaussians with centers moving apart over time.  This represents the emergence of two distinct cell types from undifferentiated stem cells; see Figure~\ref{fig:toy-example}.

\begin{figure}[t]
\centering
\includegraphics[scale=0.25]{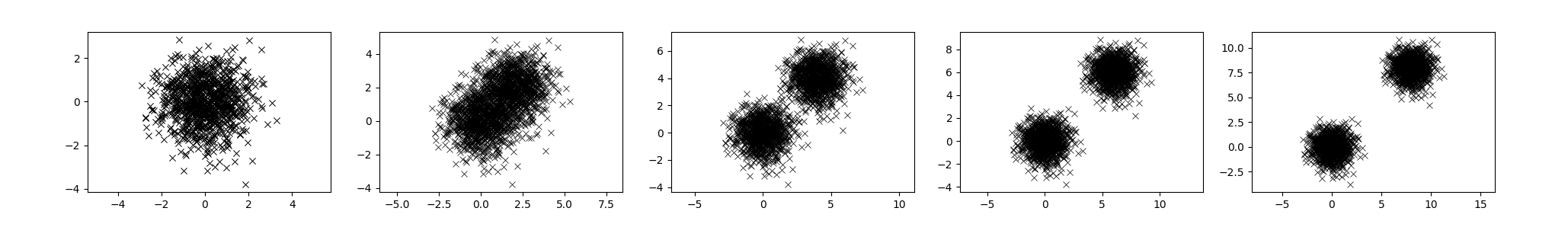}
\caption{Differentiated cell types emerging over time.}
\label{fig:toy-example}
\end{figure}

The need for reliable inference on shape and topological features in applications has led to  substantial interest in integrating classical statistical techniques with topological invariants. Roughly speaking, TDA provides qualitative multiscale shape descriptors for point clouds, notably {\em persistent homology}. This is a higher-dimensional generalization of hierarchical clustering, encoding the feature scales at which ``holes'' of various dimensions appear and disappear. We refer the unfamiliar reader to Section \ref{sec:revTDA}.  Key was the early work~\cite{mmh11, bgmp14}, which established  that the space where these topological  shape descriptors (known as barcodes or persistence diagrams) take their values is Polish.
 
In fact, an important issue from the point of view of inference is that sophisticated shape descriptors such as those in TDA take values in Polish spaces that lack a vector space structure. 
  Unfortunately, Polish spaces arising as targets for shape descriptors are hard to work with directly. In particular, the lack of a vector space structure means that concise summaries of the distribution such as moments are not available, and generalized notions such as Fr{\'e}chet means are not unique and hard to compute. To avoid some of the issues of dealing with Polish-valued data, there has been interest in techniques for embedding the shape descriptors into Banach spaces \cite{bub15a,aetal17}. 
However, what these descriptors capture of the shape of the original geometric object is unclear.
    
    Furthermore, the literature relies heavily on having an iid sample of point clouds, which is often not justified, and therefore can lead to invalid inference. Although some interesting work exists on using shape descriptors in the analysis of univariate 
(deterministic) time series data (e.g., see~\cite{ph15,p19, mmk19,
  xatz21}), explicit statistical foundations are not provided. Indeed, many statistical questions
 related to convergence rates and non-asymptotic error bounds are unexplored, and crucial questions such as what dynamics of the underlying process are preserved by such  shape descriptors have been left unanswered. 
  To the best of the authors' knowledge, there is no systematic theory to perform statistical inference for topological or geometric features in the presence of temporal dynamics which potentially evolve over time, either abruptly or
gradually.  However, data sets arising in diverse application areas
are nonstationary and have complicated dependence structures.

The aim of this article is to introduce a new framework that addresses
these issues in an intrinsic way.  We provide statistical foundations
for applying a wide variety of shape descriptors to capture the
geometric features in temporal-spatial data sets. Broadly speaking,
our framework integrates statistical methodology and geometric data
analysis in the context of \textit{functional time series} by viewing
the data as arising from a time-varying metric space-valued stochastic
process.  
The perspective here is that the fundamental
datum is a function (i.e., the observations are points in a function
space), enabling the development of statistical tools that account for
the underlying structure.  Even though the burgeoning literature on functional time series (e.g., see~\cite{b00,hk,detkokvol2020,HyndShang09, pt12,vDD22}
and the references therein) focuses on
processes with elements in normed vector spaces, mainly Hilbert spaces, 
its conceptual essence encompasses processes with elements in
metric spaces, and thus naturally supports the analysis of geometric
features.  

By viewing an ensemble of point clouds as arising from an underlying time-varying metric space-valued stochastic
process, we put forward a comprehensive  theory. In \autoref{sec2}, we start by introducing a notion of
locally stationary metric space-valued processes, which enables
development of statistical  methodology even if the probabilistic structure of
the process changes over time. Within this setting, we address fundamental questions and provide the groundwork for meaningful inference on invariants that capture shape of the latent process via point clouds. We explain the general nature of the shape descriptors that fit into our framework, show  what
distributional properties of the process are preserved on the space of
shape descriptors, and 
introduce mild assumptions on the sampling regime 
that not only ensure inference drawn based on point clouds is consistent for the geometric features of the latent process, but also that enable the study of the corresponding level of statistical accuracy. 

Our framework allows us to develop suitable statistical inference techniques which capture the distributional properties but are not constructed directly on the underlying Polish space, thereby avoiding the aforementioned difficulties. This part of our theory consists of two foundational results (\autoref{sec3}).

 First, we establish a result that is reminiscent of Gromov's reconstruction theorem 
~\citep[3$\frac{1}{2}.5$]{Gromov}. This theorem establishes that
metric measure spaces (i.e., metric spaces equipped with a measure on
the Borel $\sigma$-algebra) are completely characterized up to
isomorphism by the infinite-dimensional distance matrix distribution
resulting from iid random sampling. Geometric data analysis depends on
this assumption in the sense that the invariants computed are almost
always derived from approximations of the finite-dimensional distance
matrix distributions. To enable the analysis of non-iid processes, we introduce a reconstruction theorem for metric
space-valued stochastic processes (\autoref{ergodiccase}). For this,
we propose equivalence classes of what we call
\textit{metric measure-preserving dynamical systems} (\autoref{def1}), which can be viewed as generalizations of metric measure spaces.  This then provides a foundation for using distance-based shape descriptors.  In the case of our toy example, this implies that the corresponding process of time-varying distance matrices completely characterize the underlying geometry.

\begin{figure}[t]
\centering
\includegraphics[scale=0.27]{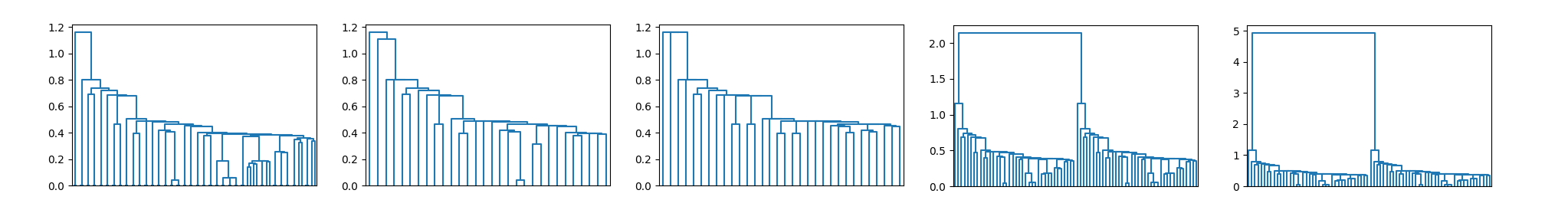}
\caption{Hierarchical clustering dendrograms reveal the change in shape of the developmental process.  Each panel corresponds to the spherical Gaussians from \autoref{fig:toy-example}; as they separate, cluster structure emerges.}
\label{fig:toy-example-dendrograms}
\end{figure}

\begin{figure}[t]
\centering
\includegraphics[scale=0.4]{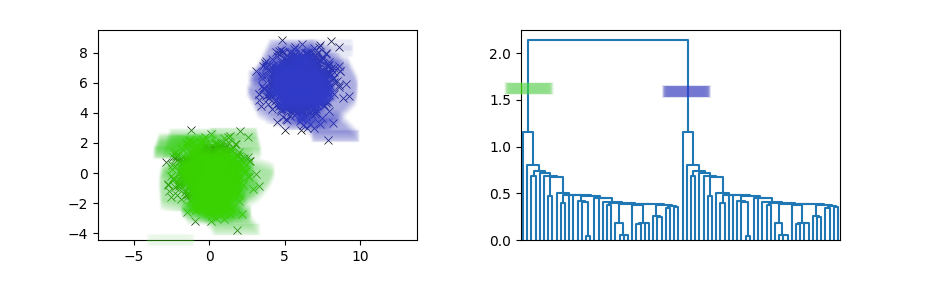}
\caption{Each slice of the dendrogram corresponds to a clustering of the data.}
\label{fig:toy-example-slices}
\end{figure}

Although it is possible to apply this characterization directly for inference, we can often use a more direct approach.  Specifically, our second result  establishes general conditions under which the geometric features of a stochastic process can in fact be fully characterized by the process of ball volumes (\autoref{thm:charmux}).  The idea is that,  although the process generating the point clouds might not satisfy these conditions, the induced process on the space of values for a shape descriptor (such as those mentioned above) will.  

In \autoref{sec4}, we apply our
machinery and establish a weak invariance principle for the empirical
ball volumes of shape descriptors of point clouds on various spaces of
invariants (\autoref{thm:2parproccon}).  This is done for the class of locally stationary metric
space-valued stochastic processes. For this, we prove  convergence
of the sequential empirical versions in the Skorokhod space
$D([0,1]\times [0,\mathscr{R}])$ to a Gaussian process, which reduces
to a simple form under stationarity. The latter result, which is of
independent interest, is used to introduce a family of maximum and
quadratic range-based self-normalized test statistics to detect gradual and abrupt topological changes. This provides the applied researcher with a simple tool for inference on the true underlying shape dynamics over time. Returning to the example of cell differentiation, we use these results to test for changes in shape via topological shape descriptors in Section \ref{sec:realdata}.  However, our theory can also be used to test for the evolution of the number of clusters by working with other shape descriptors, such as  hierarchical clustering dendrograms (\autoref{fig:toy-example-dendrograms}-\ref{fig:toy-example-slices}). 
Because of space constraints, certain proofs as well as the 
simulations to demonstrate the finite sample performance of
our tests, are deferred to the Online Supplement.

\section{Framework}\label{sec2}

Let $(M,\partial_M)$ and $(M', \partial_{M'})$ be compact metric spaces and let $\bor(M)$ and $\bor(M')$ be the Borel \si-algebras of $M$ and $M'$ respectively, i.e., the \si-algebras generated by their respective metric topologies. The space of  continuous functions $f\colon (M,\partial_M) \to (M',\partial_{M'})$ with the topology induced by the uniform metric $\rho$ will be denoted by $(C(M,M'), \rho)$. Observe that $(C(M,M'), \rho)$ is again a compact metric space. We consider stochastic processes $({{\mb{X}}}_{t}\colon t\in \znum)$ defined on some common probability space $(\Omega, \F, \pr)$ with values in $C(M,M')$, i.e., the mappings 
\[{\mb{X}}_{t} : \Omega \to C(M,M'),\tageq \label{eq:proc}\] 
are $\F/ \bor(C)$ measurable, where $\bor(C)$ denotes the $\sigma$-algebra generated by $\rho$. Observe that $\omega \in \Omega$
 yields a continuous function $${\mb{X}}_t(\omega): M \to M^\prime, \quad m\mapsto \mb{X}_t(\omega)(m) \in M^\prime$$ and $\mb{X}_{t}(M) \subseteq M^\prime$. As is common in the literature, we slightly abuse notation and drop $\omega$ from the notation, unless confusion can arise.
We can relate this abstract representation of $\mb{X}_t$ to an  equivalent process in terms of sample paths in $M^\prime$. More specifically, we have the evaluation functional $
e_m \colon C(M,M^\prime) \to M^\prime$, $\chi \mapsto \chi(m)$,
which is continuous with respect to the topology induced by $\rho$. Then the process $
(\tilde{\mb{X}}_t(m): t \in \znum, m \in M)$
defined by $$
\tilde{\mb{X}}_t(m) :=e_m \circ \mb{X}_t$$  
takes values in $M^\prime$. In the literature both processes are often denoted by $\mb{X}$. However, given the context, we will try to avoid confusion and simply notice the equivalence $\tilde{\mb{X}}_{t}=\mb{X}_{t}(M)$.

 In practice, we do not observe the random functions fully on their domain of definition. Instead, we observe a so-called \textit{point cloud} for each $t$.  To make this precise, define the multiple evaluation  functional \[e_{m_1,\ldots,m_n}: C(M,M^\prime) \to \prod_{i=1}^n M^\prime,\quad \chi \mapsto  \big\{\chi(m_1), \chi(m_2), \ldots, \chi(m_n)\big\}, \quad m_1, \ldots, m_n \in M. \]
Then a point cloud of size  $n$ of $\mb{X}_t$ can be represented as
\[{\tilde{\mb{X}}}^n_{t}:= e_{m_1,\ldots,m_n} \circ \mb{X}_{t}~,\tageq \label{eq:pointcl}\]
for some $m_1,\ldots,m_n \in M$. We remain agnostic about the way the point clouds arise;  ~\autoref{as:mincov} below encodes a variety of sampling regimes. Intuitively, we require 
$\lim_{n\to \infty}\tilde{\mb{X}}^{n}_{t} \approx \mb{X}_{t}(M) $ in an appropriate sense (see \eqref{eq:cloudapprox}). A  flexible point of view to take is that the ensemble of  point clouds arises as atoms of an appropriately specified family of point processes $(\xi_t: t \in \znum)$ on $M^\prime$, so that the support of $\xi_t$ is a (locally finite) random subset $\mathcal{X}_t \subset {\mb{X}}_t(M) \subset M^\prime$, which then lets us write  $\tilde{\mb{X}}^k_t \equiv \mathcal{X}_t \,|\,\, \xi_t(M^\prime)=k$ for some $m_1, \ldots, m_k$ (see also \autoref{ex:cox}).
Essential in the  analysis is that the sequence of point clouds, whether arising from (noise-corrupted) sampling on the domain of definition or as  atoms of a point process defined on $M^\prime$, 
is required to have a growth intensity that is dependent on the sample size of the ensemble, i.e., $n=n(T)$.
We think of $M$ as a parameter space and the images in $M'$ as representing the geometric object of interest.  For example, $M'$ could be an ambient Euclidean space or a compact Riemannian manifold parametrized by a function whose ``slices" at time $t$ represent the geometric objects of interest.  In these cases, $M$ could coincide with $M'$ or could be the domain of a function parametrizing the points of the underlying geometric objects. 

Our formulation as a metric space-valued process contrasts with the literature on geometric data analysis.
In existing methods, the point clouds arise from sampling according to a pre-specified distribution $\mathcal{P}$ supported on a \textit{fixed} latent  underlying topological space $\mb{M}$. Apart from notable exceptions \citep[][]{oa17,k11}, the sampling regime is assumed iid and $\mathcal{P}$ is taken as the uniform distribution  on $\mb{M}$.  In our setup, a realization of $\mb{X}_{t}$ is itself a metric subspace of the space of bounded functions from $M$ to $M^\prime$, and arises according to the (unknown) law of a  function-valued stochastic process. This formulation not only naturally incorporates (nonstationary) temporal and spatial dependence, 
but also allows us to view the actual atoms of the ensemble of point clouds as originating from a source of randomness, either on the parameter space or the image space.
In addition, it enables a  comprehensive analysis of non-asymptotic bounds and convergence rates.

To capture the shape of these latent processes via the corresponding ensemble of point clouds, we can work with any shape descriptor that is {\em stable}, in a precise sense.

\subsection{Stable shape descriptors} \label{sec:ph}

We now set up an axiomatic characterization of stable shape invariants that fit into our statistical framework.  The key property of a shape descriptor that we need is stability with respect to a suitable metric on the space of compact metric spaces, which we now make precise.  Quantifying how well a point cloud (i.e., a finite metric space) approximates a latent object (e.g., a compact metric space) or another point cloud, requires a meaningful way to express a notion of distance between two different metric spaces, possibly of different cardinality. A metric that allows for this is the {\em Gromov-Hausdorff} distance.  There are various equivalent definitions of the Gromov-Hausdorff distance. We make use of its formulation in terms of correspondences.  Specifically, given two sets $X$ and $Y$, we call a subset  $\mathcal{R} \subset X \times Y$ a correspondence if it satisfies
 that $\forall x \in X$, $\exists y \in Y$ such that $(x,y) \in \mathcal{R}$ and 
 $\forall y \in Y$, $\exists x \in X$ such that $(x,y) \in \mathcal{R}$.  The \textit{distortion} of the correspondence is given by \[
  \text{dist}(\mathcal{R})= \sup_{(x,y),(x^\prime, y^\prime) \in \mathcal{R}} \Big\vert \partial_X(x,x^\prime)-\partial_Y(y,y^\prime)\Big\vert.
  \]
  The Gromov-Hausdorff distance can be defined as 
  \[
  d_{GH}(X,Y)=\frac{1}{2}\inf\Big\{\text{dist}(\mathcal{R})\,\,|\,\,\mathcal{R} \text{ correspondence between $X$ and $Y$} \Big\}~.\tageq \label{eq:GHdist}
  \]
  We remark that the Gromov-Hausdorff distance between two compact metric spaces is 0 if and only if they are isometric.  In other words, this is really a metric on the set of isometry classes of compact metric spaces (and a pseudo-metric on the set of all compact metric spaces).

In what follows, we will write $\mathcal{M}$ for the set of compact metric spaces.
 
\begin{definition}\label{def:stabsh}
A {\em stable shape descriptor} with values in a Polish space $(\mathscr{B}, d_{\mathscr{B}})$ is a function $\sh \colon \mathcal{M} \to \mathscr{B}$ that is stable in the sense that there exists a constant $L > 0$ such that for any $X,Y \in \mathcal{M}$,
\[
d_{\mathscr{B}}(\sh(X), \sh(Y)) \leq L d_{GH}(X,Y). \tageq \label{eq:stable}
\]
\end{definition}
That is, a stable shape descriptor is a Lipschitz function from the set of  compact metric spaces to a Polish space.   As we shall see, this condition suffices for our framework.  Notable examples of stable shape descriptors include dendograms~\cite{cm10}, many of the invariants of TDA such as persistent homology and zigzag persistence, and metric geometry invariants such as the distance distribution.   In Section \ref{sec:realdata}, we will explore an application in the context of {\em persistent homology}, the main shape descriptor from TDA.

\begin{example}[Dendograms]
Consider the function which assigns to a point cloud $X$ (i.e., a finite metric space) the associated single-linkage hierarchical clustering dendrogram. For fixed $\epsilon > 0$, single-linkage clustering works by assigning two points $x,z \in X$ to be in the same cluster if they are connected by a path $x = x_0, x_1, x_2, \ldots, x_k = z$ where $\partial_X(x_i, x_j) < \epsilon$. A dendrogram encodes clustering information across feature scales which are represented by the $y$-coordinates; for $y = \epsilon$, the horizontal slice of the dendrogram has separate components for each single-linkage cluster at scale $\epsilon$.  To see what the entire dendrogram encodes, observe that in figure~\ref{fig:toy-example-dendrograms} the nearby points yield many tiny clusters (for small $\epsilon$) that merge into a single cluster, whereas the separated points have a wide range of feature scales for which there are two distinct clusters before they finally merge when $\epsilon$ is large enough.  See figure~\ref{fig:toy-example-slices} for an indication of the clusters that correspond to a particular $y$-coordinate in the dendrogram.
 There is a metric on dendrograms (using the Gromov-Hausdorff distance) such that assignment of a dendrogram is a stable shape descriptor with respect to this metric~\cite{cm10}.
\end{example}

\begin{figure}[t]
\centering
\includegraphics[scale=0.4]{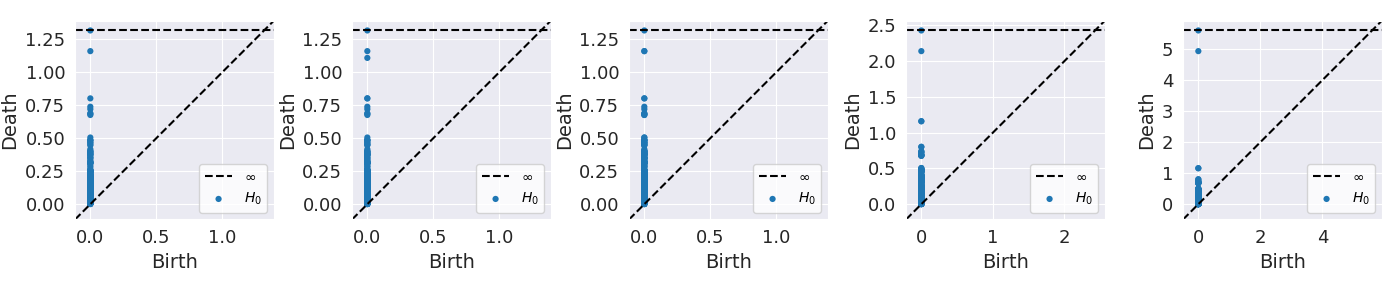}
\caption{Each panel corresponds to $\PH_0$ of the example in \autoref{fig:toy-example}. The $x$-axis represents the scale at which the cluster appears and the $y$-axis the scale at which it disappears.  The two points appearing far away from the  diagonal represent the emergence of two distinct cell types.}
\label{fig:toy-example-PH0}
\end{figure}
\begin{example}[Persistent homology] 
Persistent homology gives a rich family of qualitative stable shape descriptors that capture topological features of a point cloud progressively across different feature scales.  We give a brief review of TDA and specifically persistent homology in   \autoref{sec:5}. In short, the $k$th persistent homology, $\PH_k$, captures how $k$-dimensional holes (i.e., connected components, tunnels, voids, etc.) appear and disappear as the feature scale changes. 
 As illustrated in figure~\ref{fig:toy-example-PH0}, $\PH_0$ captures the number of connected components and how they persist across feature scales.  Each point of the persistence diagram represents a cluster from the scale at which it appears to the scale at which it merges with a neighboring cluster.   
But in higher dimensions, persistent homology captures much more interesting features; we refer to Section \ref{sec:revTDA} and particularly the illustration of constructing $PH_1$ in figures~\ref{fig1} and~\ref{fig2}. 
\end{example}

\subsection{Locally stationary metric space-valued processes}

To provide statistical theory which allows for changes in the probabilistic structure, we consider an ``infill'' asymptotic framework in which more observations become available at a local level as the observation  length $T$ increases.  Nonstationary processes that can be analyzed in such a framework are known as \textit{locally stationary time series}. The theory was introduced in \citep{dahlhaus1997} for the linear case, and refined in \cite{SR06} to nonlinear time series. This paved the way for the development of various inference methods including estimation, prediction, and detecting deviations from stationarity (e.g., see~\cite{kp15,jensub2015,dzhou20,pd2021,ZW21,cper21} and references therein). Theory and corresponding inference techniques for Banach space-valued processes 
were later put forward in \cite{vde-loc,vDD22}, known as \textit{locally stationary functional time series}. To develop analogous theory and inference techniques for our case, we require a notion of \textit{locally stationary metric space-valued processes}, extending the previous to random elements with values in a function space that inherits the metric structure.

To make this precise, we think of having a triangular array of processes $(\mb{X}_{t,T}: t= 1,\ldots, T: T \in \mathbb{N})$ indexed by $T \in \mb{N}$, where each $\mb{X}_{t,T}$ is a random element with values in a complete, separable metric space $(S, \partial_s)$. We remark that the double-indexed process can be extended on $\znum$ by setting $\mb{X}_{t,T}=\mb{X}_{1,T}$ if $t<0$ and $\mb{X}_{t,T}=\mb{X}_{T,T}$ if $t\ge T$. Define the $\mathcal{L}^p$-distance, $1 \le p <\infty$, by 
$$\partial_{S,p}(\mb{X},\mb{Y}) = \Big(\E \partial^{p}_S(\mb{X}, \mb{Y})\Big)^{1/p} = \Big(\int_{\omega \in \Omega} \partial^p_S(\mb{X}(\omega),\mb{Y}(\omega)) d\pr(\omega) \Big)^{1/p}, \quad \mb{X},\mb{Y}\colon \Omega \to S,$$ 
and the $\mathcal{L}^\infty$-distance by $\partial_{S,\infty}(\mb{X},\mb{Y}) = \inf_{\pr(A) = 0} \sup_{\omega \in \Omega\setminus A} \partial_S(\mb{X}(\omega),\mb{Y}(\omega))$. We say that $\mb{X}$ and $\mb{Y}$ are equivalent if $\mb{X}=\mb{Y}$ $\pr$-almost surely. The corresponding $\mathcal{L}^p$-space is  denoted by
\[\mathcal{L}^p_S=\Big\{\text{ equiv. classes of $\F$-measurable $\mb{X}\colon \Omega \to S$ with $\partial_{S,p}(s,\mb{X})<\infty$, $s \in S$ }\Big\}~.\]

\begin{definition}\label{def:Localstationarity}
Let $(\mb{X}_{t,T}: t \in \znum, T \in \nnum)$ be an $(S,\partial_S)$-valued stochastic process. 
\begin{enumerate}[label=(\roman*)]
\item
$(\mb{X}_{t,T}: t \in \znum, T \in \nnum)$ is \textit{locally stationary} if for all  $u=t/T \in [0,1]$, there exists an $(S,\partial_S)$-valued stationary process $(\mb{X}_t(u): t \in \znum, u \in [0,1])$ such that
\[
 \partial_S\big(\mb{X}_{t,T}, \mb{X}_t(\frac{t}{T}) \big)=O_p(T^{-1}) \quad \text{and} \quad \partial_S(\mb{X}_{t}(u), \mb{X}_t(v))=O_p(|u-v|) \tageq \label{eq:locstatp}
\] 
uniformly in $t=1, \ldots, T$ and $u,v \in [0,1]$.\\
\item
If $(\mb{X}_{t,T}: t \in \znum, T \in \nnum) \in \mathcal{L}^p_S(\F)$, then we call it \textit{locally stationary in $\mathcal{L}^p_S(\F)$} if for all  $u=t/T \in [0,1]$, there exists an $\mathcal{L}^p_S(\F)$-valued stationary process $(\mb{X}_t(u): t \in \znum, u \in [0,1])$ such that for some constant $K >0$,
\[
\partial_{S,p}\big(\mb{X}_{t,T}, \mb{X}_t(\frac{t}{T})\big) \le K T^{-1} \quad \text{and} \quad  \partial_{S,p}\big(\mb{X}_{t}(u), \mb{X}_t(v)\big) \le K|u-v|  \tageq \label{eq:locstat}
\] 
uniformly in $t=1, \ldots, T$ and $u,v \in [0,1]$.
\end{enumerate}
\end{definition}

We emphasize that if the process is in fact stationary then $\mb{X}_{t,T}=\mb{X}_t$ for all $t=1,\ldots, T$ and $T \in \mathbb{N}$. In this case, infill asymptotics simply coincides with classical asymptotics. Thus, processes that fit within the infill asymptotic framework encompass the class of stationary processes; this is a useful setting to detect deviations from stationarity. 

\subsection{Inference on the latent process via point clouds} \label{pointcloudpart}
We now turn to inference on the geometric features of the processes of the form $(\mb{X}_{t,T})_{t\in [T]}$ 
by means of the corresponding available ensemble of point clouds $(\tilde{\mb{X}}^n_{t,T})_{t\in [T]}$. In order to do so, we study the latter in the target space for the stable  shape descriptor $\sh$. For geometric inference, we thus have a  realization of the $(\mathscr{B},d_{\mathscr{B}})$-valued process $\big(\sh(\tilde{\mb{X}}^n_{t,T})\big)_{t\in [T]}$ at our disposal. 
In this subsection, we make precise what distributional properties of the latent process are preserved by $\big(\sh(\tilde{\mb{X}}^n_{t,T})\big)_{t\in [T]}$ in the space of stable shape descriptors  $(\mathscr{B},d_{\mathscr{B}})$. We start with the following lemma.  

\begin{lemma}\label{lem:dghds}
For all $X, Y \in S=(C( M,M^\prime),\rho) $ and 
 $\tilde{X}^n = e_{m_1,\ldots,m_n} \circ {X}, \tilde{Y}^n =  e_{m_1,\ldots,m_n} \circ {Y}$, \[\sup_n d_{GH}(\tilde{X}^n,\tilde{Y}^n) \le \partial_S(X,Y).\]
\end{lemma}
\begin{proof}[Proof of \autoref{lem:dghds}]
Recall the Gromov-Hausdorff distance \eqref{eq:GHdist} and 
 consider the correspondence $\mathcal{R}=\big\{(x,y)=\big(X(m), Y(m)\big):m \in M\big\}$. The triangle inequality yields
   \begin{align*}
 \text{dist}(\mathcal{R})&
 =\quad \sup_{(X(m),Y(m)),(X(\tilde{m}), Y(\tilde{m})) \in \mathcal{R}} \Big\vert \partial_{M^\prime}(X(m),X(\tilde{m}))-\partial_{M^\prime}(Y(m),Y(\tilde{m}))\Big\vert
 \\& \le \quad 2\sup_{m \in M} \partial_{M^\prime}\big(X(m),Y(m)\big) = 2\partial_S(X,Y).
  \end{align*}
  Since $\sup_{n}\mathrm{dist}(\mathcal{R}_n) \le \mathrm{dist}(\mathcal{R})$ where $\mathcal{R}_n=\big\{(x,y)=\big(X(m_i), Y(m_i)\big):m_i \in M,i \in [n]\big\}$, the result follows.
\end{proof}

\begin{Proposition}\label{prop:borm}
For a compact metric space $(S,\partial_S)$,     $\sh(\cdot)\colon (S,\partial_S) \to (\mathscr{B},\partial_{\mathscr{B}})$ is a Borel measurable transformation.
\end{Proposition}
\begin{proof}
By \autoref{eq:stable} and \autoref{lem:dghds}, 
$d_{\mathscr{B}}(\sh(X), \sh(Y)) \leq Ld_{GH}(X, Y)\le L d_{S}(X, Y)$ for all $X,Y \in S$. 
Thus, $\sh(\cdot)\colon (S,\partial_S) \to (\mathscr{B},\partial_{\mathscr{B}})$ is Lipschitz-continuous with Lipschitz constant $L$. Hence,  the preimage of every open set is open. 
\end{proof}
Using continuity of the evaluation functionals and  \autoref{prop:borm}, we obtain: 
\begin{Corollary}\label{prop0}
Let $(S,\partial_S)=(C(M,M^\prime),\rho)$. If $(\mb{X}_{t}: t \in \znum)$ is an $(S,\partial_S)$-valued stationary and ergodic process then  $(\sh(\tilde{\mb{X}}^n_t): t\in \znum)$ is a stationary and ergodic $({\mathscr{B}},d_\mathscr{B})$-valued process.
\end{Corollary}

A second implication of \autoref{prop:borm}, which we will make use of is the following.
\begin{Corollary}
\label{prop1} 
Let $\G_t$ be a sub-sigma algebra of $\F$ and let $(\G_t)$ be a filtration. Suppose 
$(\mb{X}_{t})$ is adapted to $(\G_t)$. Then $\sh(\tilde{\mb{X}}^n_{t})$, is adapted to  $\G_{t}$ for all $t$ and $n$.
\end{Corollary}

\autoref{prop1} ensures for example that if the elements of $(\mb{X}_t)$ are measurable functions of a strongly mixing sequence $(\zeta_t: t\in \znum)$, the tail $\sigma$-algebra $\G_{-\infty}= \cap_{t \ge 0} \G_{-t}$, where $\G_t=\sigma(\zeta_s, s\le t)$ is the natural filtration of $(\zeta_t: t\in \znum)$, is again trivial. In case $(\zeta_t\colon t\in \znum)$ is a sequence of independent elements this follows immediately from Kolmogorov's 0-1 law, whereas the strongly mixing case requires some more effort (we refer to~\citep{brad05} for details). We will make use of \autoref{prop0} and \autoref{prop1} in \autoref{sec4}. 
Finally, the following result yields that local stationarity is preserved. 

\begin{Proposition}\label{prop2}
Suppose $(\mb{X}_{t,T})$ is locally stationary in the sense of  \autoref{def:Localstationarity} with  $(S,\partial_S)=(C(M,M^\prime),\rho)$. Then $\big(\sh(\tilde{\mb{X}}^n_{t,T})\big)_{t=1}^T$ is a $(\mathscr{B},d_\mathscr{B})$-valued process that satisfies \autoref{def:Localstationarity}. If $(\mb{X}_{t,T})$ is locally stationary in $\mathcal{L}^p_S(\F)$, then it is locally stationary in $\mathcal{L}^p_\mathscr{B}(\F)$.
\end{Proposition}

\begin{proof}
By assumption, there exists an auxiliary process $(\mb{X}_{t}(u)\colon t\in \znum, u\in[0,1])$ that satisfies \eqref{eq:locstat} in relation to $(\mb{X}_{t,T}: t\in \znum,T \in \nnum)$.  By \autoref{prop0}, $(\sh(\tilde{\mb{X}}^n_{t}(u)): t\in \znum, u\in [0,1])$ is then strictly stationary for all $u \in [0,1]$. The stability property of $\sh$ and \autoref{lem:dghds} yield 
\begin{align*}
d_{\mathscr{B}}\Big(\sh(\tilde{\mb{X}}^n_{t,T}), \sh\big(\tilde{\mb{X}}^n_{t}\big(t/T\big)\big)\Big)
\le L\partial_S\big(\mb{X}_{t,T},\mb{X}_{t}\big(t/T\big)\big).
\end{align*}
Similarly, $
d_{\mathscr{B}}\big(\sh(\tilde{\mb{X}}^n_{t}(u)), \sh(\tilde{\mb{X}}^n_{t}(v))\big) \le L\partial_S\big(\mb{X}_{t}(u),\mb{X}_{t}(v)\big)$.
The first part now  follows from \eqref{eq:locstatp}. The second part follows analogously from \eqref{eq:locstat} and from the fact that $
d_{\mathscr{B}}\big(\sh(\tilde{\mb{X}}^n_{t,T}), \sh(\tilde{s}^n)\big) \le   \partial_S(\mb{X}_{t,T},s)$ for all $s \in S$.
\end{proof}

\subsection{Sampling regime}
Next, we explore a mild assumption on the sampling regime to ensure that inference methods based on point cloud functionals consistently capture the topological features of the latent process. For this, we formalize a setting that enables the analysis of convergence rates by relating the sampling regime of the point clouds to the sample size of the ensemble, i.e., $n:=n(T)$ such that  $n(T) \to \infty$ as $T \to \infty$. Furthermore, we exploit the basic structure of compact metric spaces to keep the assumptions minimal and widely applicable. More specifically, since the latent random object $\mb{X}_{t,T}(M)$ is the continuous image of a compact metric space, there exists a collection of points  $\mb{C}^n_{t,T}=\{C_{t,i,T}\}_{i \in [n(T)]} \subset M^\prime$ and a radius $r^n_{t,T}:=r(\mb{X}_{t,T}(M),n(T))$ such that 
$\bigcup_{i\in [n(T)]} B({C}_{t,i,T}, r^n_{t,T})$
provides a minimal $r^n_{t,T}$-radius cover of $\mb{X}_{t,T}(M)$. That is, $n(T)$ is the smallest number of balls with centers $\mb{C}^n_{t,T}$ and with radius $r^n_{t,T}$ such that the union covers $\mb{X}_{t,T}(M)$.  Observe that $\mb{C}^n_{t,T}$ is a $r^n_{t,T}$-net for $\mb{X}_{t,T}(M)$ and thus $d_{GH}(\mb{C}^n_{t,T}, {\mb{X}}_{t,T}(M)) \leq r^n_{t,T}$. 
We impose conditions which ensures that the point cloud is `close enough' to the cover at a controlled rate.

\begin{assumption}\label{as:mincov}
Let $\mb{C}^n_{t,T}$ be the set of centers of an $r^n_{t,T}$-minimal radius cover of\, $\mb{X}_{t,T}(M)$. Then we assume that there exists a function $\alpha_{n,T}$ that satisfies \[
 d_{GH}(\mb{C}^n_{t,T}, \tilde{\mb{X}}^n_{t,T})=O_p(\alpha_{n,T}), \tageq \label{eq:rate}\]
 for all $t \in [T], T \in \mb{N}$, 
where $\alpha_{n,T} \to 0$ as $n(T) \to \infty$. 
\end{assumption}
Observe that if the process is stationary then the subscript $T$ can be dropped from the notation. If $\tilde{\mb{X}}^n_{t,T}$ is not an $r^n_{t,T}$-minimal radius cover then it forms a cover with some radius $\alpha^n_{t,T}>r^n_{t,T}$. Hence, take $\alpha_{n,T} = \sup_t \alpha^n_{t,T}$. It follows from the triangle inequality that \begin{align*}
d_{GH}(\tilde{\mb{X}}^n_{t,T}, \mb{X}_{t,T}(M) \le d_{GH}( \mb{X}_{t,T}(M),\mb{C}^n_{t,T})+ d_{GH}(\mb{C}^n_{t,T}, \tilde{\mb{X}}^n_{t,T}) =O_p(\alpha_{n,T}),
  \tageq \label{eq:cloudapprox}
   \end{align*}
   since 
   $\pr\big(d_{GH}(\mb{C}^n_{t,T}, \mb{X}_{t,T}(M))> \alpha_{n,T} \big) \le\pr\big(d_{GH}(\mb{C}^n_{t,T}, \mb{X}_{t,T}(M))> r^n_{t,T} \big) = 0$
  as $\mb{C}^{n}_t$ is a $r^n_{t,T}$-minimal radius cover by assumption. 
In addition, the stability property therefore yields
\begin{align*}
&|d_\mathscr{B}(\sh(\tilde{\mb{X}}^n_{s,T}), \sh(\tilde{\mb{X}}^n_{t,T}))-d_\mathscr{B}(\sh(\mb{X}_{s,T}(M)), \sh(\mb{X}_{t,T}(M)))|  
=O_p(\alpha_{n,T}).
\end{align*}
\autoref{as:mincov} tells us that the point cloud will get closer to a minimal radius cover as $T \to \infty$ in probability. 
Clearly, the rate at which the point cloud converges to forming a minimal radius cover will be application-dependent. However, the above assumption is mild and encodes a variety of sampling regimes; all of the available conditions used in the geometric data analysis literature (see e.g., \cite{nsw08, bjpr22,cglm15,oa17})
imply \autoref{as:mincov}. 

\begin{example}\label{ex:cox}
To illustrate~\autoref{as:mincov} in a setting that allows for possible heterogeneous interaction between spatial locations,  assume that the point clouds arise as the atoms of a family of Cox processes. Cox processes can be seen to generate a wide class of interesting spatial point patterns and encompass the class of inhomogeneous and mixed Poisson processes.  
To make this precise, 
 let $\eta_T$ denote a random measure on $M^\prime$. For $\omega \in \Omega$, define the function 
\[
\Lambda_{t,T}(\omega, A) =\eta_T(\omega, A \cap \mb{X}_{t,T}(M)(\omega))
  \quad A \in \mc{B}(M^\prime),
\]
and let the point process $\xi_{t,T}$ be a Cox process directed by $\Lambda_{t,T}$, i.e., $\xi_{t,T}$ is conditionally Poisson given $\Lambda_{t,T}$.
Then, to verify \autoref{as:mincov} note that
 \begin{align*}
\pr\big(d_{GH}(\mb{C}^n_{t,T}, \tilde{\mb{X}}^n_{t,T}) > \alpha_{n,T}\big)
& \le \E\Big[\pr\big( B({C}_{t,j,T},\alpha_{n,T} )\cap \mc{X}_{t,T}=\emptyset \text{ for some $j \in [n]$} \bv \xi_{t,T}(M^\prime)=n, \Lambda_{t,T} \big)\Big]
        \\ &\le \sum_j \E\Big[\big( 1- {\Lambda_{t,T}(B({C}_{t,j,T},\alpha_{n,T} ))}/{\Lambda_{t,T}(M^\prime)}\big)^n\Big]
\\& \le  \sum_j \E[e^{-n\cdot \breve{\eta}_{t,T}(B({C}_{t,j,T},\alpha_{n,T} ))}]~,
  \end{align*}
where $\breve{\eta}_{t,T}(\omega, B({C}_{t,j,T},\alpha_{n,T} )) = \frac{\Lambda_{t,T}(\omega, B({C}_{t,j,T},\alpha_{n,T} ))}{\Lambda_{t,T}(\omega, M^\prime)} $. 
If, for almost all $\omega \in \Omega$,  $\min_j \breve{\eta}_{t,T}(\omega, B({C}_{t,j,T},\alpha_{n,T} )) \ge$ \\$\frac{1}{n}\big(\log(n)+c\log(1/\delta)\big)$ for some $\delta, c>0$, then 
\[ \pr\big(d_{GH}(\mb{C}^n_{t,T}, \tilde{\mb{X}}^{n}_{t,T}) > \alpha_{n,T}\big) \le \delta^c, \tageq \label{eq:gromovhd}\]
 which gives \eqref{eq:rate}.
Without changing the requirement on the lower bound, the configuration $c=1$, $\delta_n=1/n$ yields $d_{GH} (\mb{C}^n_{t,T}, \tilde{\mb{X}}^n_{t,T})=o_p(\alpha_{n,T})$. 
Assumptions imposed in the geometric data analysis literature on the probability measure of interest (in the above example this corresponds to $\breve{\eta}_{t,T}$) essentially provide a lower bound on the volume of a ball that has nonempty intersection with the object as a polynomial or affine function of the radius. Notably, this will be the case if the measure is assumed uniform, has a positive density, or more generally satisfies the $(a,b)$-standard assumption \citep{nsw08,flrwbs14,cglm15}. 

As a simple case in point, suppose  $\Lambda_{t,T}(\omega,A) = \int_{A \cap \mb{X}_{t,T}(M)(\omega)} \lambda_T(x) dx$, which corresponds to an inhomogeneous Poisson process. Then the requirement on the lower bound for balls with nonempty intersection reduces to 
   $ B(x,\alpha_{n,T}) \ge \frac{1}{\iota n}\big(\log(n)+c\log(1/\delta)\big)$ where $\iota := \inf_{x,T} \lambda_T(x)/ \int_{\mb{X}_{t,T}(\omega)(M)}\lambda_T(x)dx >0$. For simplicity, if $\mb{X}_{t,T}(\omega)(M)$ is the unit disk for all $t,T$, then we find \autoref{as:mincov} is satisfied with $\alpha_{n} =C\sqrt{\log(n)/n}$,  $C^2 \simeq 2/(\iota \pi)$.
 \end{example}

\begin{Remark} \label{rem:coverstuff} 
Control on the covering number can make the lower bound on $\breve{\eta}_{t,T}(\cdot,A)$, where $A \cap X_{t,T}(M) \neq\emptyset$ almost surely, more refined. To see how, note that  there exists a uniform $\alpha_{n,T}$-minimal radius cover with covering number $k_{n,T}$. Let  $\mb{U}^k_T$ denote the corresponding set of centers. Then
\begin{align*}
     \pr\big(d_{GH}(\mb{C}^n_{t,T}, \tilde{\mb{X}}^n_{t,T}) > 4\alpha_{n,T}\big) \le    \pr\big(d_{GH}(\mb{U}^k_T, \tilde{\mb{X}}^n_{t,T}) > \alpha_{n,T}\big)+
     \pr\big(d_{GH}(\mb{U}^k_T, {\mb{C}}^n_{t,T}) > 3\alpha_{n,T}\big),  \tageq\label{eq:boundfink}
\end{align*}
where the triangle inequality shows the second term drops out since $\mb{C}^n_{t,T}$ and $\mb{U}^k_T$ are, respectively, a $r^n_{t,T}$-minimal radius cover of $X_{t,T}(M)$ and a uniform $\alpha_{n,T}$-minimal radius cover of $(\mb{X}_{t,T}(M))_{t \in [T]}$. 
An argument as above shows that each point cloud $\tilde{\mb{X}}_{t,T}^n$ forms a $\alpha_{n,T}$-net with probability at least $1-\delta^c$ if, almost surely,  
$\min_j\breve{\eta}_{t,T}(\cdot,B(C_{t,j,T},\alpha_{n,T}))$ $\ge$ $\frac{1}{n}\big(\log(k_{n,T})+c\log(\delta^{-1})\big)$. Hence, a lower bound on the latter be made precise by finding a lower bound on the covering number $k_{n,T}$.  Geometric assumptions made in the literature give control on the covering number, expressed in terms of invariants of the underlying metric space. For example, in the case of a doubling metric space this would be the doubling constant, and in the case of a compact Riemannian manifold this includes the injectivity radius and the sectional curvature. We refer for details to \citep{bjpr22} and \citep{nsw08}, respectively. 
 \end{Remark}

\begin{Remark}[{noise corrupted sampling}]
\rm{\autoref{as:mincov} facilitates handling noisy sampling. Formalizing this can be done by considering noise-corrupted versions of \eqref{eq:bershift}, i.e., $
\ddot{\mb{X}}_{t,T}=g\big(t,T,\mathfrak{f}_t, \varepsilon_t \big)$ which satisfy $d_{GH}(\ddot{\tilde{\mb{X}}}^n_{t,T}, {\tilde{\mb{X}}}^n_{t,T}) =O_p(\alpha_{n,T})$, and where $(\varepsilon_t)$ is an iid sequence  independent of $(\mathfrak{f}_t)$. We will study convergence rates and bounds under such assumptions in more detail in future work.}
\end{Remark}
\section{Characterizing $(S,\partial_S)$-valued processes}\label{sec3} 

\subsection{A reconstruction theorem for metric space-valued stochastic processes}

In this section, we provide an appropriate complete invariant in the spirit of Gromov \cite{Gromov}. Gromov's ``mm-reconstruction theorem" yields that a metric measure space (mms) $(S,\partial_S,\nu_S)$ is up to a measure-preserving isomorphism characterized by the pushforward of the product measure $\nu^{\otimes \nnum}$ along $\phi_+\colon S^\nnum\to \mb{M}^{\text{met}}$, where 
\[
\phi_+(s_1,s_2,\ldots)= \big(\partial_S(s_i,s_j)\big)_{(i,j) \in \mb{N}\times \mb{N}} \tageq\label{mmet}
\]
and where $\mb{M}^{\text{met}}$ denotes the space of positive, symmetric matrices.  
In other words, a representative mms $(S,\partial_S,\nu_S)$ of the equivalence class is uniquely determined by the infinite-dimensional random matrix of distances $
\big\{\partial_S(\zeta_i,\zeta_j)\big\}_{(i,j) \in \mb{N}\times \mb{N}}$,
where $(\zeta_i)$ is an iid sample with common distribution $\nu$.  This equivalence enables inference on mms via the corresponding random distance matrix distributions on the convex cone of positive, symmetric matrices, which is considerably more convenient. However, this result and its proof are inadequate to say something similar regarding a metric space-valued stochastic process; it heavily relies upon pushing forward the product measure $\nu^{\otimes \mb{N}}$ and the notion of a metric measure space fails to meaningfully describe the inherent temporal dynamics.

The question arises whether an $S$-valued process $(X_t: t\in \znum)$ with law $\mu_X$ can be characterized via the pushforward of $\mu_X$ along $\phi$, where $\phi:S^\znum \to \mb{M}^{\text{met}}$ is the extension of \eqref{mmet} to include negative indices. That is, if the infinite-dimensional random matrix of distances
\[
\phi(X_1,X_2,\ldots)= \big(\partial_S(X_i,X_j)\big)_{(i,j) \in \mb{Z}\times \mb{Z}} \tageq\label{mmet2}
\]
characterizes $\mu_X$ up to an appropriate notion of isometry. In the following, we answer this question affirmatively and prove that $\phi_\star \mu_X$ uniquely characterizes an equivalence class of processes that can be viewed as  \textit{ergodic metric measure-preserving dynamical systems}. In particular, this result implies that the infinite-dimensional distance matrix obtained by a single realization of a stationary and ergodic metric space-valued processes characterizes the law up to a measure-preserving isometry. 

We start with introducing a definition of a \textit{metric measure-preserving  dynamical system}. 

 \begin{definition}[mmpds]\label{mmds}
 Let $(S,\partial_S)$ be a complete, separable metric space and let $(S,\mathcal{B}(S), \mu_S)$ be a corresponding probability space, where the Borel $\sigma$-algebra is generated by the metric $\partial_S$. Further, let $\theta_S: S\to S$ be a Borel measurable function. We refer to $(S,\partial_S,\mu_S, \theta_S)$ as a \textit{metric measure dynamical system}. If $\theta_S: S \to S$ is a measure-preserving transformation, we call it a \textit{metric measure-preserving dynamical system (mmpds)}.
 \end{definition}
 
\begin{example}
[Stationary and ergodic processes]
Consider the Cartesian product $S^\znum =\prod_{t\in \znum}S$. Then an $(S,\partial_S)$-valued stationary and ergodic stochastic process $ X\colon(\Omega,\F,\pr)\to (S^\znum,\mathcal{B}(S^\znum), \mu_X)$ may be viewed as a metric measure-preserving dynamical system of which the measure-preserving transformation is ergodic under the measure $\mu_X$. To see this, note that separability of $S$ implies that the  cylinder $\sigma$-algebra $\otimes_{t \in \znum}\mathcal{B}(S)$ and the Borel $\sigma$-algebra $\mathcal{B}(S^\znum)$ on the Cartesian product $S^\znum$ (endowed with its product topology) coincide \citep[see e.g.,][]{kallenberg}. Hence, let  $\rho_S$ denote a metric that metrizes the product topology on $S^\znum$ (see e.g., \eqref{eq:supnorm}). Furthermore, consider the left shift map $\theta_S: S^\znum \to S^\znum$, i.e.,
\[
\theta_S(\ldots, x_{-1}, x_0, x_1, \ldots) = (\ldots, x_{0}, x_1, x_2, \ldots)
\]
 and recall that we call $X$ stationary if  $\theta_S X\overset{d}{=} X$, that is,  $\theta_S$ is a measure-preserving transformation. Furthermore, let  $\mathscr{I}$ denote the $\sigma$-algebra generated by the invariant sets, i.e., those sets which satisfy $\theta^{-1}_S I = I$. A measure-preserving transformation is ergodic under $\mu_X$  if the $\sigma$-algebra $ \mathscr{I}$ is  $\mu_X$-trivial, that is,
 \[
 \mu_X(I) = 0 \text{ or 1} \quad \text{for any $I \in \mathscr{I}$.}
 \]
 Thus, $(S^\znum,\rho_{S}, \mu_X, \theta_S)$ is  a mmpds with $\theta_S$ ergodic under $\mu_X$.
 \end{example}
  With \autoref{mmds} in place, we  define a measure-preserving isometry between two metric measure-preserving dynamical systems as follows.
 
\begin{definition} \label{def1}
Two mmpds $(S,\partial_S, \mu_S, \theta_S)$ and $(R,\partial_R, \mu_R, \theta_R)$ are {\em measure-preserving isometric} if
there are sets $S^\prime \in\mathcal{B}(S)$, $R^\prime \in \mathcal{B}(R)$ with $\mu_S(S^\prime)=1$,  $\theta_S(S^\prime) \subseteq S^\prime$ and $\mu_R(R^\prime)=1$, $\theta_R(R^\prime) \subseteq R^\prime$, respectively, such that there exists an invertible measure-preserving map $\mathcal{T}: S^\prime \to R^\prime$ satisfying
\[
\mathrm{(i)}\,\,\, \mathcal{T}(\theta_S(s)) = \theta_R(\mathcal{T}(s)), \, s \in S^\prime; \qquad
\mathrm{(ii)}\,\,\, \partial_R(\mathcal{T}(s),\mathcal{T}(s^\prime))=\partial_S(s,s^\prime), \quad \forall s,s^\prime \in S^\prime.
\]
 \end{definition}

 \begin{definition} Two metric measure-preserving dynamical systems belong to the same equivalence class $\mathfrak{X}$  if they are measure-preserving isometric. 
 \end{definition}
 
 The following statement may be viewed as  a generalization of Gromov's reconstruction theorem to ergodic metric measure-preserving dynamical systems. 
 
\begin{thm}[Reconstruction theorem for ergodic mmpds]\label{ergodiccase}
Let $(S,\partial_S)$ be a complete, separable metric space and let $X\colon (\Omega,\F,\pr)\to (S^\znum,\mathcal{B}(S^\znum), \mu_X)$  be an $S$-valued stationary ergodic stochastic process. Then the distributional properties of  $X$ are, up to measure-preserving isometry in the sense of \autoref{def1}, completely captured by the pushforward measure $\iota_X=\phi_{\star}\mu_X$ of $X$  along a given realization  $(X_t(\omega)\colon t \in \znum)$. 
\end{thm}

\begin{Remark}[Reconstruction theorem applied to locally stationary processes]
    Consider again $(\mb{X}_{t,T}:t \in \znum, T \in \mb{N})$ and its corresponding auxiliary process $({\mb{X}}_t(u):t \in \znum, u \in [0,1])$. We can represent the latter as a mapping $\dot{\mb{X}}$ defined on a probability space $(\dot{\Omega},\dot{\F}, \dot{\pr})$ with $\dot{\Omega} = \{\omega: [0,1] \to S^\znum\}$, the set of all functions from $[0,1]$ taking values in the set of functions $S^\znum$, so that the canonical coordinate mappings are given by 
\[ \dot{\mb{X}}_u(\omega) =\pi_u(\omega)=\omega(u) \quad \omega \in \dot{\Omega}.
\]
 The corresponding  marginal law, denoted by $\mu_{\dot{\mb{X}}(u)}$, i.e., $\dot{\mb{X}}_u(\omega)\sim \mu_{\dot{\mb{X}}(u)}$, describes the distributional properties of the auxiliary process at rescaled time $u$. For each $u$, the marginal processes  $\dot{\mb{X}}_u$ are stationary and thus --provided the shift map is ergodic under $\mu_{\dot{\mb{X}}(u)}$-- \autoref{ergodiccase} applies to the pushforward measure $\iota_{\dot{\mb{X}}(u)}=\phi_{\star}\mu_{\dot{\mb{X}}_u}$ of $\dot{\mb{X}}_u$  along a given realization  $(\mb{X}_t(u)(\omega)\colon t \in \znum)$. Under appropriate assumptions, the pushforward measures  $\iota_{\dot{\mb{X}}(u)}$, $u\in [0,1]$, describe the localized dynamics of $(\mb{X}_{t,T}:t \in \znum, T \in \mb{N})$.
\end{Remark}

\begin{proof}[proof of \autoref{ergodiccase}] Consider another complete, separable  metric space $(R, \partial_R)$, and a stationary and ergodic process  $
Y\colon (\Omega^\prime,\F^\prime,\pr^\prime)\to (R^\znum,\mathcal{B}(R^\znum), \mu_Y)$ 
 with   $\theta_R:R^\znum\to R^\znum$ denoting the corresponding left shift map.  Let $\iota_Y={\phi_R}_{\star} \mu_Y$, where $\phi_R: R^\znum\to \mb{M}^{\text{met}}$. 
One direction is obvious: if 
${\mathscr{X}}:=(S^\znum,\rho_S, \mu_X, \theta_S)$ and ${\mathscr{Y}}:=(R^\znum,\rho_R, \mu_Y, \theta_R)$ belong to the same equivalence class ${\mathfrak{X}}$, then $\iota_X=\iota_Y$ and the random distance distributions coincide.

For the other direction, we show that if $\iota_X=\iota_Y$ then ${\mathscr{X}}, {\mathscr{Y}}\in {\mathfrak{X}}$. We proceed by  establishing the existence of an isometry between the supports of $\mu_X$ and $\mu_Y$ using the ergodic theorem, and show subsequently that this map is measure-preserving. We start by introducing an appropriate metric $\rho_V: V^\znum\times V^\znum \to \rnum$ that metrizes the product topology on $V^\znum$, where $V \in \{S, R\}$. 
 Namely, consider 
\[
\rho_V(v,v^\prime)=\sup_i \frac{\td_V(v_i,v^\prime_i)}{|i|+1}, \tageq \label{eq:supnorm} 
\]
where
\[
\td_{V}(v_i,v^\prime_i)=\min\big(\partial_V(v_i,v^\prime_i),1\big) \tageq \label{eq:dismin}.\] 
Then $(S^\znum, \rho_S)$ and $(R^\znum, \rho_R)$ are both complete, separable  metric spaces.
To establish an isomorphism $\mathcal{T}: \text{supp}(\mu_X) \to \text{supp}(\mu_Y)$, observe that by assumption 
\[
\mu_X(\phi^{-1}_X(C)) =\iota_X(C)=  \iota_{Y}(C) = \mu_X(\phi_Y^{-1}(C))\quad \forall C \in \mathcal{B}(\rnum).
\]
Further, $\mu_X$ and $\mu_Y$ are ergodic measures for the left shift map $\theta_S$ and $\theta_R$, respectively. Thus, 
 for all cylinder sets $A \in S^\znum$ and $B  \in R^\znum$, for $\mu_X$-almost all $\omega \in \Omega$ and $\mu_Y$-almost all $\omega^\prime \in \Omega^\prime$,
\begin{align*}
\mu_X(A)=\lim_T \frac{1}{T}\sum_{j=1}^T  \mathrm{1}\Big( \theta_S^j \circ X(\omega) \in A\Big) \quad {and} \quad \mu_Y(B) =\lim_T \frac{1}{T}\sum_{j=1}^T  \mathrm{1}\Big(\theta^j_R \circ  Y(\omega^\prime)  \in B\Big). \tageq \label{ergthm}
\end{align*}
Specifically, this holds for the sets $A=\phi^{-1}_X(C) \in \bor(S^\znum)$ and $B =\phi_Y^{-1}(C) \in \bor(R^\znum)$. Hence, we can fix an $\omega \in \Omega$, $\omega^\prime \in \Omega^\prime$ such that  $X(\omega) \in \mathrm{supp}(\mu_{X})$ and $Y(\omega^\prime) \in \mathrm{supp}(\mu_{Y})$  for which \eqref{ergthm} holds and for which $
\phi(\theta^j_S \circ X(\omega) ) = \phi(\theta^j_R \circ Y(\omega^\prime))$.
 Then the distances in \eqref{eq:dismin}  also satisfy
 \[
\td_{S}\big(\pi_0\big(\theta_S^i \circ X(\omega) \big), \pi_0\big(\theta_S^j \circ X(\omega)\big)\big)= \td_R\big( \pi_0(\theta_R^i \circ Y(\omega^\prime)), \pi_0(\theta_R^j \circ Y(\omega^\prime)\big),  
\]
where $\pi_0$ denotes the projection onto the 0th coordinate.
Consequently, 
\begin{align*}
\rho_S(\theta_S^i \circ X(\omega),\theta_S^j \circ X(\omega)) = \rho_R\big(\theta_R^i \circ Y(\omega^\prime)),\theta_R^j \circ Y(\omega^\prime)\big)~.\tageq \label{eq:prdist} 
\end{align*}
The latter implies that the sequences $(\theta_S^j \circ X(\omega): j\in \znum)$ and $(\theta_R^j \circ Y(\omega^\prime): j \in \znum)$ (viewed as metric spaces) are isometric since their distance matrices coincide.
Let $\mathcal{T}\colon \mathrm{supp}(\mu_X) \to \mathrm{supp}(\mu_Y)$ be the continuous map such that $\mathcal{T}(\theta_S^j \circ  X(\omega))=\theta^j_R \circ  Y(\omega^\prime)$. Then
\begin{align*}
  \rho_S(\theta_S^i \circ  X(\omega),\theta_S^j \circ  X(\omega)) =\rho_R(\mathcal{T}(\theta_S^i \circ  X(\omega),\mathcal{T}(\theta_S^j \circ   X(\omega))= \rho_R(\theta^i_R \circ  Y(\omega^\prime),\theta^j_R \circ Y(\omega^\prime))~. 
\end{align*}
To establish that this extends uniquely to an isometry $\mathcal{T}: \text{supp}(\mu_X)\to \text{supp}(\mu_Y) $ it suffices to prove the following result:
\begin{lemma}\label{prop:densup}
The sequence $(\theta_S^j \circ X(\omega)\colon j \in \znum)$ is $\mu_X$-almost everywhere dense on $\text{supp}(\mu_X)$.
\end{lemma}
Continuity of $\mathcal{T}$ and \autoref{prop:densup} yield that there is a unique extension to an isometry $\mathcal{T}\colon \text{supp}(\mu_X)\to \text{supp}(\mu_Y)$. It is measure-preserving since the isometry and the properties of the left shift maps give for $B \in \text{supp}(\mu_Y)$ such that $\mathcal{T}(A) = B$ with $A \in \text{supp}(\mu_X)$, 
\[
\mu_X(A) = \mu_X\big(\mathcal{T}^{-1}(B)\big)
=\lim_T \frac{1}{T}\sum_{j=1}^T  \mathrm{1}\big(\theta^j \circ X(\omega)  \in \mathcal{T}^{-1}(B)\big)=\lim_T \frac{1}{T}\sum_{j=1}^T  \mathrm{1}(\theta_R^j \circ Y(\omega^\prime)  \in B)=\mu_Y (B).
\] 
\end{proof}
\begin{proof}[Proof of \autoref{prop:densup}]
Since $S^\znum$ is Polish every closed subset, and in particular the set $\text{supp}(\mu_X)$, is Polish. Since a separable metrizable space is second countable the latter has a countable topological base $\{U_i\}_{i \ge 0}$ which satisfies $\mu_X(U_i)>0$. To show therefore that $\text{supp}(\mu_X) \subseteq \overline{A}$, for some subspace $A$ of $S^\znum$, we recall that $x \in \overline{A}$ if and only if every basis element $U_i$ containing $x$ intersects with $A$. 
Hence, it suffices to prove that $U_i \cap A \neq \emptyset$ for all $U_i$. Or, in other words, the  set should have non-empty intersection with all base elements. We argue by contradiction. 
Suppose that there exists a $ U_i$ with $U_i \cap (\theta^j \circ X(\omega): j \in \znum) =\emptyset$. Then $X(\omega) \not\in \cup_{k=0}^\infty \theta^{-k}(U_i)$ i.e., $X(\omega) \in \cup_{i=0}^\infty V_i$, where $$V_i:=\text{supp}(\mu_X) \setminus  \bigcup_{j=0}^\infty \theta^{-j}(U_i).$$ 
Thus, it suffices to show that the set on which the orbit is not dense, $\cup_{i=0}^\infty V_i$, has measure 0. To do so, observe that $B_i=\cup_{j=0}^\infty \theta^{-j}(U_i)$ is Borel measurable and satisfies $\theta^{-1}(B_i) \subseteq B_i$ and that  $\mu_X(\theta^{-1}B_i)=\mu_X(B_i)$, since $\theta$ is measure-preserving. Thus, $B_i$ is an almost invariant Borel-measurable set, i.e., $\mu_X(\theta^{-1}B_i \triangle B_i) = 0$. Since $\theta$ is also ergodic the latter implies $\mu_X(B_i) \in \{0,1\}$. However $\theta^{-1}(U_i) \subseteq B_i$ and thus we must have $\mu(B_i)=1$ because $\mu(U_i)>0$. But then $\mu_X(V_i) =0$ and thus $\mu_X(\cup_{i=0}^\infty V_i)\le \sum_i \mu_X(V_i) =0$. Consequently, the set on which the orbit is not dense does not belong to $\text{supp}(\mu_X)$, and thus it must be $\mu_X$-almost everywhere dense on $\text{supp}(\mu_X)$. 
\end{proof}

\subsection{Characterizing measures via ball volumes}

It follows naturally from \autoref{ergodiccase} that the infinite-dimensional distance matrix distribution
\[
(\partial_S(X_t,X_s) : t,s \in \znum)  
\]
is a complete invariant of a stationary ergodic Polish-valued stochastic process $X=(X_t\colon t\in \znum)$. In other words, it is a complete invariant of a metric measure-preserving dynamical system $(S^\znum, \rho_S, \mu_X, \theta)$ where the measure-preserving map $\theta$ is ergodic under $\mu_X$.  This  property implies in particular that the stable shape descriptors built in practice from the sample of distances have access to all the geometric information of the underlying space.

For a corresponding mmpds of a stable shape descriptor, another equivalent characterization can be found. Indeed, under more restrictive assumption on the measure, a mmpds can equivalently be   completely characterized   by the process of ball volumes of the finite-dimensional distributions (fidis), via Kolmogorov's extension theorem. This result is particularly useful because even though the underlying metric measure dynamical system need not satisfy the necessary assumptions, (e.g., is not finite-dimensional or approximable by finite-dimensional spaces)
 the target spaces for most stable shape descriptors used in practice do. 
This latter characterization  forms the blueprint for very simple test statistics that are introduced in \autoref{sec4}.

To make this precise, denote the ball volume on the  $k$-dimensional product metric space $S^k$ by $B(s,r) =(s^\prime: \rho_{k}(s,s^\prime)\le r), $
where $\rho_{k}(s,s^\prime)=\max_{1\le j \le k}\partial_S(s_j, s^\prime_j)$, which metrizes the product topology for any $k \in \mathbb{N}$. Denote the distribution of the fidi on the time index set $J$ of the process $X$  by $\mu^X_{J}$.   Observe that we can associate the metric measure space $(S^{|J|},\rho_{|J|},\mu^X_{J})$ to this fidi.

\begin{thm}\label{thm:charmux}
Let $(X_t: t\in  \znum)$ be a Polish-valued stochastic process with law $\mu_{X}$
 being a regular finite Borel measure that is determined by its values on balls. 
 Let
\[
\psi( (\pi_{J} \circ X), r) =\mu^X_{J}\big( B(\pi_{J} \circ X, r)\big),  \quad \pi_J\circ X \sim \mu^X_{J}, 
\]
then the process $\big(\psi( (\pi_{J} \circ X), r): r \ge 0\big)$ characterizes the $|J|$-dimensional dynamics on the time set $J$, i.e.,
 the equivalence class of metric measure spaces $(S^{|J|}, \rho_{|J|}, \mu^X_{J})$. 
\end{thm}

\begin{proof}
See~\cite[\S 2-\S 3]{r22} for careful discussion of the required Vitali covering property and~\cite[\S 2]{BWM18} for an outline of the argument. 
\end{proof}

\begin{Remark}\label{rem:balls}
The assumption in \autoref{thm:charmux} on $\mu_X$ is satisfied by any Borel measure on a variety of underlying metric spaces that arise in geometric data analysis, including doubling metric spaces, many Banach spaces, and any Hilbert space with the norm metric.  Most importantly, it holds for most of the spaces of values taken by shape descriptors, including the {\em barcode space} of topological data analysis~\cite{bvd23}.
\end{Remark}

A consequence of the preceding theorems is that one can often characterize the geometric features as captured by shape descriptors of a Polish-valued process $X$ via the ball volume processes of the fidis.  This is convenient for hypothesis testing. For example, note that the ball volume corresponding to the fidi on the time set $J$ is given by
\[\psi( (\pi_{J} \circ X), r)= \E_{(X^\prime) \sim \mu_J}\big[\prod_{j \in J}\mathrm{1}_{\partial_S(X_j, X^\prime_j)\le r}\big]. \tageq \label{eq:exp}\]
Thus, a natural estimator based on a sample $(X_t)_{t\in [T]}$ takes the form of a $U$-process.  There is a rich literature on $U$-statistics, although primarily focused on strong mixing and stationary data (e.g., see~\cite{bz12,bn22,BBD01,Dehetal2013,HsingWu04,VW17}). In order to develop  inference methods within the more involved framework put forward in \autoref{sec2}, we require a weak invariance principle for nonstationary $U$-processes (\autoref{thm:2parproccon}) indexed by both time and radii, under less stringent assumptions. We introduce this invariance principle within the context of detecting topological change in the below, but emphasize that the result can be used for other hypotheses, and is therefore of independent interest.

\section{Application: Detecting changes in geometric features} \label{sec4}

We will illustrate the framework laid out by proposing a test to detect nonstationary behavior in the geometric features of a metric space-valued process $(\mb{X}_t: t \in \znum)$. Stationarity testing in (functional) time series data has been an active area of research for decades, mainly focusing on detecting (abrupt) nonstationary behavior in the first- and second-order structure (see e.g.,~\citep{ak12-aoas, AueRS2018,avd, bhk, bdh22+, cper21,sz10,ssrn} and references therein). 
 Here, we are interested in detecting \textit{arbitrary nonstationary} behavior in the distribution caused by changes in the geometric properties. Proofs are deferred to the online supplement.

Assume $\mb{X}:=(\mb{X}_t\colon t \in \znum)$ adheres to \autoref{def:Localstationarity}, and denote its law on the time index set $J$ again by $\mu_{J}^{\mb{X}}$. 
If the dynamics are time-dependent we have $\mu_{J}^{\mb{X}} \neq \mu_{J+h}^{\mb{X}}$ for some $h \in \znum$. To reduce clutter, we focus in the exposition below on the marginal distribution, i.e., $\nu_t :=\mu_{\{t\}}^{\mb{X}} =\mb{P} \circ (\mb{X}_t)^{-1}$. We emphasize that this easily generalizes to other choices for $J$.  

By \autoref{thm:charmux} and \autoref{rem:balls}, a natural way to formulate a test for time-invariance of the geometric features is via the corresponding ball volume processes of the shape descriptors. In addition, results in \autoref{sec2} show that we can express approximation errors in terms of the Gromov-Hausdorff distance $d_{GH}$, which allows us to translate dependence assumptions imposed on the latent process to those of the $(\mathscr{B}, d_{\mathscr{B}})$-valued process (see \autoref{lem:weakgmc} and \autoref{lem:strongmc}). It follows from 
    \autoref{prop0} that the distribution of $\sh(\mb{X}_t(M))$,  denoted by $\tilde{\nu}_t$, is invariant under time translations if $\nu_t=\nu$. 
   Define the corresponding ball volume process $(\varphi_t(r): r \ge 0)$ by \[
\varphi_t(r)=\tilde{\nu}_t\big(B(\sh(\mb{X}_t(M)),r)\big), \quad \sh(\mb{X}_t(M)) \sim \tilde{\nu}_t, \tageq \label{varmar}\]
which characterizes the measure $\tilde{\nu}_t$ up to isometry by \autoref{thm:charmux}. In order to test the hypothesis that the marginal distribution is constant over time against the alternative of the distribution being time-dependent, we propose the following pair of hypotheses
\[
H_0\colon \E\varphi_t(r)=\E{\varphi}(r), \, \forall t\in \znum,  r \in [0,\mathscr{R}] \,\, \text{ vs.} \,\, H_A\colon  \E\varphi_t(r) \neq \E{\varphi}(r) \text{ for some $ t\in \znum$,$r \in [0,\mathscr{R}]$}. \tageq\label{eq:hyp}
\]
 We explain the methodology in detail in the next subsection. 

\subsection{\textbf{Methodology: a weak invariance principle in $D([0,1]\times [0,\mathscr{R}])$}}\label{sec:wip}
To construct a test for the hypotheses in \eqref{eq:hyp}, define the partial sum process
\[
\mathscr{S}_T(u,r) =\frac{1}{T^2}\sum_{s,t=1}^{\flo{uT}} h(\tilde{\mb{X}}^{n(T)}_t,\tilde{\mb{X}}^{n(T)}_s,r) \quad u \in [0,1], r \in [0,\mathscr{R}],\tageq \label{eq:partsum}
\]
where for any compact metric space $(Y,\partial_Y)$ the kernel $h:Y \times Y \times [0,\mathscr{R}] \to \rnum$ is  given by 
$h(x,x^\prime, r) =\mathrm{1}{\big\{d_\mathscr{B}\big(\sh(x), \sh(x^\prime\big) \le r\big\}}$. 
Define
 \[
\sqrt{T} U_T(u,r)= \sqrt{T}\big(\mathscr{S}_T(u,r)-u^2 \mathscr{S}_T(1,r)\big), \tageq \label{eq:ustat2}
\]
and observe that under the null hypothesis $\E U_T(u,r) = 0$. A test can therefore be based on suitably normalized versions of, respectively,  
 \[\sup_{r \in [0,\mathscr{R}]}\sup_{u\in [0,1]}
\sqrt{T}\bv U_T(u,r)\bv\tageq \label{eq:ustat1};\]
 \[T\int_0^{\mathscr{R}}\int_0^1
 \big(U_T(u,r)\big)^2 du dr.\tageq \label{eq:ustat3}\]
 
Under regularity conditions given below, \eqref{eq:ustat2} converges weakly in Skorokhod space $D([0,1]\times [0,\mathscr{R}])$ to a Gaussian process, which has a succinct form under the null. 

In order to analyze the large sample behavior of these statistics, assumptions on the dependence structure for the original and auxiliary process are required. We assume that the random functions admit representations
\[
\mb{X}_{t}:=\mb{X}_{t,T}=g\big(t,T,\mathfrak{f}_t \big),  \quad \mb{X}_{t}(u):=\ddot{g}\big(u,\mathfrak{f}_t \big),\tageq\label{eq:bershift}
\]
where $\mathfrak{f}_t =(\zeta_t, \zeta_{t-1}, \ldots)$ for an iid sequence $\{\zeta_t:  t\in\mathbb{Z}\}$  of random elements taking values in some measurable space $\mathfrak{G}$, and where  $g\colon\znum\times \mathbb{N} \times \mathfrak{G}^{\infty} \to S$, and $\ddot{g}\colon [0,1] \times \mathfrak{G}^{\infty} \to S$ are measurable functions. Note that it follows immediately from the results in \autoref{sec2} that
\[
\sh(\mb{X}_{t})=\tilde{g}\big(t,T,\mathfrak{f}_t \big),  \quad \sh(\mb{X}_{t}(u)):=\tilde{\ddot{g}}\big(u,\mathfrak{f}_t \big),\tageq\label{eq:bershift2}
\]
where $\tilde{g} = \sh \circ g$ and $\tilde{\ddot{g}} = \sh \circ \ddot{g}$. We formulate mild conditions via appropriate coupled versions of the random objects. To make this precise, let $\G_t = \sigma(\zeta_s, s\le t)$ denote the natural filtration of $(\zeta_t: t\in\mathbb{Z})$ and $\G_{t,m}=\sigma(\zeta_t, \ldots, \zeta_{t-m+1}, \zeta_{t-m}^\prime, \zeta^\prime_{t-m-1}, \ldots)$, where $\zeta_j^\prime$ is an independent copy of $\zeta_j$. In addition, define the filtration $\G_{t,\{t-j\}} := \sigma(\zeta_t, \ldots, \zeta_{t-j+1}, \zeta^\prime_{t-j}, \zeta_{t-j-1}, \ldots)$, which is the natural filtration generated by $(\zeta_{s}, s\le t)$, but where element $\zeta_{t-j}$ is replaced by an independent copy. For $\mb{X} \in \mathcal{L}^2_S$,  consider the following notion of conditional expectation 
\[
\arg{\inf}_{\substack{Z: \Omega \to S\\ \sigma(Z) \subset \mathcal{\G}}} \E \partial^2_S\big(\mb{X}, Z\big)~. 
\]
Note that this set is non-empty and a version exists since $S$ is assumed to be compact. We shall denote the elements of this set by $\E[\mb{X}|\G]$.  If $S$ is a normed vector space, then the above definition coincides with the classical notion for separable Banach spaces.  We define $m$-dependent, $m\ge 1$,  coupled versions of the metric space-valued random variables by $
\mb{X}_{m,t}:=\E\big[\mb{X}_{t}|\G_{t,m}\big]$
so that $\mb{X}_{m,t} \in\sigma(\zeta_t, \ldots,\zeta_{t-m+1})$ and, by \autoref{prop1},  $\sh({\mb{X}}_{m,t}(M))  \in\sigma(\zeta_t, \ldots,\zeta_{t-m+1})$.  In addition, we define another coupled version of $\mb{X}_{t}$ by $
\mb{X}_{\{t-j\},t}:=\E\big[\mb{X}_{t}|\G_{t,\{t-j\}}\big]$ so that $\mb{X}_{\{t-j\},t}$ as well as $\sh(\mb{X}_{\{t-j\},t})$ are $\G_{t,\{t-j\}}$-measurable. The latter are generalizations of the form put forward in  \cite{vd}, and closely related to the physical dependence measure of \cite{wu05}. The locally stationary coupled versions can be defined similarly. Observe that if $\mb{X}_t$ is $m$-dependent then $\mb{X}_{m,t}=\mb{X}_t$, whereas if $\mb{X}_t$ is independent of $\zeta_{t-j}$ then $\mb{X}_{\{t-j\},t} = \mb{X}_t$, $\mb{P}$-almost surely. Hence, $
\E \partial^p_S\big(\mb{X}_t, \mb{X}_{j,t}\big)$ and $ \E \partial^p_S\big(\mb{X}_t, \mb{X}_{\{t-j\},t}\big) 
$  
can be seen to quantify the dependence of $\mb{X}_t$ on the past of at least $j$ steps away and on $\zeta_{t-j}$, respectively.

As shown in the next lemma, these two measures of dependence can be immediately compared if $(S,\partial_s)$ is a Hadamard space. A Hadamard space is  a complete geodesic metric space of nonpositive curvature. The convexity properties of these spaces allow one to show \cite{sturm2002} that the conditional expectation is a contraction in $\mathcal{L}_S^p$, $p \in [1,\infty]$.  Notable examples include the Billera-Holmes-Vogtman space of phylogenetic trees, the Poincar{\'e} disk (and hyperbolic spaces more generally), and any negatively curved Riemannian manifold with the intrinsic metric.
\begin{lemma}
   If $(S,\partial_s)$ is a Hadamard space, then for $p \ge 1$ 
    \[
   \E \partial^p_S\big(\mb{X}_t, \mb{X}_{\{t-j\},t}\big)\le  2\E \partial^p_S\big(\mb{X}_t, \mb{X}_{j,t}\big)~.
    \]
\end{lemma}
\begin{proof}
By the triangle inequality 
    \begin{align}
\E \partial^p_S\big(\mb{X}_t, \mb{X}_{\{t-j\},t}\big) \le  \partial^p_S\big(\mb{X}_t, \E[\mb{X}_{j,t}|\G_{t,\{t-j\}}]\big)+\partial^p_S\big( \E[\mb{X}_{j,t}|\G_{t,\{t-j\}}], \E[\mb{X}_t|\G_{t,\{t-j\}}]\big)~.
\end{align}
Both terms on the right-hand side are bounded
by $\E \partial^p_S\big(\mb{X}_t, \mb{X}_{j,t}\big)$ which follows by $\G_{t,\{t-j\}}$-measurability of $\mb{X}_{j,t}$ and the contraction property, respectively. 
\end{proof} 

We emphasize that we do not require $(S,\partial_S)$ to be a Hadamard space in the following discussion, and that the results put forward hold for any compact metric space.

\begin{Remark}
   Instead of taking $\mb{X}_{\{t-j\},t}$ one could also consider the  physical dependence version $
\ddot{\mb{X}}_{\{t-j\},t} = g(t, T, \mathfrak{f}_{t,\{t-j\}}
)$, 
where  the sequence $\mathfrak{f}_{t,\{t-j\}}$ has the entry with index $t-j$ replaced with an independent copy. However, if $(\mb{X}_t) \in \mathcal{L}^2_S$ and $(S,\partial_S)$ is Hadamard, then a similar argument as in the above lemma shows that the version $\mb{X}_{\{t-j\},t}$ implies slightly weaker assumptions. \end{Remark}

We now turn to the conditions required to prove the weak invariance principle, which inherently relies on conditions on the kernel of the $(\mathscr{B},d_{\mathscr{B}})$-valued process. We relate this to above dependence measures in \autoref{lem:weakgmc} and \autoref{lem:strongmc}, respectively. The representations \eqref{eq:bershift}, \eqref{eq:bershift2} provide natural conditions on the demeaned kernel in terms of projections. More specifically, let $\hti$ indicate the demeaned kernel; $\hti(\cdot,\cdot, r) = h(\cdot,\cdot,r)- \E h(\cdot,\cdot, r)$. Then for an integrable real-valued random variable $Z$, define the projection operator 
 $P_j(Z) = \E[Z|\G_{j}]-\E[Z|\G_{j-1}]$, $j \in \znum$, 
. For an $(Y,\partial_Y)$-valued process $W =(W_t: t \in [T], T \in \nnum)$, we introduce the following dependence measure
\[
\nu^W_{t,s}(j,r)=\sup_{m}\Big\| P_{t-j}\Big(\hti(W_{m,t}, W_{m,s},r)\Big)\Big\|_{\rnum,2}, \tageq\label{thenus}
\]
where we denoted $\|\cdot\|_{\rnum,q}=(\E\|\cdot\|^q_{\rnum})^{1/q}$. 
Upper bounds can be expressed in terms of
\begin{align*}\delta^W_{m,q}(r) &= \sup_{T \in \nnum} \sup_{\substack{t,s\in [T]: s \le t}}  \big\|\hti(W_{t},W_{s},r)-\hti(W_{m,t},W_{m,s},r)\big\|_{\rnum,q}, \tageq\label{thedeltas}\\
  \theta^W_{j,j-k,q}(r) &= \sup_m \sup_t \big\|\hti(W_{m,t},W_{m,t-k},r) -\hti(W_{\{t-j\},m,t},W_{\{t-j\},m,t-k},r)\big\|_{\rnum,q},\tageq\label{eq:thethetas}
\end{align*}
respectively. We write $\delta^W_m:=\delta^W_{m,2}$ and $\theta^W_{j,j-k}:=\theta^W_{j,j-k,2}(r)$. 
We make use of the following result, which is an extension of proposition 4 of \cite{HsingWu04}, and which relates \eqref{thenus} to \eqref{thedeltas}-\eqref{eq:thethetas}.
\begin{Proposition} \label{sumbound}
Let $(W_t)$ be adapted to the filtration $(\G_t)$. Then 
\begin{enumerate}\itemsep1.5ex
\item[(i)] $\sup_{k \ge 0} \sum_{j=0}^{\infty} \sup_t\Big\|P_{t-j}(\hti(W_t,W_{t-k},r))\Big\|_{\rnum,q}  \le 2\sum_{j=0}^{\infty} \delta^W_{j,q}(r) \wedge \sup_{k \ge 0} 2\sum_{j=0}^{\infty} \theta^W_{j,j-k,q}(r)  $;
\item[(ii)] for any $\epsilon$, 
\begin{align*}
   \sup_{k \ge 0} \sum_{j=0}^{\infty} \min({\nu}^W_{t,t-k}(j,r),\epsilon) \le  4 \sum_{j=0}^{\infty} \min(\sup_{m > j}\delta^W_m(r),\epsilon) \wedge  2\sup_{k \ge 0} \sum_{j=0}^{\infty} \min(\theta^W_{j,j-k}(r) ,\epsilon), 
\end{align*}
\end{enumerate}
where $x\wedge y =\min(x,y)$.
\end{Proposition} 
Before turning to the conditions on the latent processes, we make the assumptions on the kernel precise that ensure the weak invariance principle holds.   
\autoref{as:dist1} below gives these for a fixed collection of radii. This is then strengthened underneath \autoref{thm:limfixr} to an invariance principle with weak convergence in $D([0,1] \times [0,\mathscr{R}])$. 
\begin{assumption}\label{as:dist1} \label{as:dep}
The following conditions hold for $q \ge 2$:
\begin{enumerate}[label=(\roman*)]
\item For some $\kappa>q/2$, 
$$
\sup_{t \neq s, r \in \rnum} \pr\Big(r \le \partial_{\mathscr{B}}\big(\sh(W_{t}),
\sh(W_{s})\big) \le r+\epsilon\Big) = O\big(\log^{-2\kappa}(\epsilon^{-1}) \big)\quad \forall  \,0<\epsilon <1/2; 
$$
 \item $\sum_{j=0}^{\infty} \sup_{m > j}\delta^W_{m,q}(r)<\infty$ or $ \sup_{k \ge 0}\sum_{j=0}^{\infty}\theta^W_{j,j-k,q}(r) <\infty$ and $\lim_{j \to \infty} \delta^W_{j,q}(r)=0$.
\end{enumerate}
\end{assumption}
By \autoref{sumbound}, \autoref{as:dep}(ii) implies 
\[\lim_{\epsilon \to 0}  \sup_{k\ge 0} \sum_{j=0}^{\infty} \sup_t\min(\nu^W_{t,t-k}(j,r), \epsilon) = 0. \tageq \label{eq:limepssum}\]
To elaborate on its use, note that  \eqref{eq:partsum} can be viewed as a $U$-statistic applied to a nonstationary process. The proof  relies on a type of Hoeffding decomposition, which essentially decomposes the statistic (after scaling) into a linear part that defines the distributional properties, and an error component that should converge to zero in probability at an appropriate rate. The latter component consists of several (atypical) approximation errors, among which the error from the data arising in the form of point clouds and a nonstationary error (see the proof of \autoref{thm:limsupm}). These can be controlled meaningfully if a result such as \eqref{eq:limepssum} holds.

The  assumptions introduced here are relatively mild compared to  existing literature in the stationary case (e.g., see~\cite{BBD01,HsingWu04,DGS21}).  To the best of our knowledge, no relevant results exist in the nonstationary setting. If \autoref{as:dep}(i) holds, conditions on the latent process such as geometric moment contraction assumptions are sufficient to ensure the summability conditions on the kernel measures in \autoref{as:dep}(ii)) are satisfied. 

\begin{lemma}\label{lem:weakgmc}
Let $0<a(p), b(p)<1$ for some $p\ge 1$.   
Under  \autoref{as:dep}(i); 
\begin{enumerate}
    \item[(a)] if $\sup_{t}\E\partial^p_S\big(\mb{X}_{t}, \mb{X}_{m,t}\big) \le a(p)^m$ then  $\sup_r \sum_j \sup_{ m > j} \delta^{\sh(\tilde{\mb{X}})}_{m,q}(r) < \infty$;
    \item[(b)] if $\sup_{m,t} \E \partial^p_S\big(\mb{X}_{m,t}, \mb{X}_{\{t-j\},m,t}\big) \le b(p)^j$
then $
\sup_r \sup_k\sum_j \theta^{\sh(\tilde{\mb{X}})}_{j,j-k,q}(r)< \infty$.
\end{enumerate}
\end{lemma}

We now obtain a functional CLT for \eqref{eq:partsum} as a process in $u \in [0,1]$.

\begin{thm}\label{thm:limfixr}
Let $(\mb{X}_{t,T})$ satisfy \autoref{def:Localstationarity}(i) with  $(S,\partial_S)=(C(M,M^\prime),\rho)$, and assume that  $(\tilde{\mb{X}}_{t,T})$ and    $(\tilde{\mb{X}}_t(\frac{t}{T}))$ with
representations \eqref{eq:bershift} fulfill \autoref{as:dep} with $q=2$. Furthermore, assume  \autoref{as:mincov} with  $n(T) \to \infty$ as $T\to \infty$. Then the partial sum process given by \eqref{eq:partsum} satisfies for any fixed $r_1, \ldots, r_d \in [0,\mathscr{R}]$,  
\[\Big\{T^{1/2} \Big(\mathscr{S}_T(u,r_i)-\E \mathscr{S}_T(u,r_i)\Big) : i\in [d]\Big\}_{u \in [0,1]}{\rightsquigarrow} \Big\{\mathbb{G}(u,u,r_i)  : i\in [d]\Big\}_{u \in [0,1]} \quad (T \to \infty) \]
in $D[0,1]$ with respect to the Skorokhod topology, where $\{\mathbb{G}(u,v,r_i): i\in [d]\}_{v \le u \in [0,1], }$ is a zero-mean Gaussian process given by
\[
\mathbb{G}(u,v,r_i):=\int_0^u \sigma(\eta,v,r_i) d\mathbb{B}(\eta),  \tageq\label{eq:Gauslim}\]
where $\mathbb{B}$ is a standard Brownian motion and
\begin{align*}
 \sigma^2(\eta,v,r)
= 4\sum_{k \in \znum} \text{Cov}\Big(\int_0^v\E_{\tilde{\mb{X}}^\prime_0(w)}[\hti(\tilde{\mb{X}}_0(\eta), \tilde{\mb{X}}^\prime_0(w),r)]dw, \int_0^v\E_{\tilde{\mb{X}}^\prime_0(w)}[\hti(\tilde{\mb{X}}_k(\eta),\tilde{\mb{X}}^\prime_0(w),r)]dw\Big)
\end{align*}
with $(\mb{X}^\prime_t(u): t\in \znum, u \in [0,1])$ an independent copy of  $(\mb{X}_t(u): t\in \znum, u \in [0,1])$, and of which the covariance structure is given by
\begin{align*}
Cov(\mb{G}(u_1, v_1,r_1), \mb{G}(u_2,v_2, r_2))= \int_0^{\min(u_1,u_2)} \sigma(\eta, v_1, r_1)\sigma(\eta,v_2,r_2) d \eta.
\end{align*}
\end{thm}

Weak convergence in $D([0,1]\times [0,\mathscr{R}])$ requires slightly stronger assumptions. 
\begin{thm} \label{thm:2parproccon}
Assume the conditions of  \autoref{thm:limfixr} hold, and that 
\begin{enumerate}[label=(\roman*)]
\item  $\sup_n\sup_{k\ge 0}\sum_{j=0}^{\infty} j\sup_{t}\big\|P_{t-j}\big(\hti(\mb{\tilde{X}}^{n}_{t},\mb{\tilde{X}}^{n}_{t-k},r)\big)\big\|_{\rnum,q}<\infty$,  $q=4+\gamma$, $\gamma>0$; 
\item  There exists a constant $C^\prime$ such that for all $t,T$, 
\begin{align*}
&\sup_{k\ge 0}\sup_{a<b} \frac{1}{b-a}{\pr\Big(\partial_\mathscr{B}\big(\sh(\mb{\tilde{X}}^{n(T)}_{t}),\sh(\mb{\tilde{X}}^{n(T)}_{t-k})\big) \in (a,b]\bv \G_{t-1}\Big)}  \le C^\prime  \text{ almost surely}. \tageq \label{densuni1}
\end{align*} 
\end{enumerate}
Then
\[\Big\{T^{1/2} \Big(\mathscr{S}_T(u,r)-\E \mathscr{S}_T(u,r)\Big)\Big\}_{u \in [0,1], r\in [0,\mathscr{R}]}\overset{D}{\rightsquigarrow} \Big\{\mathbb{G}(u,u,r)\Big\}_{u \in [0,1],r\in [0,\mathscr{R}]} \quad (T \to \infty) \]
in $D([0,1]\times [0,\mathscr{R}])$ with respect to the Skorokhod topology, where $\{\mathbb{G}(u,v,r)\}_{u,v \in [0,1], r\in [0,\mathscr{R}]}$ is the zero-mean Gaussian process defined in \eqref{eq:Gauslim}.
\end{thm}

 Condition (i) is necessary to derive tightness, and is implied by the following assumption on the latent process.

\begin{lemma}\label{lem:strongmc}

Assume that $\sup_{t}\E\partial_S\big(\mb{{X}}_{t}, \mb{{X}}_{m,t}\big) \le a^m$ for some $0<a <1$ and that \eqref{densuni1} holds. Then condition (i) of \autoref{thm:2parproccon} is satisfied.
\end{lemma}

\begin{Remark}[{Correlation dimension}]
It is worth mentioning that    \autoref{thm:2parproccon} can be used to draw inference on the correlation dimension, i.e., $D=\lim_{r \to 0} \log\big(\tilde{\varphi}_t(r)\big)/\log(r)$.
Correlation dimension is a type of fractal dimension \citep{GrPr83}, which describes the complexity of a dynamical system in terms of the ratio of change of detail to the change in scale.  
    \end{Remark}
    
With the invariance principle on the partial sums in place, we turn to the test statistic. To ease notation below, we write $\mathbb{G}(u,r):=\mathbb{G}(u,u,r)$. Under the conditions of \autoref{thm:2parproccon}, the continuous mapping theorem yields for the process defined in~\eqref{eq:ustat2}
\begin{align*}
\Big\{T^{1/2} & \Big(U_T(u,r)-\E U_T(u,r)\Big) \Big\}_{u \in [0,1],r\in [0,\mathscr{R}]}
 \underset{T \to \infty}{\longsquiggly}\Big\{
 \mb{G}(u,r) - u^2 \mb{G}(1,r)
 \Big\}_{u \in [0,1],r\in [0,\mathscr{R}]}.
\end{align*}
Under the null hypothesis of stationarity this reduces to
\[
\Big\{T^{1/2} U_T(u,r) \Big\}_{u \in [0,1],r\in [0,\mathscr{R}]}\quad \underset{T \to \infty}{\longsquiggly}\quad \Big\{u \sigma
(r)\big( \mathbb{B}(u)-u\mathbb{B}(1)\big)\Big\}_{u \in [0,1],r\in [0,\mathscr{R}]}.
\]
Given a consistent estimator  $\hat{\sigma}_T(r)$ of $\sigma(r)$ one finds
\[
\sup_r\sup_{u}\frac{\sqrt{T}|\mathscr{S}_T(u,r)-u^2 \mathscr{S}_T(1,r)|}{\hat{\sigma}_T(r)} \underset{T \to \infty}{\Longrightarrow} \sup_{u}  u|\mathbb{B}(u)-u\mathbb{B}(1)|~, \tageq \label{test:nonpiv}
\]
and
\begin{align}
T\int_0^\mathscr{R} \int_0^1 \frac{  \big({\mathscr{S}}_T(u,r)-u^2 {\mathscr{S}}_T(1,r)\big)^2 }{  \hat{\sigma}^2_T(r) } du dr\underset{T \to \infty}{\Longrightarrow} \mathscr{R}  \int_0^1 u^2 \big( \mb{B}(u)-u \mb{B}(1)\big)^2 du. \tageq \label{test:nonpiv2}
\end{align}
Thus, a natural decision rule is to reject the null hypothesis in~\eqref{eq:hyp} based on~\eqref{eq:ustat1} (resp. \eqref{eq:ustat3})  
 whenever the left-hand side of~\eqref{test:nonpiv} (resp. \eqref{test:nonpiv2})
is larger than the $(1-\alpha)$th quantile of the distribution of the random variable on the right-hand side. However, even under the null hypothesis the covariance structure relies on a  long-run covariance. 
It is well-known in more classical settings that estimation of the long-run covariance is not an easy task because it requires selecting a bandwidth parameter. The size and power of the test can be heavily distorted if the choice of the parameter does not incorporate some essential information on the dependence structure \citep{Vogelsang1999,DengPerron08}. To avoid this issue altogether, we introduce a self-normalization approach next.

\subsection{Self-normalized tests}  
 
We introduce two families of range-based self-normalized versions, one for \eqref{eq:ustat1} and one for \eqref{eq:ustat3}. 
To make the first of these  precise, define  
\begin{align*}
{V}_T(r)=\max_{1\le k\le T} \big(U_T(k/T,r)-\E U_T(k/T,r\big)- \min_{1 \le k \le T} \big(U_T(k/T,r)-\E U_T(k/T,r)\big).  \tageq \label{eq:varests}
\end{align*}
We propose the following self-normalized  test for \eqref{eq:ustat1} using \eqref{eq:varests}
\[
\mathbb{D}^{\text{max}}_T:=\max_r  w(r) \max_k {V_T({r})}^{-1}|U(k/T,r)|,  \tageq \label{eq:Dtmax}\]
where $w(r)\ge 0$ is a weight function. 
Note that an introduction of a weight function is quite common in the literature to control for outlier patterns in ecdf-based tests. 
The finite sample performance and the presence of outliers is  particularly relevant in our context since genomic data applications are typically in the regime of $T<100$ and $n < 10K$. 
 In Section \ref{sub:datadriven}, we elaborate on this and introduce a data-adaptive function that changes the distribution of weights to or from the tails based on the properties of the increments of the time-dependent empirical cdf's.

\begin{Corollary}\label{cor:Dmax}
Under the conditions of \autoref{thm:2parproccon},
\begin{align*}
&\max_{r,k}  \frac{w(r)|U_T(k/T,r)-\E U_T(k/T,r)|}{{V}_T(r)} 
\\&\phantom{maxrk}\underset{T \to \infty}{\Rightarrow} 
 \sup_r  \frac{ w(r) \sup_{u}\vert \mathbb{G}(u,r) -u^2 \mb{G}(1,r)\vert}{
\sup_{u}\big(\mathbb{G}(u,r) -u^2 \mb{G}(1,r)\big) -\inf_{u}\big(\mathbb{G}(u,r) -u^2 \mb{G}(1,r)\big)}~.
\end{align*}
Furthermore, assume that  $\max_r w(r) =1$. Then under $H_0$, \eqref{eq:Dtmax} satisfies
\[
\mathbb{D}^{\text{max}}_T  \underset{T \to \infty}{\Rightarrow}  \sup_{u}  \frac{ |u\mathbb{B}(u)-u^2\mathbb{B}(1)|}{ 
\sup_u \big(u\mathbb{B}(u)-u^2\mathbb{B}(1)\big) -\inf_{u}\big(u\mathbb{B}(u)-u^2\mathbb{B}(1)\big)} =: \mathbb{D}^{\text{max}}. \tageq\label{eq:test1}
\]
\end{Corollary}
To consider quadratic tests, define 
\begin{align*}
V^L_T(r_i)=
\max_{1\le k\le T} \Big(U_T\big(\frac{k}{T},r_i\big)-\E U_T\big(\frac{k}{T},r_i\big)\Big)^2 - \min_{1 \le k \le T} \Big(U_T\big(\frac{k}{T},r_i\big)-\E U_T\big(\frac{k}{T},r_i\big)\Big)^2 ,  \tageq \label{eq:varest2}
\end{align*}
where $r_i \in (0,\mathscr{R})$. Henceforth, we take $i = 1,\ldots, T$, $r_i =i\mathscr{R}/T$. A self-normalized version of \eqref{test:nonpiv2} based on \eqref{eq:varest2} is given in the following corollary: 
\begin{Corollary}\label{cor:DL}
Under the conditions of \autoref{thm:2parproccon},
\begin{align*}
&\frac{\mathscr{R}}{T^2} \sum_{i=1}^T \sum_{k=1}^T  w(r_i) \frac{\big(U_T(k/T,r_i)-\E U_T(k/T,r_i)\big)^2}{ V^L_T(r_i)}
\\& \quad \quad \quad \underset{T \to \infty}{\Rightarrow} \int_0^{\mathscr{R}} \int_0^1 \frac{  w(r)\big(\mb{G}(u,r) - u^2 \mb{G}(1,r)\big)^2 }{\sup_{u}\big(\mathbb{G}(u,r) -u^2 \mb{G}(1,r)\big)^2 -\inf_{u}\big(\mathbb{G}(u,r) -u^2 \mb{G}(1,r)\big)^2 }du dr~.
\end{align*}
Furthermore, assume that  $\sum w(r) =1$. Then we obtain under $H_0$
\begin{align*}
\mb{D}^L_T:&=\frac{ \mathscr{R}}{T^2} \sum_{i=1}^T \sum_{k=1}^T w(r_i) \frac{\big(U_T(k/T,r_i)\big)^2}{ V^L_{T,H_0}(r)} 
\\& \underset{T \to \infty}{\Rightarrow} \frac{ \int_0^1 \big(u\mb{B}(u) - u^2 \mb{B}(1)\big)^2 du  }{\sup_{u}\big(u\mathbb{B}(u) -u^2 \mb{B}(1)\big)^2 -\inf_{u}\big(\mathbb{B}(u) -u^2 \mb{B}(1)\big)^2 } =: \mathbb{D}^{L}, \tageq\label{eq:DtL}
\end{align*}
where $V^L_{T,H_0}(\cdot)$ is the estimator \eqref{eq:varest2} with the demeaned kernel $\hti$ replaced with $h$.
\end{Corollary}
 We propose several choices for the weights, which are explained below and compared in a simulation study.

\begin{Remark}[{ Computing the test statistics}] \label{compu}
The self-normalized tests are extremely simple to implement once the ensemble of point clouds has been transformed by the stable shape descriptor. To test for \eqref{eq:hyp}, one merely needs to compute the statistic of choice $\mb{D}^{\text{max}}_{T}$ and $ \mb{D}^L_{T}$ as well as the corresponding quantiles $q_{1-\alpha}(\mb{D}^{\text{max}})$ and $q_{1-\alpha}(\mb{D}^L)$, which can be easily simulated by the user. Implementation of the stable shape descriptors and the corresponding metric relies on existing code. For example, to compute  persistent homology and the bottleneck distance in Section \ref{sec:sim}, we used Ripser and Hera~\citep{b19, ripser, hera}.
\end{Remark} 
 
 \subsubsection{Data-driven weight function} \label{sub:datadriven}

Outliers are common in genomic data, and many invariants from geometric data analysis are sensitive to outliers (e.g., see~\cite[\S 4]{bgmp14}). A suitable weight function can mitigate their effect. However, a weight function that suppresses outliers {\`a} priori can increase the type II error since irregularities are also an indication of  nonstationary behavior.  Therefore, we need an appropriate function that suppresses outlier patterns under $H_0$ but simultaneously controls the type II error. Note that this function must be time-independent.

To fix ideas, suppose $F: [0,\mathscr{R}] \to [0,1]$ is a cdf that captures the distribution pattern. Natural choices for the weights are functionals of $F$ such as  
\[
\tilde{w}_a(r) = \big( {F}(r)(1-{F}(r))\big)^a \text{ and } \dot{w}_a(r) =  {F}(r)^a+\big(1-{F}(r)\big)^a~.\tageq \label{eq:we}
\]
A disadvantage is that different outlier patterns typically require a different functional and an inappropriate choice can negatively affect statistical properties~\cite{ozturk}. 
Alternatively, a more adaptive choice would use the density approximation obtained by the increments
\[
{I}(r_i) := {F}(r_{i})-{F}(r_{i-1}). \tageq \label{eq:incr}
\]
Under stationarity, \eqref{eq:incr} can be used effectively to reduce the effect of outliers when $F$ is the common cdf. 
However, the cdf is time-dependent under the alternative and an aggregrate $F$ at most reflects the average distribution pattern.
To construct the weight function, we need an empirical version of the most suitable choice for $F$ in our case. For  this, we use the time-averaged empirical cdf
 \[
\bar{\mathscr{S}}_T(r):=\frac{1}{T}\sum_{k=1}^T \frac{\mathscr{S}_T(k/T,r)}{k^{3/2}} T^{3/2}~.
 \]
Note that under the conditions of \autoref{thm:2parproccon},
\[
 \Big\{\sqrt{T}\bar{\mathscr{S}}_T(r)-\frac{1}{\sqrt{T}}\sum_{k=1}^T \frac{\E \mathscr{S}_T(k/T,r)}{k^{3/2}} T^{3/2}\Big\}\medsquiggly\Big\{\int_0^1 u^{-3/2}\mb{G}(u,r) du\Big\}_{ r \in [0,\mathscr{R}]},
\]
which converges to 
$\sigma(r)\int_0^1 \mb{B}(u) du$ under $H_0$. 
If additionally the weight functional $w$ is differentiable with uniformly bounded first derivative, Taylor's theorem yields  $w(\bar{\mathscr{S}}_T(r)) = w(\E \bar{\mathscr{S}}_T(r))+ o_p(1)$. Thus, \autoref{cor:Dmax} and \autoref{cor:DL} hold with such weight functionals applied to the time-averaged ecdf. 
Under $H_0$, the local ecdf's are equal in distribution and $\bar{\mathscr{S}}_T$ approximates the common cdf $\mathscr{S}$. Under $H_A$, the time-averaged cdf is the closest time-independent cdf to the collection of time-dependent cdfs in an $L^2$-sense \citep{vDCD18}. 
However, $w(\bar{\mathscr{S}}_T(r))$ could suppress local abnormalities,  leading to loss of power. To avoid this, we push towards uniform weights  using the following criterion.  Define the increments of the local ecdfs, $\{I(k/T,r): k, r\}$, given by
\[
I(k/T,r)= (T/k)^{3/2}\big({\mathscr{S}}_T(k/T,r)-{\mathscr{S}}_T(k/T,r-1/T)\big).
\]
Time-localized abnormalities should cause a (temporary) change in the correlation of overlapping increments in time direction. More specifically,   
 \begin{align*}
\Cov(I(u,r),I(v,r))&
\neq \Cov(I(u+\delta,r),I(v+\delta,r)), \quad v<u, v,u,v+\delta,u+\delta \in (0,1). \tageq \label{eqref:covin}
 \end{align*}
Thus, as a measure we consider the empirical correlation of the increments between time $u=k/T$ and $v+u=1/T+k/T$; 
\begin{align*}
\hat{\varrho}_v(u) :=\text{\scalebox{0.95}{ $\frac{\frac{1}{T}\sum_{i=1}^{T}\big(I(u,r_i)-\frac{1}{T}\sum_{i=1}I(u,r_i)\big)\big(I(u+v,r_{i})-\frac{1}{T}\sum_{i=1}^T I(u+v,r_{i})\big)}{\sqrt{ \frac{1}{T}\sum_{i=1}^{T} \big(I(u+v,r_i)-\frac{1}{T}\sum_{i=1}^T I(u+v,r_{i})\big)^2}\sqrt{\frac{1}{T}\sum_{i=1}^{T} \big(I(v,r_{i})-\frac{1}{T}\sum_{i=1}^T I(v,r_{i})\big)^2}}$}}~.
\end{align*}
If the standard error $\text{sd}(\hat{\varrho}_{1/T})$ of the sample $\{\hat{\varrho}_{1/T}(k/T)\}_{k=2}^{T-1}$ exceeds a threshold (that converges to 5\% as the sample size increases) we impose uniform weights, i.e., 
\begin{align*}
w_\gamma(r) &= \mathbb{1}_{sd(\hat{\varrho}_{1}(r)) \le 0.05+\frac{1.5}{T}} W^{\ell}(r)+\mathbb{1}_{sd(\hat{\varrho}_{1}(r))>.05+\frac{1.5}{T}}, \quad \ell \in \{\text{max},\text{L}\}, 
\end{align*}
where $W^{\text{max}}(r)=$ {\small ${I_{\bar{\mathscr{S}}_T}(r)}/\max_r I_{\bar{\mathscr{S}}_T}(r)$} is taken for the test in \eqref{eq:Dtmax}, denoted by $\mb{D}^{\max}_{T,\gamma}$, and where the choice $W^{\text{L}}(r)=$ {\small${I^4_{\bar{\mathscr{S}}_T}(r)}/{\sum_r I^4_{\bar{\mathscr{S}}_T(r)}}$} is used for the test \eqref{eq:DtL}, denoted by $\mb{D}^{\max}_{T,\gamma}$. 
We emphasize that we merely push towards uniform weights if the criterion indicates that using a weight functional based on $\bar{\mathscr{S}}_T$ increases the type II error. This does not change the limiting distributional properties of the tests under the stated conditions. Note that if the increments are stationary and independent then \eqref{eqref:covin} is an equality, and that even if these time increments are stationary the process itself might not be. Thus, under certain alternatives full weight is given to $W^\ell(r)$.

\subsection{Simulation study} \label{sec:sim}
We summarize results from a simulation study applied to TDA shape descriptors here. Full details on the study are given in the online supplement. We compared $\mb{D}^{\max}_{T,\gamma}$ and $\mb{D}^{L}_{T,\gamma}$ with  other weight functions applied to $\bar{S}$. 
Specifically,  \eqref{eq:Dtmax} and \eqref{eq:DtL} were considered with $\tilde{w}(r)=\frac{\tilde{w}_2(r)}{\max_r \tilde{w}_2(r)}$ and $\tilde{w}(r)=\frac{\tilde{w}_2(r)}{\sum_r \tilde{w}_2(r)}$, respectively denoted by $\mb{D}^{\max}_{T,\tilde{w}}$ and $\mb{D}^{L}_{T,\tilde{w}}$. Further,  $\mb{D}^{\max}_{T,\text{glob}}$ and $\mb{D}^{L}_{T,\text{glob}}$ are the estimators with $
\ddot{w}^{\text{max}}(r) = \frac{V_T(r)}{\max_{r} V_{T}(r)}$ and $\ddot{w}^{\text{L}}(r_i) =${\small ${ V^L_{T,H_0}(r_i)/\sum_{i}V^L_{T,H_0}(r_i)}$}, respectively.
Finally,  
$\mb{D}^{L}_{T,\text{glob},\gamma}$, 
 has a weight function that is a combination of $\ddot{w}$ and the adaptive weight function, i.e., the weights are given by $w_{\text{glob},\gamma}^L(r)=$ {\small $
{w_\gamma(r_i)V^L_{T,H_0}(r_i)}/{
\sum_{i=1}^T w_\gamma(r_i) V^L_{T,H_0}(r_i)}$}.\\
The outcomes are consistent with the theory developed in Section \ref{sec:wip}.  We found that the data-adaptive versions of the class of max-type tests  
and of the class of quadratic tests outperform their non-adaptive counterparts,  
and are more robust against various alternatives, including potential bimodal distributions with discontinuities. Overall,  $\mb{D}^{\text{L}}_{T,{\text{glob},\gamma}}$ and $\mb{D}^{\text{max}}_{T,{\gamma}}$ provide the best trade-off between power and size, and have good properties in small samples.

\section{Application in the context of topological data analysis (TDA)} \label{sec:5}

In this section, we discuss an application of our techniques to a stable shape descriptor from TDA. We start with a very concise review; for more detailed treatments see e.g.,~\cite{c14, o15, TDAbook}.

\subsection{A rapid review of TDA} \label{sec:revTDA}
 Computing persistent homology involves the construction of a sequence
 of simplicial complexes from a finite metric space
 $(X,\partial_X)$. A simplicial complex is a combinatorial space
 constructed by ``gluing together'' translates of $n$-simplices at
 their boundaries.  To be more precise, an $n$-simplex is the convex
 hull $\sum_{i=0}^n t_i v_i$, where $\{v_i\}$ are a set of points in
 $\rnum^d$ that are affinely independent, i.e., the vectors $\{v_i - v_0\}$
 are linearly independent and $\sum_{i=0}^n t_i = 1$.  A simplicial
 complex is a union of simplices such that the intersection of any
 pair of simplices is itself another simplex, i.e., the simplices meet
 only along their faces.  Associated to a simplicial complex are
 topological invariants, most notably {\em homology}.  The homology
 vector spaces of a simplicial complex $C$ are invariants $H_k(C)$ for
 $k \geq 0$.  The rank of the $k$th homology is a measurement of the
 number of $k$-dimensional holes in $C$. For example, for a circle,
 $H_1(\mathbb{S}^1)$ is a rank one vector space. For a torus,
 $H_1(\mathbb{T})$ has rank two and $H_2(\mathbb{T})$ has rank one,
 corresponding to the two one-dimensional loops and the
 two-dimensional hole in the middle, respectively.  

For an arbitrary metric space $(X,\partial_X)$, the most widely used simplicial complex is the Vietoris-Rips complex $\VR_{\epsilon}(X)$ (see e.g.,~\cite[2.1.6]{TDAbook}). For a given scale parameter $\epsilon > 0$, $\VR_{\epsilon}(X)$ is defined to have vertices the points of $X$ and a $n$-simplex spanning points $\{x_0, x_1, \ldots, x_n\}$ if $\partial_X(x_i,x_j) \leq \epsilon, 0 \leq i,j \leq n.$

The formulation of the Vietoris-Rips complex depends on the choice of a feature scale $\epsilon$. The homology groups $H_k(-)$ of this complex, which roughly speaking count the $k$-dimensional holes, provide qualitative shape descriptors for the underlying data. However, the correct feature scale $\epsilon$ can be difficult to estimate, and moreover analysis at a single scale loses information about underlying objects with features of varying sizes.

The basic idea of persistent homology is then to combine information from all feature scales $\epsilon$ at once, by using that the Vietoris-Rips complex is {\em functorial} in the sense  that for $\epsilon < \epsilon'$, there is an inclusion map of simplicial complexes $\VR_{\epsilon}(X) \to \VR_{\epsilon^\prime}(X)$. Thus, letting $\epsilon$ vary over a range $a \leq \epsilon \le b$, we obtain a {\em filtered complex}. We will denote the \textit{Vietoris-Rips  filtration} by  $\text{Rips}(X):=\VR_{\epsilon \in \rnum_+}(X)$.  For finite metric spaces, there are only finitely many values of $\epsilon$ for which the Vietoris-Rips complex changes, and so we only need keep track of complexes before and after these critical points and the maps between them.  That is, we have a system
$\VR_{\epsilon_0}(X) \to \VR_{\epsilon_1}(X) \to \ldots \to \VR_{\epsilon_n}(X)$
for $\epsilon_0 < \epsilon_1 < \ldots < \epsilon_n$.  The $H_k(-)$ are themselves functorial: for each $k$ there is an associated {\em filtered vector space}
\[
H_k(\VR_{\epsilon_0}(X)) \to H_k(\VR_{\epsilon_1}(X)) \to \ldots \to H_k(\VR_{\epsilon_n}(X)).
\]
\begin{figure}[t]
\centering
\begin{subfigure}[t]{0.25\textwidth}
\centering
\includegraphics[height=1.35in]{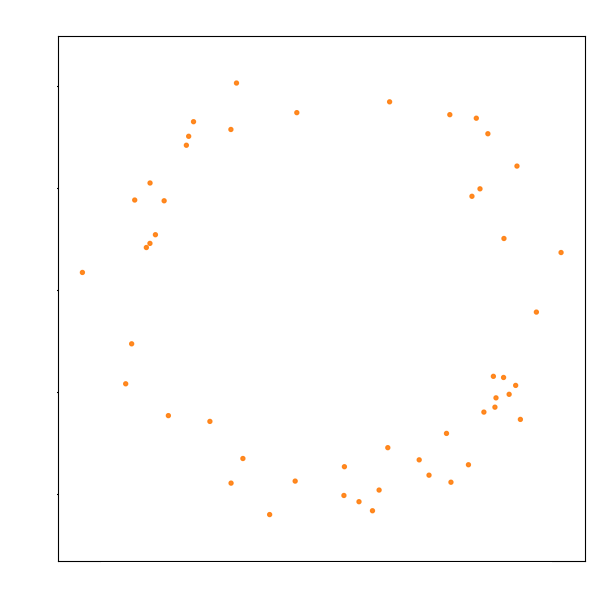}
\caption{$\epsilon = 0$}
\end{subfigure}%
\begin{subfigure}[t]{0.25\textwidth}
\centering
\includegraphics[height=1.35in]{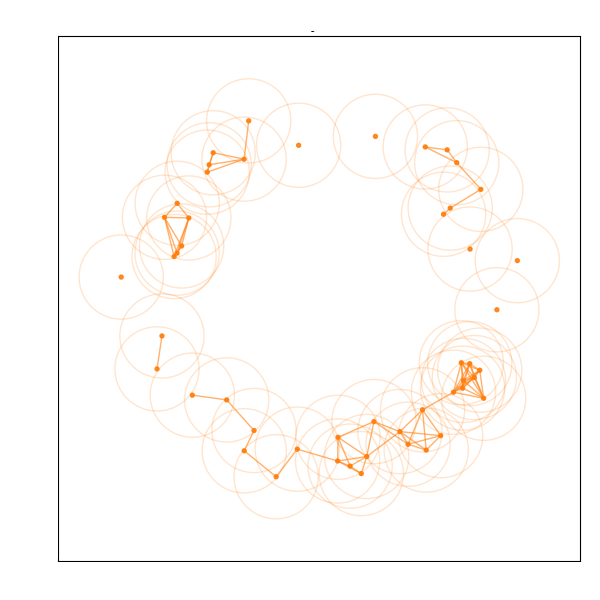}
\caption{$\epsilon = 0.25$}
\end{subfigure}%
\begin{subfigure}[t]{0.25\textwidth}
\centering
\includegraphics[height=1.35in]{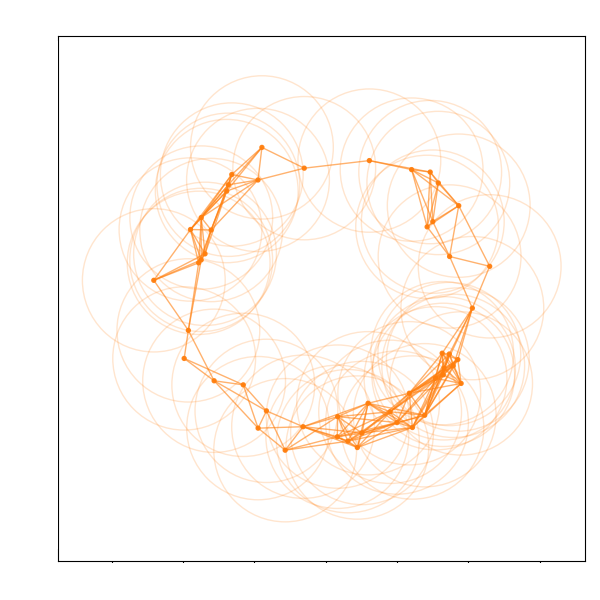}
\caption{$\epsilon = 0.5$}
\end{subfigure}%
\begin{subfigure}[t]{0.25\textwidth}
\centering
\includegraphics[height=1.35in]{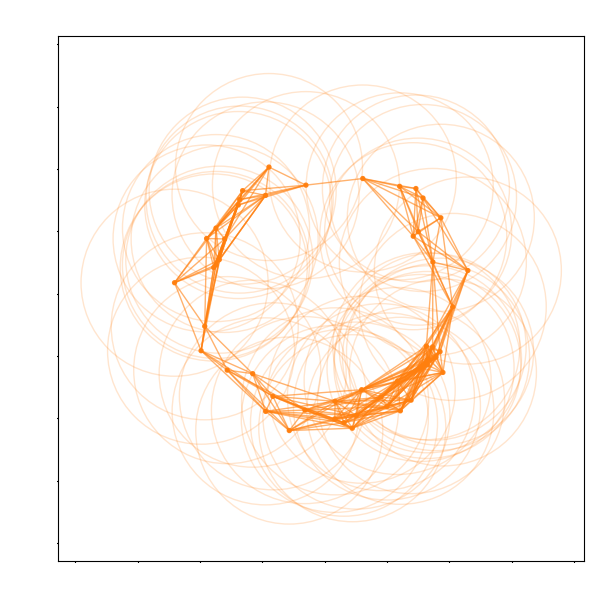}
\caption{$\epsilon = 0.75$}
\end{subfigure}%

\begin{subfigure}[t]{0.25\textwidth}
\centering
\includegraphics[height=1.35in]{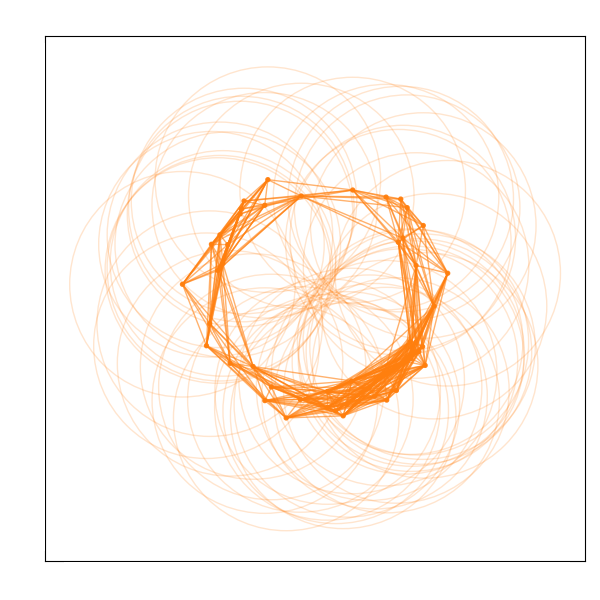}
\caption{$\epsilon = 1$}
\end{subfigure}%
\begin{subfigure}[t]{0.25\textwidth}
\centering
\includegraphics[height=1.35in]{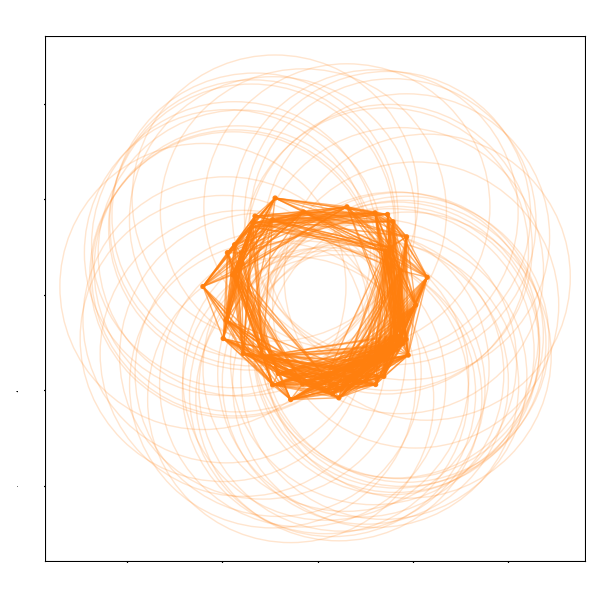}
\caption{$\epsilon = 1.5$}
\end{subfigure}%
\begin{subfigure}[t]{0.25\textwidth}
\centering
\includegraphics[height=1.35in]{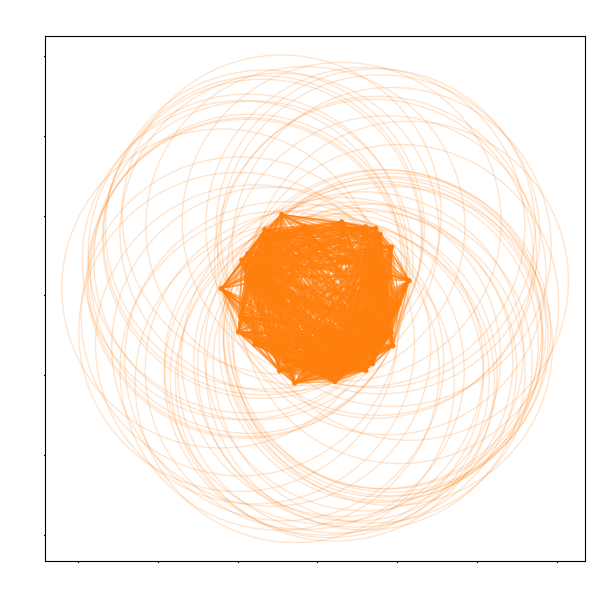}
\caption{$\epsilon = 2$}
\end{subfigure}%
\begin{subfigure}[t]{0.23\textwidth}
\centering
\includegraphics[height=1.35in]{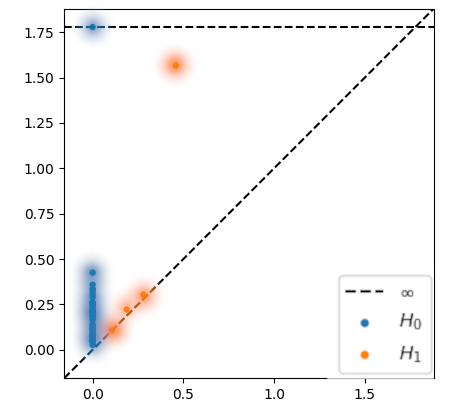}
\caption{Persistence.}
\end{subfigure}%
\caption{The Vietoris-Rips complexes and persistence diagram from a noisy circle.}
\label{fig1}
\end{figure}
\begin{figure}[t]
        \centering
 \includegraphics
 [width=1.01\textwidth]{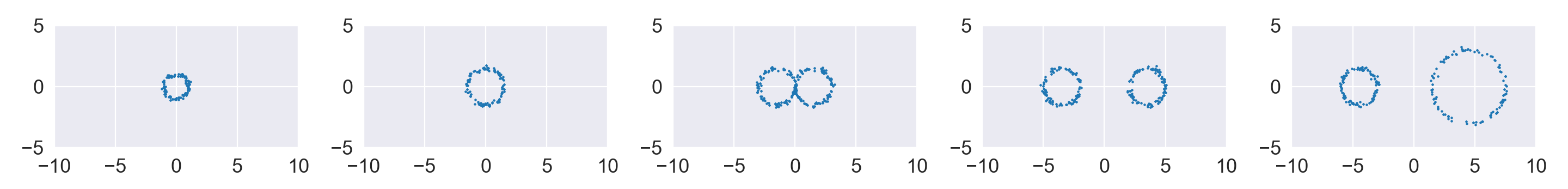}
    \caption*{Samples from slices at fixed times from an evolving surface as time increases.}
    \centering
\includegraphics[width=1.0\textwidth]{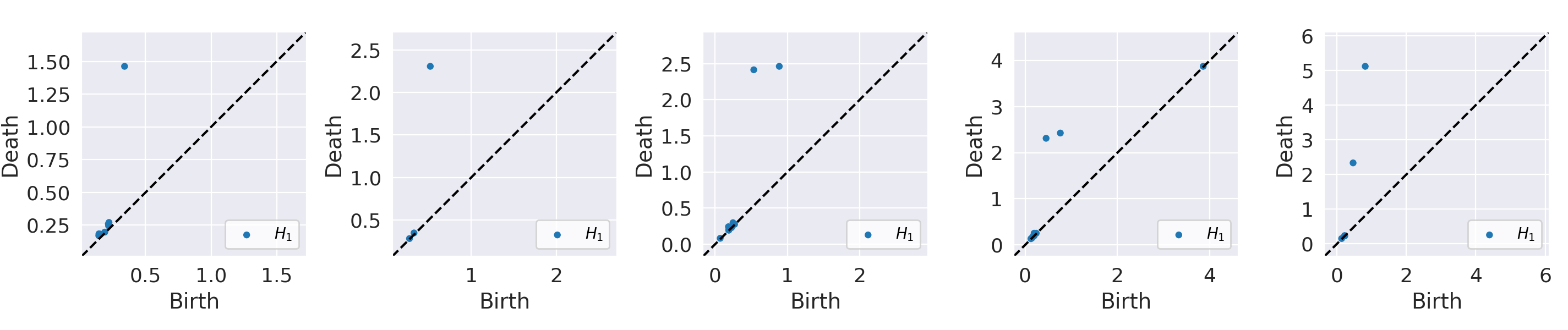}
    \caption{Persistence diagrams;  the points away from the line $x=y$ represent the circles.}
\label{fig2}
\end{figure}
A key theorem in the subject shows that these filtered vector spaces can be represented as multisets of intervals $(a,b)$, referred to as a {\em barcode} or {\em persistence diagram} and denoted by $\PH_k$.  The intervals in the barcode describe the {\em lifespan} of a generator, i.e., at what value of $\epsilon$ the generator arises and at what value of $\epsilon$ it disappears -- these values are the endpoints of the bars in the barcode.  The persistence diagram plots these pairs of endpoints as points in $\mathbb{R}^2$, where the $x$-coordinate is the birth time and the $y$-coordinate is the death (or vanishing) time; everything lies over the line $y=x$. For instance, consider the example of samples from circles as in Figure~\ref{fig1} (see also Figure \ref{fig2}).  When $\epsilon$ is smaller than the interpoint distances as in figure (a), the Vietoris-Rips complex consists only of the points themselves and there is topological signal in dimension $0$ but none in dimension $1$.  As $\epsilon$ increases, adjacent points are connected, forming clusters.  This gives rise to merging of the bars in dimension $0$.  This can be seen in figure (h) with the blue points close to the line $y=x$ resolving into one that is far from the line.  When $\epsilon$ is large enough, there exists a path around the circle, forming a generator for a vector space in dimension $1$.  This class exists until $\epsilon$ is so large that higher-dimensional simplices arise to fill in the circle; then it vanishes (and all of the points merge into one large cluster).  This is represented by the orange point far from the line $y=x$.
The completion of the set of barcodes is a Polish space~\cite{bgmp14, mmh11}, and
$B$ forms a metric space under various metrics, notably the {\em bottleneck distance}, $d_B$. Intuitively, $d_B$ measures distance between two barcodes as a minimum over all partial matchings. Given two barcodes $\beta$ and $\beta^\prime$, the bottleneck distance $d_B$ is given by 
\[
d_B(\beta,\beta^\prime) =\inf_{\varphi \in \Phi} \sup_{x \in \beta} \|x-\varphi(x)\|_{\infty}~,
\]
where $\Phi$ is the set of all bijections between $\beta$ and $\beta^\prime$. 
The {\em stability theorem}, due originally to Cohen-Steiner, Edelsbrunner, and Harer~\cite{ceh07}, is the cornerstone for all work on (statistical) properties of persistent homology. One version of the theorem states that for compact metric spaces $(X,\partial_X)$ and $(Y,\partial_Y)$, there is a strict bound
\[
d_{B}(\PH_k(\textrm{Rips}(X)), \PH_k(\textrm{Rips}(Y)) \leq d_{GH}(X, Y). \tageq \label{thm:stab}
\]
Thus, the space of barcodes provides a tractable metric space for encoding qualitative shape information about the point clouds. In our terminology, it means that persistent homology is a stable shape descriptor with respect to $d_B$, i.e., our results apply with $\sh(\cdot)=\PH_k(\textrm{Rips})(\cdot)$.  Hence, we can analyze the topological features via the corresponding barcode-valued process $\big(\PH_k(\text{Rips}(\tilde{\mb{X}}^n_{t,T}))\big)_{t \in [T]}$.  We remark that $B$ can also be equipped with a Wasserstein metric, which essentially changes the cost of the matching. Our results apply identically to this case.

\subsection{Developmental trajectories in single-cell sequencing data}\label{sec:realdata}
We now apply our methodology to test for
topological change in time series scRNA expression data. The data was obtained from the Waddington-OT work on cell
developmental trajectory inference~\citep{sstc19}.  The data
comprised expression profiles from roughly 250K cells in mouse embryo
fibroblasts undergoing differentiation.  The measurements were taken at 37 time points across 18 days of development.  Explicitly, at each
time point the data was comprised of a point cloud in $\rnum^{20000}$,
where each point was an scRNA expression vector, i.e., a vector of mRNA expression levels for a
particular cell, with each entry in the vector corresponding to a given gene.  Thus, we obtained a time series of 37 point clouds that were
subsets of $\rnum^{20000}$.  We then log-transformed the data (as is standard) to account for variation in absolute
levels of expression across different genes, and we then selected the
2000 most variable genes (using CellRank~\cite{cellrank}) for
analysis.  Figure \ref{fig:trajectory} indicates the shape of the data, projected into $\mathbb{R}^2$ and laid out as a force-directed graph.  The graph indicates that the underlying shape of the data and the cell type changes over time.

 \begin{figure}[t]
        \centering
\includegraphics[scale=0.35]{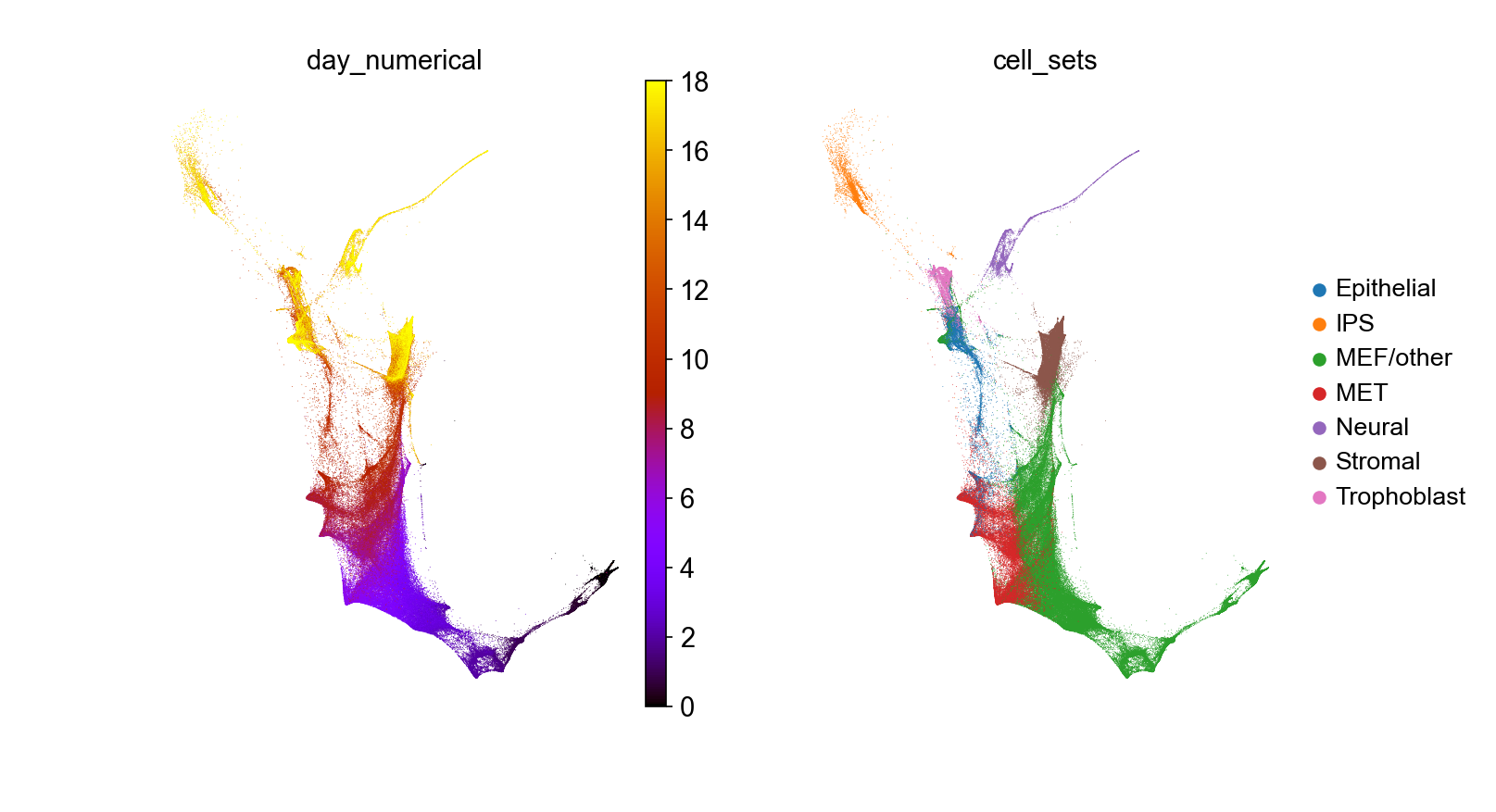}
    \caption{Mouse embryo fibroblast cell trajectories under differentation.}
\label{fig:trajectory}
\end{figure}

 We tested whether the nonstationarity in $H_0$ and $H_1$ is statistically significant using the adaptive test statistics. Despite the small sample size of $T=37$, we are able to reject the hypothesis of a stationary distribution at the 5\% level using $\mb{D}^{\text{max}}_{T,{\gamma}}$ for homology 1 and at the 10\% level for homology 0 ($p$-values of 0.04 and 0.08, resp.). The quadratic test $\mb{D}^{\text{L}}_{T,{\text{glob},\gamma}}$ rejects the null at the 10\% level for $H_0$, but fails to reject the null at the 10\% level for $H_1$ ($p$-values of 0.033 and 0.147, resp.).
 
\FloatBarrier
\section*{Acknowledgements}
 Both authors were partially supported by NSF grant DMS-2311338.
A.J. Blumberg was partially supported by ONR grant N00014-22-1-2679.  
The authors sincerely thank the Associate Editor and two anonymous referees for their constructive comments that helped to produce an improved version of the original paper.

\newpage

\begin{supplement}
\stitle{Supplement to ``A statistical framework for analyzing shape in a time series of random geometric objects''}
\sdescription{Appendix A provides a Monte Carlo study of the finite sample properties of the statistical tests introduced in \autoref{sec4}. Appendix B-E   provide the proofs of the
statements of the main paper as well as additional auxiliary results needed to complete the proofs.}
\end{supplement}
\appendix

\section{Simulations}\label{sec:simulations}

We illustrate the finite sample performance of the  self-normalized tests via a Monte Carlo study. The quantiles of equations~\eqref{eq:Dtmax} and~\eqref{eq:DtL} are obtained by simulating 200,000 times on a grid of Brownian Motion equal to the sample size $T$, reflecting the  grid resolution.  The Monte Carlo simulation consists of 500 repetitions in each case. In the results provided below we fixed $n=300$, which is quite conservative; most data sets will have much larger point clouds, and in these cases convergence will be faster. However, as noted above, larger choices of $n$ are computationally prohibitive with currently available algorithms. 
Furthermore, in computing the test statistics, values of the equidistant radius grid $\{r_i\}$ that had frequency zero were ignored. 

\begin{itemize}[leftmargin=*]
\item {\bf Model I:} In the first model we consider data arising as a sequence of random circles, where the spatial observations are generated from a uniform distribution. More specifically, 
\[
\tilde{\mb{X}}^n_{t,T} =\tau_t \Big\{ \big(\cos(\vartheta_m), \sin(\vartheta_m)\big)\Big\}_{m=1}^n
\]
where $\{\vartheta_m\} \sim \text{Unif}(0,2\pi)$ and $\tau_t=\frac{1}{2}$. 
\item {\bf Model II:} In the second model, we add dependence via an AR structure via $\tau_t$, that is, we let 
$\tau_t = 0.5 \tau_{t-1} + \epsilon_t, \epsilon_t \sim \mathcal{N}(\frac{1}{2},0.05\frac{1}{2}).
$
The radius thus changes according to an AR process with strength parameter with absolute modulus less than 1, ensuring stationarity of this process. 
\item {\bf Model III:}
In the third model, we let $
\tilde{\mb{X}}^n_{t,T} = \Big( \tilde{\mb{Y}}^n_{t,T} ,\tilde{\mb{Z}}^n_{t,T}  \Big)$ 
where 
\[
\tilde{\mb{Y}}^n_{t,T} =\Big\{ \big(\rho_{t,T}\cos(\vartheta_m), \rho_{t,T}\sin(\vartheta_m)\big)\Big\}_{m=1}^n
 \,,\, 
\tilde{\mb{Z}}^n_{t,T} = \Big\{ \big(\cos(\vartheta_\ell)+\rho_{t,T}, \sin(\vartheta_\ell) +\rho_{t,T}\big)\Big\}_{\ell=m+1}^{2n}
\]
with time-varying process $\{\rho_t\}$ given by 
$\rho_t = a(t/T) \rho_{t-1} +\epsilon_t$, $\epsilon_t \sim \mathcal{N}(\frac{1}{5},0.01\frac{1}{5})$ and $a(t/T) = \frac{1}{2}\cos(\frac{\pi}{2} \frac{t}{T})$. Thus, we have two circles that are dependent, where the radii and location change very gradually. This creates the behavior of holes that appear and disappear in random fashion across time. 
\item {\bf Model IV:} In this case we consider  $
\tilde{\mb{X}}^n_{t,T} = \Big( \tilde{\mb{Y}}^n_{t,T} ,\tilde{\mb{Z}}^n_{t,T}  \Big)$
where 
\[
\tilde{\mb{Y}}^n_{t,T} =\Big\{ \big(\rho_{1,T}+\cos(\vartheta_m), \rho_{1,T}+\sin(\vartheta_m)\big)\Big\}_{m=1}^n
\,,\, 
\tilde{\mb{Z}}^n_{t,T} = \Big\{ \big(\cos(\vartheta_\ell)+\rho_{t,T}, \sin(\vartheta_\ell) +\rho_{t,T}\big)\Big\}_{\ell=m+1}^{2n}
\]
with time-varying process $\{\rho_t\}$ is again a time-varying AR but now with
$a(t/T) =\frac{1}{2}\cos(2\pi \frac{t}{T})$ and $\{\epsilon_t\}$ as in model I. In this scenario, the effect is that we create two circles that collide at iteration 1 and then move away from one another in a specific way. This creates nonstationary behavior that affects higher-dimensional holes and specifically homology 1. The effect on homology 0 is that a certain set of connected components might overlap for a brief period, creating minor nonstationary behavior, but will otherwise not be distinguishable from noise. We thus expect that this is difficult to detect in small sample sizes. 
\item {\bf Model V:} Finally, we consider a break structure by taking the model as in Model I but where not just the grid is random but where now $
\tau_t \sim \mathcal{N}(\frac{1}{2},0.05\frac{1}{2})$ for $t \le \frac{1}{3}T$ and for $t \ge \frac{1}{3}T$, and the mean of $\tau_t$ is decreased to $\frac{1}{4}$.
\end{itemize}

The results are given in \autoref{tab:sim1} for homology 1 and for homology 0 in \autoref{tab:sim0}. 
\linespread{1.0}{
\begin{table*}[b]
\begin{center}
\caption{\label{tab:sim1} \it Empirical rejection probabilities of the test \eqref{eq:Dtmax}  for the hypotheses in \eqref{eq:hyp}
 under the null hypothesis for homology 1.}
\label{Sim_size}
\begin{tabular}{c|c|cccccc|cccccccc}
\hline
& &  \multicolumn{2}{c} {$\mathbb{D}^{max}_{T,\text{glob}}$} & \multicolumn{2}{c} {$\mathbb{D}^{max}_{T,\gamma}$} & \multicolumn{2}{c} {$\mathbb{D}^{max}_{T,\tilde{w}}$}& \multicolumn{2}{c} {$\mathbb{D}^{L}_{T,\text{glob}}$} &\multicolumn{2}{c} {$\mathbb{D}^{L}_{T,\tilde{w}}$} & \multicolumn{2}{c}  {$\mathbb{D}^{L}_{T,\text{glob},\gamma}$} & \multicolumn{2}{c} {$\mathbb{D}^{L}_{T,\gamma}$}\\
\cmidrule(lr){3-4} \cmidrule(lr){5-6} \cmidrule(lr){7-8}  \cmidrule(lr){9-10} \cmidrule(lr){11-12} \cmidrule(lr){13-14} \cmidrule(lr){15-16}
 &T & 10\% & 5\% & 10\% & 5\%  & 10\% & 5\% & 
 10\% & 5\%  & 10\% & 5\%& 10\% & 5\% & 10\% & 5\% \\
 \hline
 I&50&
 11.0 & 3.8 & 7.4 & 2.4 & 10.2 & 4.6 & 8.6 & 4.2 & 9.2 & 4.0 & 9.0 & 4.2 & 7.8 & 2.8 \\
& 100& 9.6 & 3.6 & 9.2 & 2.4 & 10.4 & 4.4 & 8.8 & 2.8 & 7.4 & 2.2 & 6.8 & 3.2 & 6.2 & 1.2 \\ 
 & 200&  13.6 & 6.0 & 10.6 & 4.4 & 11.6 & 6.2 & 10.8 & 4.6 & 10.8 & 5.6 & 12.0 & 6.2 & 10.0 & 4.0 \\ 
 & 300& 10.2 & 5.6 & 10.4 & 4.4 & 10.4 & 4.6 & 8.0 & 2.8 & 6.8 & 3.4 & 8.4 & 4.6 & 8.0 & 2.6 \\ \hline 
II& 50& 10.2 & 4.6 & 10.6 & 6.8 & 10.6 & 5.2 & 8.0 & 4.6 & 9.2 & 4.6 & 10.0 & 3.8 & 8.0 & 2.8 \\  
& 100 & 9.8 & 4.2 & 10.6 & 4.8 & 10.2 & 4.2 & 9.2 & 4.0 & 8.2 & 3.8 & 8.6 & 3.4 & 7.4 & 3.0 \\
& 200& 14.0 & 7.8 & 11.4 & 6.8 & 13.6 & 7.0 & 11.2 & 6.0 & 11.0 & 4.8 & 10.8 & 6.2 & 9.8 & 4.4 \\
& 300 &10.4 & 5.2 & 9.8 & 6.0 & 11.0 & 5.4 & 11.4 & 6.0 & 11.8 & 5.0 & 12.2 & 6.2 & 10.8 & 6.0 
\\ \hline
III & 50  & 100.0 & 99.6 & 100.0 & 100.0 & 100.0 & 99.0 & 100.0 & 100.0 & 100.0 & 100.0 & 100.0 & 100.0 & 100.0 & 100.0 \\
&100 & 100.0 & 100.0 & 100.0 & 100.0 & 100.0 & 100.0 & 100.0 & 100.0 & 100.0 & 100.0 & 100.0 & 100.0 & 100.0 & 100.0 \\ 
&200&  100.0 & 100.0 & 100.0 & 100.0 & 100.0 & 100.0 & 100.0 & 100.0 & 100.0 & 100.0 & 100.0 & 100.0 & 100.0 & 100.0 \\
& 300 & 100.0 & 100.0 & 100.0 & 100.0 & 100.0 & 100.0 & 100.0 & 100.0 & 100.0 & 100.0 & 100.0 & 100.0 & 100.0 & 100.0 \\  \hline
IV & 50 & 70.2 & 46.8 & 100.0 & 99.8 & 55.2 & 12.8 & 98.6 & 67.8 & 98.4 & 75.2 & 98.6 & 67.8 & 98.2 & 48.0 \\
&100&   98.0 & 75.8 & 100.0 & 100.0 & 96.4 & 35.2 & 100.0 & 100.0 & 100.0 & 99.8 & 100.0 & 100.0 & 100.0 & 99.6 \\ 
&200 & 100.0 & 99.0 & 100.0 & 100.0 & 100.0 & 85.2 & 100.0 & 100.0 & 100.0 & 100.0 & 100.0 & 100.0 & 100.0 & 100.0 \\ 
& 300  & 100.0 & 100.0 & 100.0 & 100.0 & 100.0 & 94.6 & 100.0 & 100.0 & 100.0 & 100.0 & 100.0 & 100.0 & 100.0 & 100.0 \\ 
\hline
V & 50 & 0.4 & 0.0 & 54.6 & 34.0 & 10.4 & 3.2 & 82.2 & 23.6 & 99.8 & 92.4 & 97.6 & 88.2 & 96.6 & 86.4 \\
& 100 & 0.0 & 0.0 & 84.0 & 65.0 & 11.2 & 2.0 & 78.6 & 5.2 & 97.2 & 40.2 & 99.6 & 96.8 & 99.2 & 96.2 \\ 
& 200 &  0.0 & 0.0 & 98.2 & 88.0 & 30.8 & 3.0 & 99.8 & 22.2 & 100.0 & 65.2 & 100.0 & 100.0 & 100.0 & 100.0 \\ 
& 300  & 0.0 & 0.0 & 100.0 & 96.8 & 40.8 & 4.2 & 96.4 & 4.0 & 100.0 & 25.4 & 100.0 & 100.0 & 100.0 & 100.0 \\ 
\end{tabular}
\end{center}
\end{table*}}
 From models I and II, it can be observed that all tests have decent size, though the quadratic tests are a bit undersized in certain cases. Note that the grid of the quantiles are themselves sample size-dependent and the numerical fluctuation across rows is therefore natural for the available sample sizes. The quadratic tests $\mb{D}^L_{T,\text{glob}}$ and $\mb{D}^L_{T,\gamma}$ appear to have tendency to undersize, especially for homology 0. This is not surprising since persistent $H_0$ can be very sensitive to noise; an additional point far from the remaining points can add a long bar in the barcode. $\mb{D}^L_{T,\text{glob},\gamma}$  appears overall best from the quadratic tests. We note that the undersizing of the quadratic tests will be negligible by simply taking a grid resolution of $T+100$ for the computation of the limiting quantiles.

We now look at the power properties.  For homology 1, we see very fast convergence to full power for model III and IV across the test statistics. However, we observe a clearer benefit in the break scenario of the adaptive weighting test $\mb{D}^{\text{max}}_{T,{\gamma}}$  compared to the other max-type tests. Visual inspection indicates that the distance matrices of  persistent homology 1 have a mixed distribution, which technically violates our assumptions.  Consequently, the range estimator $V(r)$ seems to overestimate, leading to a loss of power of $\mb{D}^{\text{max}}_{T,{glob}}$ and  $\mb{D}^{\text{max}}_{T,{w_1}}$. The quadratic data-adaptive tests $\mb{D}^{\text{L}}_{T,{\gamma}}$ and $\mb{D}^{\text{L}}_{T,{\text{glob},\gamma}}$ have fast and steady convergence to full power. On the other hand, the two non-adaptive quadratic versions $\mb{D}^{\text{L}}_{T,{glob}}$ and $\mb{D}^{\text{L}}_{T,{w_1}}$ suffer from loss of power as $T$ increases, indicating lack of robustness to discontinuities of the non-adaptive versions. Similar behavior is visible for homology 0 for this model. Furthermore, for homology 0, the nonstationary behavior present in model III is well-captured. For model IV, we again see superiority of the data-adaptive max statistic compared to the non-adaptive versions. Note that the nonstationary behavior in model IV for homology 0 is a type of nonstationarity that is rather difficult to detect; the nonstationarity is only detectable around the left extreme $r \to 0$, meaning that it is easily missed by any type of non-adaptive weighing scheme that suppresses the tails. Indeed, observe that for the max-type tests only the data-adaptive scheme is capable of detecting nonstationarity in homology 0.

\linespread{1.0}{
\begin{table*}[t]
\begin{center}
\caption{\label{tab:sim0} \it Empirical rejection probabilities of the test \eqref{eq:DtL}  for the hypotheses in \eqref{eq:hyp}
 under the null hypothesis for homology 0.}
\label{Sim_sizehom0}

\begin{tabular}{c|c|cccccc|cccccccc}
\hline
& &  \multicolumn{2}{c} {$\mathbb{D}^{max}_{T,\text{glob}}$} & \multicolumn{2}{c} {$\mathbb{D}^{max}_{T,\gamma}$} & \multicolumn{2}{c} {$\mathbb{D}^{max}_{T,\tilde{w}}$}& \multicolumn{2}{c} {$\mathbb{D}^{L}_{T,\text{glob}}$} &\multicolumn{2}{c} {$\mathbb{D}^{L}_{T,\tilde{w}}$} & \multicolumn{2}{c}  {$\mathbb{D}^{L}_{T,\text{glob},\gamma}$} & \multicolumn{2}{c} {$\mathbb{D}^{L}_{T,\gamma}$}\\
\cmidrule(lr){3-4} \cmidrule(lr){5-6} \cmidrule(lr){7-8}  \cmidrule(lr){9-10} \cmidrule(lr){11-12} \cmidrule(lr){13-14} \cmidrule(lr){15-16}
 &T & 10\% & 5\% & 10\% & 5\%  & 10\% & 5\% & 
 10\% & 5\%  & 10\% & 5\%& 10\% & 5\% & 10\% & 5\% \\
 \hline
 I & 50&
 9.4 & 5.8 & 9.8 & 4.2 & 10.2 & 5.2 & 7.4 & 3.6 & 6.8 & 2.8 & 6.4 & 3.6 & 5.8 & 3.2 \\ 
 &100 & 8.6 & 3.4 & 8.0 & 4.8 & 10.6 & 3.0 & 5.6 & 2.8 & 4.8 & 2.4 & 6.8 & 2.8 & 6.4 & 2.4 \\ 
  & 200 &  11.8 & 5.8 & 13.0 & 6.8 & 12.8 & 6.0 & 6.8 & 2.6 & 6.2 & 2.2 & 7.2 & 3.0 & 6.6 & 2.8 \\
 & 300&  9.2 & 3.4 & 9.0 & 4.4 & 10.2 & 4.8 & 6.0 & 2.2 & 6.4 & 2.0 & 8.4 & 3.2 & 8.0 & 2.6 \\ \hline
II& 50 & 9.4 & 4.8 & 9.6 & 3.2 & 9.2 & 5.0 & 8.4 & 4.6 & 7.6 & 4.0 & 7.4 & 4.2 & 7.2 & 4.0 \\ 
&100&  12.4 & 5.6 & 11.0 & 6.4 & 12.4 & 7.2 & 8.2 & 2.6 & 7.8 & 2.8 & 9.2 & 3.6 & 8.4 &3.6\\ &200 &10.6 & 3.6 & 8.0 & 2.4 & 10.6 & 3.2 & 7.6 & 3.0 & 7.0 & 3.2 & 7.6 & 3.8 & 6.8 & 2.2 \\ 
&300 & 9.0 & 4.6 & 9.2 & 3.6 & 9.6 & 3.8 & 6.8 & 3.0 & 6.6 & 3.2 & 7.2 & 4.0 & 7.0 & 4.2 \\ \hline
III& 50& 100.0 & 100.0 & 100.0 & 100.0 & 100.0 & 100.0 & 100.0 & 6.4 & 100.0 & 100.0 & 100.0 & 6.4 & 100.0 & 96.0\\
&100  & 100.0 & 100.0 & 100.0 & 100.0 & 100.0 & 100.0 & 100.0 & 100.0 & 100.0 & 100.0 & 100.0 & 100.0 & 100.0 & 100.0 \\ 
& 200 & 100.0 & 100.0 & 100.0 & 100.0 & 100.0 & 100.0 & 100.0 & 100.0 & 100.0 & 100.0 & 100.0 & 100.0 & 100.0 & 100.0  \\
& 300 & 100.0 & 100.0 & 100.0 & 100.0 & 100.0 & 100.0 & 100.0 & 100.0 & 100.0 & 100.0 & 100.0 & 100.0 & 100.0 & 100.0 \\ \hline
IV & 50 &1.2 & 0.2 & 23.2 & 12.6 & 1.0 & 0.2 & 32.2 & 18.0 & 36.6 & 21.6 & 31.6 & 16.6 & 14.4 & 5.2 \\
& 100 &0.6 & 0.0 & 48.8 & 31.8 & 0.8 & 0.0 & 51.0 & 31.2 & 53.0 & 33.0 & 50.0 & 30.6 & 23.0 & 8.6 \\ 
& 200  & 0.2 & 0.0 & 76.4 & 69.2 & 0.8 & 0.0 & 81.2 & 58.8 & 82.0 & 59.6 & 80.6 & 58.8 & 48.0 & 19.0 \\ 
&300  & 1.0 & 0.0 & 88.2 & 82.6 & 1.4 & 0.2 & 89.4 & 76.0 & 89.8 & 75.8 & 89.4 & 75.8 & 62.4 & 24.6 \\ 
\hline
V & 50& 9.4 & 5.4 & 58.0 & 39.2 & 24.2 & 4.0 & 99.4 & 94.8 & 99.0 & 95.4 & 98.8 & 92.8 & 96.6 & 84.0 \\ 
& 100 &  3.2 & 1.0 & 88.4 & 80.2 & 33.8 & 5.4 & 99.8 & 97.8 & 99.8 & 97.8 & 99.8 & 98.0 & 99.2 & 91.0 \\ 
&200 &0.0 & 0.0 & 98.8 & 93.8 & 50.8 & 9.6 & 100.0 & 100.0 & 100.0 & 100.0 & 100.0 & 100.0 & 100.0 & 100.0 \\ 
& 300 & 0.0 & 0.0 & 99.8 & 98.4 & 61.2 & 12.2 & 100.0 & 100.0 & 100.0 & 100.0 & 100.0 & 100.0 & 100.0 & 100.0 \\ 
\end{tabular}
\end{center}
\end{table*}}

\FloatBarrier

\section{Proof of Theorem \ref{thm:limfixr}}

\begin{proof}[Proof of \autoref{thm:limfixr}]
The main line of argument is given here and  proofs of the underlying auxiliary statements are deferred to \autoref{sec:appc}. We use an argument based on iterated limits to determine the limiting law \citep[see e.g.,][Theorem 4.2]{bil68}. More specifically, let
\begin{align*}
M_{m,T}(u,r):=2 \frac{1}{T}
\sum_{t=1}^{\flo{uT}}\sum_{j=0}^{m-1} P_{t}\Big(Y_{m,T,t+j}\big(\frac{t+j}{T},u, r\big) \Big)
\end{align*}
with
\[Y_{m,T,t}(u,v,r) =\E\Big[\frac{1}{T}\sum_{s=1}^{\flo{vT}} \E_{\tilde{\mb{X}}^\prime_{m,0}(\frac{s}{T})}\Big[\hti(\tilde{\mb{X}}_{m,t}(u),\tilde{\mb{X}}^{\prime}_{m,0}(\frac{s}{T}), r)\Big]\Big\vert\G_{t,m}\Big]. \tageq \label{eq:Ymt}
\]
A Hoeffding-type of decomposition of the process then yields the following.
\begin{thm} \label{thm:limsupm}
Under the conditions of \autoref{thm:limfixr}, for each $\epsilon>0$, and $r \in [0,\mathscr{R}]$, 
\begin{align*}
\lim_{m \to \infty} \limsup_{T \to \infty}\pr\Big(T^{1/2} \sup_{u} |\mathscr{S}_T(u,r) -\E\mathscr{S}_T(u,r)-M_{m,T}(u,r)|\ge \epsilon \Big) =0.
\end{align*}
\end{thm}
Secondly, using a martingale functional central limit theorem, we show that
\[\big\{T^{1/2}M_{m,T}(u,r)\big\}_{ u \in [0,1]}\,\underset{T \to \infty}{\medsquiggly} \,
\big\{\mathbb{G}_m(u,r)\big\}_{ u \in [0,1]} \, \underset{m \to \infty}{\medsquiggly} \,
\big\{\mathbb{G}(u,r)\big\}_{ u \in [0,1]}. \]
Let $D_{m,T,t,j}=P_{t}\big(Y_{m,T,t+j}(\frac{t+j}{T},u,r) \big)$. Then it follows in the same way as in the proof of \autoref{lem:shift}, that for fixed $m$, and for each $u \in [0,1]$, and  $\epsilon>0$, the Lindeberg condition holds;
\begin{align*}
&\frac{1}{T}\sum_{t=1}^{\flo{uT}}\E\Big[ \bv \sum_{j=0}^{m-1} D_{m,T,t,j} \bv^2 \mathrm{1}_{\{| \sum_{j=0}^{m-1} D_{m,T,t,j}|>\sqrt{T}\epsilon\}}\Big]  
 \to 0 \quad  T \to \infty.
\end{align*}

The following result on the limit of the  conditional variance and an application of the Cram{\'e}r-Wold device completes the proof. 
\begin{lemma}\label{lem:convar}
Let $k=\flo{uT}$. Then, under the conditions of \autoref{thm:limfixr}
\begin{align*}
&\lim_{m\to \infty}\lim_{T\to \infty}\frac{1}{T}\sum_{t=1}^{k}\E\Big[\bv\sum_{j=0}^{m-1} P_{t}\Big(Y_{m,T,t+j}\big(\frac{t+j}{T},\frac{k}{T},r\big)\Big)\bv^2 \bv \G_{t-1}\Big]
\\&\overset{p}{\to}  \lim_{m \to \infty}
\int_0^u\sum_{\ell \in \znum} \mathrm{Cov}\big(g_m(\tilde{\mb{X}}_{m,0}(\eta),u,r),g_m(\tilde{\mb{X}}_{m,\ell}(\eta),u,r)\big) d\eta 
\\&\overset{p}{\to} \int_0^u\sum_{\ell \in \znum} \mathrm{Cov}\big(g(\tilde{\mb{X}}_{0}(\eta),u,r),g(\tilde{\mb{X}}_{\ell}(\eta),u,r)\big) d\eta,
\end{align*}
where 
\begin{align*}
    g_m(\tilde{\mb{X}}_{m,j}(\eta),u,r)&=\int_0^u \E\big[\E_{\tilde{\mb{X}}^\prime_{0}(v)}[\hti(\tilde{\mb{X}}_{m,j}(\eta), \tilde{\mb{X}}^\prime_{m,0}(v),r)]\bv \G_{j,m}\Big]dv
    \\ 
    g(\tilde{\mb{X}}_{j}(\eta),u,r)&= \int_0^u \E_{\tilde{\mb{X}}^\prime_{0}(v)}[\hti(\tilde{\mb{X}}_{j}(\eta), \tilde{\mb{X}}^\prime_{0}(v),r)]dv.
\end{align*}
\end{lemma}
\end{proof}

\section{Auxiliary statements}
\begin{lemma} \label{lem:maxmom}
Let $\{G_{t}=\sum_{s=1}^{t}G_{t,s}\}$ be zero-mean and $\G_t$-measurable. Then 
\[
\Big\| \max_{1\le  k\le T} \Big\vert\sum_{t,s=1}^{k} G_{t,s}\Big\vert\Big\|_{\rnum,2} \le 4 \sum_{j=0}^{\infty} \Big\|\sum_{t=1}^T P_{t-j}\big(\sum_{s=1}^{t}G_{t,s}\big)\Big\|_{\rnum,2}.
\]
\end{lemma}
\begin{proof}
Since $ \sum_{t=1}^k\sum_{j=0}^{\infty} P_{t-j}(\sum_{s=1}^{t}G_{t,s})= \sum_{t=1}^k\sum_{s=1}^{t}G_{t,s}$ almost surely and in $L^2$, elementary calculations give
\begin{align*}
 \Big \|\max_{1\le  k\le T}\Big\vert  \sum_{t=1}^{k} \Big(\sum_{s=1}^{t}G_{t,s}\Big)\Big)\Big\|_{\rnum,2}
&= \Big\|\max_{1\le  k\le T}\Big\vert \sum_{t=1}^{k} \sum_{j=0}^\infty P_{t-j}\Big(\sum_{s=1}^{t} G_{t,s}\Big)\Big\vert\Big\|_{\rnum,2}
\\& \le 
\sum_{j=0}^\infty\sqrt{\E\bv\max_{1\le  k\le T}\Big\vert \sum_{t=1}^{k}P_{t-j}\Big(\sum_{s=1}^{t}G_{t,s}\Big)\Big\vert^2\bv}
\\& \le 2  \sum_{j=0}^\infty \Big\| \sum_{t=1}^{T}P_{t-j}\Big(\sum_{s=1}^{t} G_{t,s}\Big)\Big\|_{\rnum,2}
\end{align*}
where the second inequality follows from Doob's maximal inequality. The result now follows from symmetry of the kernel and the fact that $\sum_{s,t=1}^k = \sum_{t=1}^{k} \sum_{s=1}^{t-1} +\sum_{s=t} +\sum_{s=1}^{k} \sum_{t=1}^{s-1}$.
\end{proof}
\begin{proof}[Proof of \autoref{sumbound}]
To reduce notational clutter, we drop the dependence on $r$ in the measures throughout the proof. We start with bounding (i) and (ii) using \eqref{thedeltas}. We have,
\begin{align*}
&\sum_{j=0}^{\infty} \sup_t\Big\|P_{t-j}(\hti(X_t,X_{t-k},r))\Big\|_{\rnum,q}
\\&=\sum_{j=0}^{k-1}\sup_t\Big\|P_{t-j}(\hti(X_t,X_{t-k},r))\Big\|_{\rnum,q}+\sum_{j=0}^{\infty}\sup_s\Big\|P_{s-j}(\hti(X_{s+k},X_{s},r))\Big\|_{\rnum,q}.
\end{align*}
Now, $P_{t-j}\big(\hti(X_{m,t}, X_{m,t-k},r)\big) = 0$ a.s. if both $X_{m,t}$ and $X_{m,t-k}$ are independent of $\G_{t-j}$, i.e., $t-k-m+1> t-j \Rightarrow j > k+m-1$ or if $X_{m,t}$ is independent of $\G_{t-j}$ and $X_{m,t-k}$ is $\G_{t-j-1}$-measurable since then $$\E[\hti(X_{m,t},X_{m,t-k},r)|\G_{t-j}]=\E[\hti(X_{m,t},X_{m,t-k},r)|\G_{t-j-1}],$$ which occurs if $t-m+1>t-j$ and $t-k \le t-j-1$, i.e., if $j > m-1$ and $j <k$. Consequently,  
\begin{align*}
& 
\Big\|P_{s-j}(\hti(X_{s+k},X_{s},r))\Big\|_{\rnum,q}=\Big\|P_{s-j}(\hti(X_{s+k},X_{s},r))-P_{s-j}(\hti(X_{j,s+k},X_{j,s},r))\Big\|_{\rnum,q}  \quad  j,k \ge 0,
\\&\Big\|P_{t-j}(\hti(X_t,X_{t-k},r))\Big\|_{\rnum,q}=\Big\|P_{t-j}(\hti(X_t,X_{t-k},r))-P_{t-j}(\hti(X_{j,t},X_{j,t-k},r))\Big\|_{\rnum,q}, \quad  k>j, j \ge 0
\end{align*}
 and thus $\sum_{j=0}^{\infty} \sup_t\Big\|P_{t-j}(\hti(X_t,X_{t-k},r))\Big\|_{\rnum,q}\le 2\sum_{j=0}^{\infty} \delta_{j,q}$.
For (ii)  
\begin{align*}
\sum_{j=0}^\infty \min\Big(\nu_{t,t-k,q}(j), \epsilon\Big) & = \sum_{j=0}^{k-1}  \min\Big(\sup_{m\ge 1}  \Big\|P_{t-j}(\hti(X_{m,t},X_{m,t-k},r))\Big\|_{\rnum,q}  , \epsilon\Big) \\&+\sum_{j=k}^{\infty} \min\Big( \sup_{m\ge 1}\Big\|P_{t-j}(\hti(X_{m,t},X_{m,t-k},r))\Big\|_{\rnum,q}, \epsilon\Big) 
\\& \le \sum_{j=0}^{k-1}\min( \sup_{ m>j} (\delta_{j,q}+\delta_{m,q}), \epsilon)+
\sum_{j=0}^{\infty}\min\Big( \sup_{m> j}\Big\|P_{s-j}(\hti(X_{m,s+k},X_{m,s},r))\Big\|_{\rnum,q}, \epsilon\Big), 
\end{align*}
where we used that a) $P_{t-j}\hti(X_{m,t}, X_{m,t-k}) =0$ a.s. if $j \le k-1$ and $j \ge m$ and that b) $P_{s-j}(\hti(X_{m,s+k},X_{m,s},r)=0$ a.s. if $s-m+1 > s-j$ $\Rightarrow j>m-1$. Finally, since the triangle inequality yields
\begin{align*}
\Big\|P_{s-j}(\hti(X_{m,s+k},X_{m,s},r))\Big\|_{\rnum,q}& \le \Big\|P_{s-j}(\hti(X_{s+k},X_{s},r))-P_{s-j}(\hti(X_{j,s+k},X_{j,s},r))\Big\|_{\rnum,q}
\\&+ \Big\|P_{s-j}(\hti(X_{s+k},X_{s},r))-P_{s-j}(\hti(X_{m,s+k},X_{m,s},r))\Big\|_{\rnum,q} \le \delta_{j,q}+\delta_{m,q}.
\end{align*}
Thus $
\sum_{j=0}^\infty \min\Big(\nu_{t,t-k}(j), \epsilon\Big) \le 4 \sum_{j=0}^{\infty}\min (\sup_{m> j} \delta_m,\epsilon)$.
For the bounds in terms of \eqref{eq:thethetas}, we note that (almost surely)
\[
P_{t-j}(\hti(W_{m,t},W_{m,t-k},r)) =\E\Big[ \hti(W_{m,t},W_{m,t-k},r)-\E\big[\hti(W_{m,t},W_{m,t-k},r) |\G_{t,\{t-j\}}\big] \bv\G_{t-j} \Big].
\] 
 Then
\begin{align*}
&\|P_{t-j}(\hti(W_{m,t},W_{m,t-k},r))\Big\|_{\rnum,q}
\\&\le  \sup_t \Big\| \hti(W_{m,t},W_{m,t-k},r)-\E\Big[ \hti(W_{m,t},W_{m,t-k},r) |\G_{t,\{t-j\}}\Big]\Big\|_{\rnum,q}
\\&
\le \sup_t\Big\| \hti(W_{m,t},W_{m,t-k},r)-\E\Big[ \hti(W_{\{t-j\},m,t},W_{\{t-j\},m,t-k},r) |\G_{t,\{t-j\}}\Big]\Big\|_{\rnum,q}
\\& + \sup_t \Big\| \E\Big[ \hti(W_{m,t},W_{m,t-k},r) |\G_{t,\{t-j\}}\Big]-\E\Big[ \hti(W_{\{t-j\},m,t},W_{\{t-j\},m,t-k},r) |\G_{t,\{t-j\}}\Big]\Big\|_{\rnum,q}
\\& \le 2 \sup_t \Big\|\hti(W_{m,t},W_{m,t-k},r) -\hti(W_{\{t-j\},m,t},W_{\{t-j\},m,t-k},r)\Big\|_{\rnum,q},  \tageq \label{eq:upthetam}
\end{align*}
where we used that $\hti(W_{\{t-j\},m,t},W_{\{t-j\},m,t-k},r) $ is measurable with respect to $\G_{t,\{t-j\}}$ and the contraction property of the conditional expectation. The result is now straightforward.
\end{proof}

\begin{lemma}\label{lem:errorcontrol}
Let $Z=(Z_t: t\in \znum)$ and $Y=(Y_t: t \in \znum)$ by adapted to $(\G_t)$. Then
\begin{align*}&
\Big\| \max_{1\le  k\le T}\frac{1}{T^{3/2}}\Big\vert\sum_{t,s=1}^{k} \hti(Z_t, Z_{s},r) - \hti(Y_t, Y_{s},r)\Big\vert\Big\|_{\rnum,2}
 \\& \le 2 \sup_{k\ge 0} \sum_{j=0}^\infty \sup_t\min\Big(\nu^Z_{t,t-k}(j)+\nu^Y_{t,t-k}(j),\Big\|\hti(Z_t, Z_{t-k},r) - \hti(Y_t, Y_{t-k},r)\Big\|_{\rnum,2}\Big)
\end{align*}

\end{lemma}

\begin{proof}[Proof of \autoref{lem:errorcontrol}]
From \autoref{lem:maxmom} 
\begin{align*}
 &\Big \|\max_{1\le  k\le T}\frac{1}{T^{3/2}}\Big\vert  \sum_{t=1}^{k} \Big(\sum_{s=1}^{t} \hti(X_t, X_{s},r) - \hti(Y_t, Y_{s},r)\Big)\Big\|_{\rnum,2}
\\& \lesssim \frac{1}{T^{3/2}}  \sum_{j=0}^\infty \Big\| \sum_{t=1}^{T}P_{t-j}\Big(\sum_{s=1}^{t} \hti(X_t, X_{s},r) - \hti(Y_t, Y_{s},r)\Big)\Big\|_{\rnum,2}
\end{align*}
Using orthogonality of the projections and a change of variables.
\begin{align*}
 & \frac{1}{T^{3/2}} \sum_{j=0}^\infty \sqrt{\Big\| \sum_{t=1}^{T}P_{t-j}\Big(\sum_{s=1}^{t} \hti(X_t, X_{s},r) - \hti(Y_t, Y_{s},r)\Big)\Big\|^2_{\rnum,2}}
\\  & =\frac{1}{T^{3/2}} \sum_{j=0}^\infty \sqrt{\sum_{t=1}^{T}\Big\|P_{t-j}\Big(\sum_{k=0}^{t-1} \hti(X_t, X_{t-k},r)- \hti(Y_t, Y_{t-k},r)\Big)\Big\|^2_{\rnum,2}}
\\  & =\frac{1}{T^{3/2}} \sum_{j=0}^\infty \sqrt{\sum_{t=1}^{T}\sum_{k=0}^{t-1} \sum_{k^\prime=0}^{t-1} \Big\|P_{t-j}\Big( \hti(X_t, X_{t-k},r)- \hti(Y_t, Y_{t-k},r)\Big)\Big\|_{\rnum,2}\Big\|P_{t-j}\Big(\hti(X_t, X_{t-k^\prime},r)- \hti(Y_t, Y_{t-k^\prime},r)\Big)\Big\|_{\rnum,2}}
\\  & \le \sup_{k\ge 0} \sum_{j=0}^\infty \sup_t  \Big\|P_{t-j}\Big( \hti(X_t, X_{t-k},r)- \hti(Y_t, Y_{t-k},r)\Big)\Big\|_{\rnum,2}
\\& \le \sup_{k\ge 0} \sum_{j=0}^\infty \sup_t\min\Big(\nu^X_{t,t-k}(j)+\nu^Y_{t,t-k}(j),\Big\|\hti(X_t, X_{t-k},r) - \hti(Y_t, Y_{t-k},r)\Big\|_{\rnum,2}\Big).
\end{align*}
\end{proof}
\begin{lemma}\label{boundprojerrors}
Let $X =(X_t: t\in \znum)$ and $Y=(Y_t: t \in \znum)$ such that $X_t, X_s, Y_t, Y_s$ are $\G_{t-j}$-measurable. Suppose \autoref{as:dep}(i) holds for $W \in \{X,Y\}$. Then 
\begin{align*}
&\Big\| P_{t-j}\Big(\hti(X_t, X_{s},r) - \hti(Y_t, Y_{s},r)\Big)\Big\|_{\rnum,2}
\\& \le \min\Big(\nu^X_{t,s}(j)+\nu^Y_{t,s}(j), C \log^{-\kappa}((2\epsilon)^{-1})+ 2\sup_{i}\pr^{1/2}\big(\partial(X_i, Y_i)> \epsilon\big)\Big)~.
\end{align*} 
\end{lemma}
\begin{proof}[Proof of \autoref{boundprojerrors}]
By \autoref{lem:inddif}
\begin{align*}
\Big\|\hti(X_t, X_s,r) - \hti(Y_t, Y_s,r)\Big\|_{\rnum,p}& \le \min\Big(\pr^{1/p}(r-2\epsilon \le \partial(X_t,X_s) \le r)
+\pr^{1/p}(r\le \partial(X_t,X_s)  \le r+2\epsilon), \\& \phantom{\,\,\quad \min\Big(}\pr^{1/p}(r-2\epsilon \le \partial(Y_t,Y_s) \le r)
+\pr^{1/p}(r\le \partial(Y_t,Y_s)  \le r+2\epsilon)\Big)
\\&+\pr^{1/p}\Big(\partial(X_t, Y_t)> \epsilon \cup \partial(X_s, Y_s) > \epsilon\Big).
\end{align*}
The result now follows from the triangle inequality.
\end{proof}

\begin{lemma}\label{lem:inddif}
\begin{align*}
\E|\mathrm{1}_{\partial(X_1,Y_1)\le r}-\mathrm{1}_{\partial(X_2, Y_2)\le r}| &
\le \min_{i}\pr(r-2\epsilon \le \partial(X_i,Y_i) \le r)
+\min_i\pr(r\le \partial(X_i,Y_i)  \le r+2\epsilon)
\\&+\E \mathrm{1}_{\max(\partial(X_1, X_2), \partial(Y_1, Y_2)) > \epsilon}~.
\end{align*}
\end{lemma}

\begin{proof}[Proof of \autoref{lem:inddif}]
We have
\begin{align*}
|\mathrm{1}_{\partial(X_1,Y_1)\le r}-\mathrm{1}_{\partial(X_2, Y_2)\le r}|&
 \le |\mathrm{1}_{\partial(X_1, Y_1)\le r}-\mathrm{1}_{\partial(X_2, Y_2)\le r}| \mathrm{1}_{\max(\partial(X_1, X_2), \partial(Y_1, Y_2))\le \epsilon} +\mathrm{1}_{\max(\partial(X_1, X_2), \partial(Y_1, Y_2)) > \epsilon}.
 \tageq \label{eq:bound1}
\end{align*}
Note that  
\begin{align*}
\partial(X_1, Y_1) &> r+2\epsilon \text{ and} \max(\partial(X_1, X_2), \partial(Y_1, Y_2))\le \epsilon \Rightarrow \partial(X_2, Y_2)>r \\
 \partial(X_1, Y_1) &< r-2\epsilon \text{ and} \max(\partial(X_1, X_2), \partial(Y_1, Y_2))\le \epsilon \Rightarrow \partial(X_2, Y_2)\le r 
\end{align*}
in which case the first term of \eqref{eq:bound1} is zero. The same argument holds if $X_1$ and $X_2$ and $Y_1$ and $Y_2$ are switched. Hence, 
\begin{align*}
&\E|\mathrm{1}_{\partial(X_1,Y_1)\le r}-\mathrm{1}_{\partial(X_2, Y_2)\le r}| 
\mathrm{1}_{\max(\partial(X_1, X_2), \partial(Y_1, Y_2))\le \epsilon}
\\&\le \min_{i}\E\mathrm{1}_{\{\max(0,r-2\epsilon) \le \partial(X_i,Y_i) \le  r+2\epsilon\}}
\\&\le  \min_i\Big\{\pr(r-2\epsilon \le \partial(X_i,Y_i) \le r)
+\pr(r\le \partial(X_i,Y_i) \le r+2\epsilon)\Big\}.
\end{align*}
\end{proof}
\begin{proof}[Proof of \autoref{lem:weakgmc}]

Apply \autoref{lem:inddif} to find
\begin{align*}
&\Big\|\hti(\tilde{\mb{X}}_t, \tilde{\mb{X}}_{s},r) - \hti(\tilde{\mb{X}}_{m,t}, \tilde{\mb{X}}_{m,s},r)\Big\|_{\rnum,q}\\& \le \min\Big\{\pr^{1/q}(r-2\epsilon \le d_\mathscr{B}\big(\sh(\tilde{\mb{X}}_{t}), \sh(\tilde{\mb{X}}_{s}) \big) \le r)
+\pr^{1/q}(r\le d_\mathscr{B}\big(\sh(\tilde{\mb{X}}_{t}), \sh(\tilde{\mb{X}}_{s}) \big)  \le r+2\epsilon), \\& \phantom{\,\,\quad \min\Big(}\pr^{1/q}(r-2\epsilon \le d_\mathscr{B}\big(\sh(\tilde{\mb{X}}_{m,t}), \sh(\tilde{\mb{X}}_{m,s}) \big)\le r)
+\pr^{1/q}(r\le d_\mathscr{B}\big(\sh(\tilde{\mb{X}}_{m,t}), \sh(\tilde{\mb{X}}_{m,s}) \big)\le r+2\epsilon)\Big\}
\\&+\pr^{1/q}\Big(d_\mathscr{B}\big(\sh(\tilde{\mb{X}}_{t}), \sh(\tilde{\mb{X}}_{m,t}) \big)> \epsilon \cup d_\mathscr{B}\big(\sh(\tilde{\mb{X}}_{s}), \sh(\tilde{\mb{X}}_{m,s}) \big) > \epsilon\Big).
\end{align*}
Thus, using the stability theorem and \autoref{lem:dghds} 
$$d_\mathscr{B}\big(\sh(\tilde{\mb{X}}_{t}), \sh(\tilde{\mb{X}}_{m,t})\big)\le  L \partial_{S}(\mb{X}_{t},\mb{X}_{m,t}),$$ 
and thus \autoref{as:dep}(i) yields
\begin{align*}
&\sup_{t,s}\Big\|\hti(\tilde{\mb{X}}_t, \tilde{\mb{X}}_{s},r) - \hti(\tilde{\mb{X}}_{m,t}, \tilde{\mb{X}}_{m,s},r)\Big\|_{\rnum,q} \le  O\big(\log^{-2\kappa/q}((2\epsilon)^{-1}) \big)
+\sup_t\pr^{1/q}\Big(L \partial_{S}(\mb{X}_{t},\mb{X}_{m,t})> \epsilon \Big).
\end{align*}
Now take $\epsilon=\beta^m$ with $\beta=a(p)^{1/(2p)}$ so that $\log^{-2\kappa/q}((2\beta)^{-m}) = m^{-2\kappa/q} \log^{-2\kappa/ q}( (2\beta)^{-1})=O(m^{-2\kappa/q})$. Then Markov's inequality yields 
\begin{align*}
\sup_{t,s}\Big\|\hti(\tilde{\mb{X}}_t, \tilde{\mb{X}}_{s},r) - \hti(\tilde{\mb{X}}_{m,t}, \tilde{\mb{X}}_{m,s},r)\Big\|_{\rnum,q} &\le  O(m^{-2\kappa/q})
+O\Big(\frac{a(p)^{m/q}}{\beta^{p m/q}}\Big)
\\& = O(m^{-2\kappa/q})
+O(a(p)^{m/(2q)}) =O(m^{-2\kappa/q})
\end{align*}
which is summable over $m$ for $2\kappa /q>1$.  Thus, $\delta_{m,q} = O(m^{-2\kappa/q})$, and hence 
\[
\sum_j \sup_{m: m \ge j} \delta_{m,q} = \sum_j \sup_{m: m \ge j} O(m^{-2\kappa/q}) =O( \sum_j j^{-2\kappa/q}) =O(1), 
\]
proving (a). For (b), a similar argument gives
\begin{align*}
&\Big\|\hti(\tilde{\mb{X}}_{m,t},\tilde{\mb{X}}_{m,t-k},r) -\hti(\tilde{\mb{X}}_{\{t-j\},m,t},\tilde{\mb{X}}_{\{t-j\},m,t-k},r)\Big\|_{\rnum,q}
\\& \le \min\Big\{\pr^{1/q}(r-2\epsilon \le d_\mathscr{B}\big(\sh(\tilde{\mb{X}}_{m,t}), \sh(\tilde{\mb{X}}_{m,t-k}) \big) \le r)
+\pr^{1/q}(r\le d_\mathscr{B}\big(\sh(\tilde{\mb{X}}_{m,t}), \sh(\tilde{\mb{X}}_{m,t-k}) \big)  \le r+2\epsilon), \\& \,\pr^{1/q}(r-2\epsilon \le d_\mathscr{B}\big(\sh(\tilde{\mb{X}}_{\{t-j\},m,t}), \sh(\tilde{\mb{X}}_{\{t-j\},m,t-k}) \big)\le r)
+\pr^{1/q}(r\le d_\mathscr{B}\big(\sh(\tilde{\mb{X}}_{\{t-j\},m,t}), \sh(\tilde{\mb{X}}_{\{t-j\},m,t-k})\big) \le r+2\epsilon\Big\}
\\&+\pr^{1/q}\Big(d_\mathscr{B}\big(\sh(\tilde{\mb{X}}_{m,t-k}), \sh(\tilde{\mb{X}}_{\{t-j\},m,t-k}) \big)> \epsilon \cup d_\mathscr{B}\big(\sh(\tilde{\mb{X}}_{m,t}), \sh(\tilde{\mb{X}}_{\{t-j\},m,t}) \big)> \epsilon\Big)
\\ & \le  O\big(\log^{-
2\kappa/q}((2\epsilon)^{-1}) \big)
+\frac{L^{1/q}}{\epsilon^{p/q}}
\Bigg\{\big(\E \partial^p_{S}(\mb{X}_{m,t-k},\mb{X}_{\{t-j\},m,t-k}) \big)^{1/q} +\big(\E \partial^p_{S}(\mb{X}_{m,t},\mb{X}_{\{t-j\},m,t})\big)^{1/q}\Bigg\}.  \tageq \label{eq:minniemouse}
\end{align*} The second term is zero if 
\begin{align*}
    t-j > t-k \text{ or } t-k-m+1 >t-j \leftrightarrow j<k  \text{ or } j> k+m-1
\end{align*}
whereas the third term will be zero if 
\[
t-m+1> t-j \leftrightarrow j > m-1.
\]
Thus the second term only plays a role if $k \le  j\le k+m-1$ whereas the third term only plays a role if $j\le m-1$. For the supremum over $m \ge 1$, only the second term drops in \eqref{eq:minniemouse}
if $j<k$.  
Therefore take 
\begin{align*}
\epsilon =
\begin{cases}
 a^{(j+1)/(2p)} & \text{ if  } j<k; 
 \\
    a^{(j-k+1)/(2p)} & \text{ if } j \ge k. 
\end{cases}
\end{align*}
Then an analogous derivation as for (a) yields the upper bound
\begin{align*}
  &  \sup_{k \ge 0}\sum_{j=0}^{k-1} O\Big( (j+1)^{-2\kappa/q}+\frac{b(p)^{j/q}}{b(p)^{j/(2q)}}\Big)  +    \sup_{k \ge 0} \sum_{j=k}^{\infty}  O\Big( (j-k+1)^{-2\kappa/q}+\frac{b(p)^{(j-k)/q}+b(p)^{(j)/q}}{b(p)^{(j-k)/(2q)}}\Big) 
    \\& \le 2 \sum_{j=1}^{\infty} O(j^{-2\kappa/q})
    +    \sup_{k \ge 0}\sum_{j=k}^{\infty}  O\Big( (j-k+1)^{-2\kappa/q}+2\frac{b(p)^{(j-k)/q}}{b(p)^{(j-k)/(2q)}}\Big) 
    \\& = O\Big(\sum_{j=1}^{\infty} j^{-2\kappa/q}\Big) = O(1).
\end{align*}
where we used that $b(p)^{j/q} \le b(p)^{(j-k)/q}$ and that by assumption $2\kappa/q>1$. 
\end{proof}

\begin{proof}[Proof of \autoref{lem:strongmc}]
It follows from \autoref{sumbound}(i) that it suffices to show
\begin{align*}
  \sup_n  \sum_j j \sup_{m: m \ge j} \delta^{\tilde{\mb{X}}^n}_{m,q} <\infty.
\end{align*}
Now an argument as in the proof of \autoref{lem:weakgmc} gives in this case that for $p=1$
\begin{align*}
\sup_n\sup_{t,s}\Big\|\hti(\tilde{\mb{X}}^n_t, \tilde{\mb{X}}^n_{s},r) - \hti(\tilde{\mb{X}}^n_{m,t}, \tilde{\mb{X}}^n_{m,s},r)\Big\|_{\rnum,q} =  O(a^{m/(2pq)}) +O(a^{m/(2q)}) = O(a^{m/(2q)}). 
\end{align*}
Using the power series and taking derivatives yields
\begin{align*}
   \frac{1}{2q}\sum_{j=0}^\infty j a^{j/(2q)-1}=
\frac{1}{2q}a^{1/(2q)-1} \frac{1}{(1-a^{1/(2q)})^2}<\infty,
\end{align*}
and thus
\begin{align*}
   \sup_n \sum_j j \sup_{m: m \ge j}  \delta^{\tilde{\mb{X}}^n}_{m,q} \le \sum_{j} j a^{j/(2q)}<\infty.
\end{align*}
\end{proof}

\section{Proof of Theorem \ref{thm:limsupm} and Lemma \ref{lem:convar}}\label{sec:appc}

\begin{proof}[Proof of \autoref{thm:limsupm}]
Set $k=\flo{uT}$. We decompose the error as follows
\begin{align*}
T^{1/2}R_{m,T}(k,r)&=T^{1/2}(\mathscr{S}_T(u,r)-\E \mathscr{S}_T(u,r)) -T^{1/2}M_{m,T}(u,r)
\\& = \frac{1}{T^{3/2}}\sum_{t,s=1}^{k}\hti(\tilde{\mb{X}}^{n(T)}_t, \tilde{\mb{X}}^{n(T)}_s, r) -\hti(\tilde{\mb{X}}_t,\tilde{\mb{X}}_s, r) \tageq \label{eq:pcpprox}
\\&+\frac{1}{T^{3/2}}\sum_{t,s=1}^{k}\hti(\tilde{\mb{X}}_t,\tilde{\mb{X}}_s, r) -\hti(\xu{t}, \xu{s}, r) \tageq \label{eq:locpprox}
\\&+\frac{1}{T^{3/2}}\sum_{t,s=1}^{k}\hti(\xu{t}, \xu{s}, r) -\hti(\xum{t}, \xum{s}, r) \tageq \label{eq:mdepapprox}
 \\&+\frac{1}{T^{3/2}}\sum_{t,s=1}^{k}\hti(\xum{t}, \xum{s}, r)
-\sum_{\substack{t,s=1\\ |t-s|>m}}^k \sum_{i \in \{s,t\}}\E\Big[\hti(\xum{t}, \xum{s}, r)\bv\G_{i,m}\Big]~.
 \tageq \label{eq:mtailapprox}
\end{align*}
After showing the above terms converge to zero, it remains to control a  final error term with the part that drives the distributional properties. Namely, \begin{align*}
&\frac{1}{T^{3/2}} \sum_{\substack{t,s=1\\ |t-s|>m}}^k \E\Big[\hti(\xum{t}, \xum{s}, r)\bv\G_{t,m}\Big]
 -\frac{1}{T^{1/2}}\sum_{t=1}^{\flo{uT}}\sum_{j=0}^{m-1} P_{t}\big(Y_{m,t+j}(\frac{t+j}{T}, u, r)\big) \tageq \label{eq:error4}
\end{align*}

where we recall that  
\[Y_{m,T,t+j}(\frac{t+j}{T},u,r) =\E\Big[\frac{1}{T}\sum_{s=1}^{\flo{uT}}\E_{\tilde{\mb{X}}^\prime(\cdot)}\Big[\hti(\xum{t+j},\tilde{\mb{X}}^{\prime}_{m,0}(\frac{s}{T}), r)\Big]\bv\G_{t+j,m}\Big].
\]
We treat these terms in turn. For \eqref{eq:pcpprox},\eqref{eq:locpprox} and \eqref{eq:mdepapprox}, we apply \autoref{lem:errorcontrol} and \autoref{boundprojerrors}. More specifically, for \eqref{eq:pcpprox}, this yields under \autoref{as:mincov}
\begin{align*}&
\limsup_{m\to \infty} \limsup_{T \to \infty} \Big\| \max_{1\le  k\le T}\frac{1}{T^{3/2}}\Big\vert\sum_{t,s=1}^{k} \hti(\tilde{{\mb{X}}}^{n(T)}_t, \tilde{{\mb{X}}}^{n(T)}_{s},r) - \hti(\tilde{\mb{X}}_t, \tilde{\mb{X}}_{s},r)\Big\vert\Big\|_{\rnum,2}
\\& \lesssim \limsup_{m\to \infty} \limsup_{T \to \infty}  \sup_{k\ge 0} \sum_{j=0}^\infty \sup_t\min\Big(\nu^{\tilde{\mb{X}}}_{t,t-k}(j),C \log^{-\kappa}( (2\epsilon_T)^{-1})+ 2\sup_t\pr^{1/2}\big(\partial(\tilde{\mb{X}}_t,  \tilde{{\mb{X}}}^{n(T)}_t)> \epsilon_T\big)\Big)=0
\end{align*}
for all $\alpha_{n,T} << \epsilon_T << \alpha^{1/2}_{n,T}$.
Similarly, for \eqref{eq:locpprox}, we get
\begin{align*}
&\limsup_{m\to \infty} \limsup_{T \to \infty} \Big\| \max_{1\le  k\le T}\frac{1}{T^{3/2}}\Big\vert\sum_{t,s=1}^{k} \hti(\tilde{\mb{X}}_t, \tilde{\mb{X}}_{s},r) - \hti(\tilde{\mb{X}}_t(\frac{t}{T}), \tilde{\mb{X}}_{s}(\frac{s}{T}),r)\Big\vert\Big\|_{\rnum,2}
\\& \lesssim \limsup_{m\to \infty} \limsup_{T \to \infty}  \sup_{k\ge 0} \sum_{j=0}^\infty \sup_t\min\Big(\nu^{\tilde{\mb{X}}}_{t,t-k}(j)+\nu^{\tilde{\mb{X}}(\cdot)}_{t,t-k}(j),C \log^{-\kappa}((2\epsilon_T)^{-1})+ 2\sup_t\pr^{1/2}\big(\partial(\tilde{\mb{X}}_t,  \tilde{\mb{X}}_t(\frac{t}{T}))> \epsilon_T\big)\Big)=0~
\end{align*}
for any $\epsilon_T =T^{-\alpha}, 0<\alpha<1$. For \eqref{eq:mdepapprox}, \autoref{lem:errorcontrol} yields under \autoref{as:dep}(ii)
\begin{align*}
&\limsup_{m\to \infty} \limsup_{T \to \infty} \Big\| \max_{1\le  k\le T}\frac{1}{T^{3/2}}\Big\vert\sum_{t,s=1}^{k} \hti(\tilde{\mb{X}}_t(\frac{t}{T}), \tilde{\mb{X}}_{s}(\frac{s}{T}),r) - \hti(\tilde{\mb{X}}_{m,t}(\frac{t}{T}), \tilde{\mb{X}}_{m,s}(\frac{s}{T}),r)\Big\vert\Big\|_{\rnum,2}
\\& \lesssim \limsup_{m\to \infty} \sup_{k\ge 0} \sum_{j=0}^\infty \sup_t\min\Big(2\nu^{\tilde{\mb{X}}(\cdot)}_{t,t-k}(j),\delta_m\Big)=0. \tageq \label{eq:limmdepapprox}
\end{align*}
For \eqref{eq:mtailapprox} we use
\autoref{hoeffding}. Finally, that the error terms in \eqref{eq:error4} are negligible follows from \autoref{lem:projpart}.
\end{proof}

\begin{lemma}\label{hoeffding}
Suppose \autoref{as:dep} holds. Then 
\[
\limsup_m \limsup_{T} \Big\| \max_{k \le T} \frac{1}{T^{3/2}} \bv \sum_{t,s=1}^k \hti(\xum{t}, \xum{s}, r)- \sum_{\substack{t,s=1\\ |t-s|>m}}^k \sum_{i \in \{s,t\}} \E[\hti(\xum{t}, \xum{s}, r)|\G_{i,m}]
\bv \Big\|_{\rnum,2}=0.\]
\end{lemma}
\begin{proof}
We decompose this term as 
\begin{align*}
&\Big\| \max_{k \le T} \frac{1}{T^{3/2}} \bv \sum_{t,s=1}^k\hti(\xum{t}, \xum{s}, r)-\E[\hti(\xum{t}, \xum{s}, r)|\G_{t,m}]-\E[\hti(\xum{t}, \xum{s}, r)|\G_{s,m}] \bv \Big\|_{\rnum,2}
\\& \le 
\Big\| \max_{k \le T} \frac{1}{T^{3/2}} \bv\sum_{t,s=1}^{k}\hti(\xum{t}, \xum{s}, r)- \sum_{\substack{t,s=1\\ |t-s|>m}}^k\hti(\xum{t}, \xum{s}, r) \bv \Big\|_{\rnum,2}
\\&+
 \Big\| \max_{k \le T} \frac{1}{T^{3/2}} \bv \sum_{\substack{t,s=1\\ |t-s|>m}}^k\hti(\xum{t}, \xum{s}, r)-\E[\hti(\xum{t}, \xum{s}, r)|\G_{t,m}]-\E[\hti(\xum{t}, \xum{s}, r)|\G_{s,m}] \bv \Big\|_{\rnum,2}
 \\&=: I_1+II_2.
\end{align*}
We treat $I_1$ first. Using \[\sum_{\substack{t,s=1\\ |t-s|\le m}}=\sum_{t=1}^k \sum_{s=(t-m) \vee 1}^{(t+m) \wedge k} =\sum_{t=1}^k \sum_{s=(t-m) \vee 1}^t +\sum_{t=1}^k \sum_{s=t+1}^{(t+m) \wedge k} =\sum_{t=1}^k \sum_{s=(t-m) \vee 1}^t +\sum_{s=2}^k \sum_{t=s-m\vee 1}^{(s-1) }.\]
and \autoref{lem:maxmom}, and a similar argument as in the proof of \autoref{lem:errorcontrol}, yields that for any fixed $m$
\begin{align*}
\Big\| \max_{k \le T}  \bv\sum_{t=1}^k \sum_{s=t-m}^t \hti(\xum{t}, \xum{s}, r) \bv \Big\|_{\rnum,2} &\lesssim \sum_{j=0}^{\infty} \big( \E\Big\|    \sum_{t=1}^T P_{t-j} \Big( \sum_{s=t-m}^t \hti(\xum{t}, \xum{s}, r)\Big)  \Big\|^2\big)^{1/2}
\\& =O(T^{1/2}m)=o(T^{3/2}). \tageq \label{eq:mapprox1term}
\end{align*}
To treat $I_2$,  write 
\[
J_{t,s}= \hti(\xum{t}, \xum{s}, r)-\E[\hti(\xum{t}, \xum{s}, r)|\G_{t,m}]-\E[\hti(\xum{t}, \xum{s}, r)|\G_{s,m}]
\]
and let $p_1, p_2 = 1, \ldots, m$ and $q_1 \neq q_2 \le \flo{k/m}$. The Cauchy Schwarz inequality yields
\begin{align*}
\E\bv \max_{k \le T} \bv \sum_{t,s=1: |t-s| > m}^{k}J_{t,s}\bv \bv^2 &=\E\bv \max_{k \le T}\bv \sum_{1\le p_1, p_2 \le m} \sum_{1\le q_1 \neq q_2\le \flo{k/m}} J_{p_1+m q_1, p_2+m q_2 } \bv^2  \bv
\\& \le m^2 \E\bv \max_{k \le T} \sum_{1\le p_1, p_2 \le m}\bv \sum_{1\le q_1 \neq q_2\le \flo{k/m}} J_{p_1+m q_1, p_2+m q_2 } \bv^2 \bv 
\\& \le m^2 \sum_{1\le p_1, p_2 \le m}\E\bv \max_{k \le T} \bv \sum_{1\le q_1 \neq q_2\le \flo{k/m}} J_{p_1+m q_1, p_2+m q_2 } \bv^2 \bv~. 
\end{align*}
For the term inside the expectation we note that it can be verified that $J_{p_1+mq_1, p_2+m q_2}$ and $J_{r_1+m s_1, r_2+m s_2}$ are uncorrelated for different pairs $(q_1, q_2) \neq (s_1, s_2)$ 
\begin{align*}
  \E \bv \sum_{1\le q_1 < q_2\le \flo{k/m}} J_{p_1+m q_1, p_2+m q_2 } \bv^2 & \le  \E \bv \sum_{q_2=2}^{\flo{k/m}} \sum_{q_1=1}^{q_2-1} J_{p_1+m q_1, p_2+m q_2 } \bv^2 
 \\& =  \sum_{q_2=2}^{\flo{k/m}} \sum_{q_1=1}^{q_2-1} \E \bv J_{p_1+m q_1, p_2+m q_2 } \bv^2 \\&\le C \sum_{q_2=2}^{\flo{k/m}} (q_2-1)
 \\&= C \frac{(\flo{k/m}-1)\flo{k/m}}{2}~.
\end{align*}
Hence, if we define our function $ S_{b,n}=
 \sum_{l=b+1}^{b+n} \zeta_l$ where $\zeta_l = \sum_{q_2=b+2}^{l-1}$. Then we find that 
 \[
 \E|S_{b,n}|^2 \le g(F_{b,n}): = C \sum_{q_2=b+1}^{b+\flo{k/m}} \sum_{q_1=b}^{q_2-1} \le C \sum_{i=1}^{\flo{k/m}} i  
 \]
 and it follows from theorem 3 in \cite{Mor76}
 \begin{align*}
 \E\bv \max_{1 \le k \le T} \bv \sum_{1\le q_1 < q_2\le \flo{k/m}} J_{p_1+m q_1, p_2+m q_2 } \bv^2 \bv \le C \log_2(2 \flo{T/m})^2 \flo{T/m}^2~. \tageq\label{eq:mapprox2term}
 \end{align*}
 Therefore, for fixed $m$
\[
\E\bv \max_{k \le T} \bv \sum_{t,s=1: |t-s| > m}^{k}J_{t,s}\bv \bv^2 
 \le  C m^2 \log_2(2 \flo{T/m})^2 T^2 =o(T^3).  \]
\end{proof}

\begin{lemma}\label{lem:projpart}
\begin{align*}
\limsup_{m \to \infty} \lim_{T \to \infty}  \pr\Big(\sup_{u }\frac{1}{T^{3/2}} \bv \sum_{\substack{t,s=1\\ |t-s|>m}}^{\flo{uT}} \E\Big[\hti(\xum{t}, \xum{s}, r)\bv\G_{t,m}\Big]
  -\sum_{\substack{t=1}}^{\flo{uT}} \sum_{j=0}^{m-1} P_{t}\big(Y_{m,T,t+j}(\frac{t+j}{T},u,r) \big)\bv\ge \epsilon \Big)=0.
\end{align*}
\end{lemma}

\begin{proof}[Proof of \autoref{lem:projpart}]

We recall the notation \eqref{eq:Ymt}, and make a few remarks. Observe that $\xum{s}$ is independent of $\G_{t,m}$ for $|s-t|>m$ and $\xum{t}$ is measurable with respect to  $\G_{t,m}$. Let $(\tilde{\mb{X}}^\prime_t(u): t\in \znum, u\in[0,1])$ be an independent copy of the  auxiliary process $(\mb{X}_t(u): t\in \znum, u\in[0,1])$. Then, almost surely, for $|s-t|>m$ 
\begin{align*}
\E\Big[\hti(\xum{t}, \xum{s}, r)\bv\G_{t,m}\Big] &=
\E\Big[\E_{\xum{s}}\big[\hti(\xum{t}, \xum{s}, r)\big]\bv\G_{t,m}\Big] 
\\&= \E\Big[\E_{\tilde{\mb{X}}^\prime_{m,0}(\frac{s}{T})}\big[\hti(\xum{t},\tilde{\mb{X}}^\prime_{m,0}(\frac{s}{T}), r)\big]\bv\G_{t,m}\Big].
\end{align*}

It follows similar to the proof of term $I_1$ in the proof of \autoref{hoeffding} that
\[
\limsup_{m \to \infty} \lim_{T \to \infty}\Big\| \max_{k \le T} \frac{1}{T^{3/2}} \bv \big(\sum_{t,s=1}^{k}-\sum_{\substack{t,s=1\\ |t-s|>m}}^k\big)\E\Big[\E_{\tilde{\mb{X}}_{m,0}(\frac{s}{T})}[\hti(\xum{t},\tilde{\mb{X}}^{\prime}_{m,0}(\frac{s}{T}), r)]\bv\G_{t,m}\Big]  \bv \Big\|_{\rnum,2}=0.
\]
Thus  with $Y_{m,T,t}(\frac{t}{T},u,r) =\E\Big[\frac{1}{T} \sum_{s=1}^{\flo{uT}} \E_{\tilde{\mb{X}}^\prime_{m,0}(\frac{s}{T})}\Big[\hti(\xum{t},\tilde{\mb{X}}^\prime_{m,0}(\frac{s}{T}), r)\Big]\Big\vert\G_{t,m}\Big]$, we have 
\begin{align*}
\frac{1}{T^{3/2}} \sum_{t,s=1}^{\flo{uT}}\E\Big[\E_{\tilde{\mb{X}}^\prime_{m,0}(\frac{s}{T})}\Big[\hti(\xum{t},\tilde{\mb{X}}^\prime_{m,0}(\frac{s}{T}), r)\Big]\Big\vert\G_{t,m}\Big] &= \frac{1}{T^{1/2}} \sum_{t=1}^{\flo{uT}}{Y}_{m,T,t}(\frac{t}{T},u,r).
\end{align*}
Therefore, it suffices to control
\begin{align*}
\frac{1}{T^{1/2}} \sum_{t=1}^{\flo{uT}}Y_{m,T,t}(\frac{t}{T},u,r)
 -\frac{1}{T^{1/2}}\sum_{\substack{t=1}}^{\flo{uT}} \sum_{j=0}^{m-1} P_{t}\Big({Y}_{m,T,t+j}(\frac{t+j}{T},u,r) \Big)~.
\end{align*}
Note that for all $u,v\in [0,1]$, $Y_{m,T,t}(v,u,r)$ is zero mean, $\G_{t,m}$-measurable and is independent of $\G_{l}$ where $l <  t-m+1$. Hence, almost surely
\[
T^{-1/2} \sum_{t=1}^{\flo{uT}}Y_{m,T,t}(\frac{t}{T},u,r) =T^{-1/2} \sum_{t=1}^{\flo{uT}}\sum_{j=0}^{m-1}P_{t-j}\Big(Y_{m,T,t}(\frac{t}{T},u,r) \Big).
\] 
 \autoref{lem:projpart} now follows from \autoref{cor:shiftfixr}, which controls the shift in rescaled time.  \end{proof}
\begin{lemma}\label{lem:shift}
\begin{align*}
& \pr\Big(\sup_{u \in [0,1]} \frac{1}{\sqrt{T}}\bv \sum_{t=1}^{\flo{uT}}\sum_{j=0}^{m-1} P_{t-j}\big(Y_{m,T,t}(\frac{t}{T},u,r)\big)-\sum_{t=1}^{\flo{uT}}\sum_{j=0}^{m-1} P_{t}\big(Y_{m,T,t+j}(\frac{t+j}{T},u,r)\big)  \bv >\epsilon\Big) 
\\&
\lesssim T m^2 \frac{m^{2\ell}}{T^{\ell/2}} \max_{j=0,\ldots,m-1}  \E\Big[ \big(\sup_{u,v \in [0,1]}\bv
P_0\big(Z_{m,j}(u,v)\big)\bv\Big)^\ell \mathrm{1}_{\big\{\sup_{u,v \in [0,1]}\bv P_0\big(Z_{m,j}(u,v)\big)\bv > \sqrt{T}/m^2 \epsilon\big\}}\Big] 
\end{align*}
where $Z_{m,j}(u,v)$ is defined in \eqref{eq:zmt}.
\end{lemma}
\begin{proof}[Proof of \autoref{lem:shift}]
Making the identification $k=\flo{u T}$, and using a change of variables
\begin{align*}
\sum_{t=1}^{k} P_{t-j}\big(Y_{m,T,t}(\frac{t}{T},\frac{k}{T},r)\big) 
&=\sum_{t=1-j}^{0} P_{t}\big(Y_{m,T,t+j}(\frac{t+j}{T},\frac{k}{T},r)\big) +\sum_{t=1}^{k-j} P_{t}\big(Y_{m,T,t+j}(\frac{t+j}{T},\frac{k}{T},r)\big) 
\\&=\sum_{t=1}^{j} P_{t-j}\big(Y_{m,T,t}(\frac{t}{T},\frac{k}{T},r)\big) +\sum_{t=1}^{k-j} P_{t}\big(Y_{m,T,t+j}(\frac{t+j}{T},\frac{k}{T},r)\big),
\end{align*}
from which we find
\begin{align*}
&\bv \sum_{t=1}^{k}\sum_{j=0}^{m-1} P_{t-j}\big(Y_{m,T,t}(\frac{t}{T},\frac{k}{T},r)\big)-\sum_{t=1}^{k}\sum_{j=0}^{m-1} P_{t}\big(Y_{m,T,t+j}(\frac{t+j}{T},\frac{k}{T},r)\big)\bv \\&\le \sum_{j=0}^{m-1}  \sum_{t=1}^j\bv P_{t-j}\big(Y_{m,T,t}(\frac{t}{T},u,r)\big)\bv +\sum_{j=0}^{m-1}  \sum_{t=k-j+1}^{k}\bv P_{t}\big(Y_{m,T,t+j}(\frac{t+j}{T},\frac{k}{T},r)\big)\bv
\\&\lesssim m^2 \max_{j=0,\ldots,m-1}\max_{t=1,\ldots,j}\bv P_{t-j}\big(Y_{m,T,t}(\frac{t}{T},\frac{k}{T},r)\big)\bv + m^2 \max_{j=0,\ldots,m-1}\max_{t=k-j+1,\ldots,k}\bv P_{t}\big(Y_{m,T,t+j}(\frac{t+j}{T},\frac{k}{T},r)\big)\bv.
\end{align*}
For the next step, denote
 \[Z_{m,t}(\frac{t}{T},\frac{s}{T}):= \E\Big[\E_{\tilde{\mb{X}}^\prime_{m,0}(\frac{s}{T})}\Big[\hti(\xum{t},\tilde{\mb{X}}^\prime_{m,0}(\frac{s}{T}), r)\Big]\Big\vert\G_{t,m}\Big]. \tageq\label{eq:zmt}\]
Observe $Y_{m,T,t}(\frac{t}{T},u,r) =\frac{1}{T} \sum_{s=1}^{\flo{u T}} Z_{m,t}(\frac{t}{T},\frac{s}{T})$ and that  
 $P_{t}\big(Y_{m,T,t+j}(\frac{t+j}{T},u,r) \big)
 \overset{d}{=} P_{0}\big({Y}_{m,T,j}(\frac{t+j}{T},u,r)\big)$. A union bound and $b^\ell\, \pr(|X| \ge b) = b^\ell\cdot \E[\mathrm{1}_{\{|X|\ge b\}}] \le   \E[|X|^\ell \mathrm{1}_{\{|X| \ge b\}}] $ therefore yield
 \begin{align*}
&\pr\Big(\max_{1\le k \le T} \frac{1}{\sqrt{T}}\sum_{j=0}^{m-1}  \sum_{t=k-j+1}^{k}\bv P_{t}\big(Y_{m,T,t+j}(\frac{t+j}{T},\frac{k}{T},r) \big)\bv >\epsilon\Big)  
\\& \le \pr\Big(\max_{1\le k \le T} \frac{1}{T^{3/2}}\sum_{j=0}^{m-1}  \sum_{t=k-j+1}^{k}\sum_{s=1}^T \bv P_{t}\big(Z_{m,t+j}(\frac{t+j}{T},\frac{s}{T},r) \big)\bv >\epsilon\Big)  
\\& \le \pr\Big(\max_{1\le k \le T} \frac{1}{T^{1/2}}m^2 \max_{j=0,\ldots,m-1}\max_{t= 1 \vee k-j+1,\ldots,k} \sup_s \bv P_{t}\big(Z_{m,t+j}(\frac{t+j}{T},\frac{s}{T},r) \big)\bv >\epsilon\Big)  
   \\& \lesssim T m^2 \frac{m^{2\ell}}{T^{\ell/2}} \max_{j=0,\ldots,m-1}  \E\Big[ \big(\sup_{u,v \in [0,1]}\bv
P_0\big(Z_{m,j}(u,v)\big)\bv\Big)^\ell \mathrm{1}_{\big\{\sup_{u,v \in [0,1]}\bv P_0\big(Z_{m,j}(u,v)\big)\bv > \sqrt{T}/m^2 \epsilon\big\}}\Big] \to 0
\end{align*}
 as $T \to \infty$, for any fixed $m$ and $\ell \ge 2$, since
\[\Big\|\sup_{u,v \in [0,1]}\bv P_0\big(Z_{m,j}(u,v)\big)\bv\Big\|_{\rnum,\ell}\le \Big\|\sup_{u,v \in [0,1]}\bv\E_{\tilde{\mb{X}}_{m,0}^\prime(v)}\hti(\tilde{\mb{X}}_{m,0}(u),\tilde{\mb{X}}^\prime_{m,0}(v),r)\bv\Big\|_{\rnum,\ell} \le 1.\]
\end{proof}
We obtain the following immediately by taking $\ell \ge 2$ in \autoref{lem:shift}. 
\begin{Corollary}
    \label{cor:shiftfixr}
\begin{align*}
\limsup_{m \to \infty} \limsup_{T \to \infty} \pr\Big(\sup_{u \in [0,1]} \frac{1}{\sqrt{T}}\bv \sum_{t=1}^{\flo{uT}}\sum_{j=0}^{m-1} P_{t-j}\big(Y_{m,T,t}(\frac{t}{T},u,r)\big)-\sum_{t=1}^{\flo{uT}}\sum_{j=0}^{m-1} P_{t}\big(Y_{m,T,t+j}(\frac{t+j}{T},u,r)\big)  \bv >\epsilon\Big) = 0.  
\end{align*}
\end{Corollary}

\begin{proof}[Proof \autoref{lem:convar}]
We recall the notation
\[Y_{m,T,t}(\frac{t}{T},u,r) =\E\Big[\E_{\tilde{\mb{X}}^{\prime}_{m,0}(\frac{s}{T})}\Big[\frac{1}{T} \sum_{s=1}^{\flo{uT}}\hti(\xum{t},\tilde{\mb{X}}^{\prime}_{m,0}(\frac{s}{T}), r)\Big]\Big\vert\G_{t,m}\Big].
\]

Define $\mc{E}_{L,l,k}= \Big\{t: \frac{t}{k}\in \big(\frac{l-1}{2^L},\frac{l}{2^L}\big]\Big\}$ as well as
\[
\Gamma_{m,T,t}\Big(\frac{t+j}{T}, \frac{t+j^\prime}{T}\Big)= P_{t}\Big(Y_{m,T,t+j}\big(\frac{t+j}{T},\frac{k}{T},r\big)\Big)  P_{t}\Big(Y_{m,T,t+j^\prime}\big(\frac{t+j^\prime}{T},\frac{k}{T},r\big)\Big)
\]
and 
\[
\Gamma_{m,T,t,j,j^\prime}\Big(v, w\Big)= P_{t}\Big(Y_{m,T,t+j}\big(v,\frac{k}{T},r\big)\Big)  P_{t}\Big(Y_{m,T,t+j^\prime}\big(w,\frac{k}{T},r\big)\Big).
\]
Then the triangle inequality yields
\begin{align*}
&\E\Bigg\vert\frac{1}{k}\sum_{t=1}^{k}\E\Big[ \Gamma_{m,T,t}\Big(\frac{t+j}{T}, \frac{t+j^\prime}{T}\Big) \vert \G_{t-1}\Big]-\frac{1}{2^L}\sum_{l=1}^{2^L} \frac{1}{|\mc{E}_{L,l,k}|}\sum_{t \in \mc{E}_{L,l,k}}\E\Big[ \Gamma_{m,T,t,j,j^\prime}\Big(\frac{l}{2^L}u,\frac{l}{2^L}u\Big)\Big\vert^2 \Big \vert \G_{t-1}\Big] \Bigg\vert
 \\& 
\le \E\Bigg\vert\frac{1}{k}\sum_{t=1}^{k}\E\Big[  \Gamma_{m,T,t}\Big(\frac{t+j}{T}, \frac{t+j^\prime}{T}\Big) \Big \vert \G_{t-1}\Big]-\frac{1}{2^L}\sum_{l=1}^{2^L} \frac{1}{|\mc{E}_{L,l,k}|}\sum_{t \in \mc{E}_{L,l,k}}\E\Big[  \Gamma_{m,T,t}\Big(\frac{t+j}{T}, \frac{t+j^\prime}{T}\Big) \Big \vert \G_{t-1}\Big]\Bigg\vert  \tageq \label{erI}
\\& +\E\Bigg\vert \frac{1}{2^L}\sum_{l=1}^{2^L} \frac{1}{|\mc{E}_{L,l,k}|}\sum_{t \in \mc{E}_{L,l,k}}\Bigg(\E\Big[  \Gamma_{m,T,t}\Big(\frac{t+j}{T}, \frac{t+j^\prime}{T}\Big) \Big \vert \G_{t-1}\Big]-\E\Big[ \Gamma_{m,T,t,j,j^\prime}\Big(\frac{l}{2^L}u,\frac{l}{2^L}u\Big) \Big \vert \G_{t-1}\Big]\Bigg)\Bigg\vert. \tageq \label{erII}
\end{align*}
Note that if $2^L\ge k$, then the error \eqref{erI} is zero. If $k >2^L$, the points are not evenly distributed across the bins. The number of points per bin is $
\flo{\frac{k}{2^L}}$ or $\ceil{\frac{k}{2^L}}$ and we can therefore bound \eqref{erI} by
\begin{align*}
 \E\Big \vert\sum_{l=1}^{2^L} \Big(\frac{|\mc{E}_{L,l,k}|}{k}-\frac{1}{2^L}\Big) 
 \frac{1}{|\mc{E}_{L,l,k}|}\sum_{t \in \mc{E}_{L,l,k}} (\cdot)\Big \vert
& \le \sum_{l=1}^{2^L} \Big\vert\frac{|\mc{E}_{L,l,k}|}{k}-\frac{1}{2^L}\Big\vert \max_{t=1,\ldots, T}\max_{j,j^\prime =0,\ldots, m-1} \E|\Gamma_{m,T,t}\Big(\frac{t+j}{T}, \frac{t+j^\prime}{T}\Big)|
 \\& \le 
2^L\Big\vert\frac{2^L(k/2^L\pm 1) - k}{k 2^L}\Big\vert \max_{t=1,\ldots, T}\max_{j,j^\prime =0,\ldots, m-1} \E|\Gamma_{m,T,t}\Big(\frac{t+j}{T}, \frac{t+j^\prime}{T}\Big)|
\\&=O(2^L/k). 
\end{align*}
Since, by the Cauchy Schwarz inequality,
\begin{align*}
&\E\bv \Gamma_{m,T,t}\Big(\frac{t+j}{T}, \frac{t+j^\prime}{T}\Big)\bv
\le \sqrt{\E \bv P_{t}\Big(Y_{m,T,t+j}\big(\frac{t+j}{T},\frac{k}{T},r\big)\Big) \bv^2}  \sqrt{\E \bv P_{t}\Big(Y_{m,T,t+j^\prime}\big(\frac{t+j^\prime}{T},\frac{k}{T},r\big)\Big) \bv^2}
\end{align*}
and by the contraction property
\begin{align*}
\E \bv P_{t}\Big(Y_{m,T,t+j}\big(\frac{t+j}{T},\frac{k}{T},r\big)\Big) \bv^2
& \le \sup_{u,v \in [0,1]}\E\bv \E_{\tilde{\mb{X}}^{\prime}_{m,0}(v)}\Big[\frac{1}{T} \sum_{s=1}^{k}\hti(\tilde{\mb{X}}_{m,0}(u),\tilde{\mb{X}}^{\prime}_{m,0}(v), r)\Big]\bv^2
\\& =O(\frac{k^2}{T^2}).
\end{align*}
For \eqref{erII}, we simply observe that $t \in \mc{E}_{L,l,k}$ implies
$\frac{t}{T} \in \Big(\frac{l-1}{2^L}\frac{k}{T}, \frac{l}{2^L}\frac{k}{T}\Big]$
and since $k=\flo{uT}$, 
\[
\bv\frac{t+j}{T}-\frac{l}{2^L}u \bv \le \frac{j}{T} +\bv \frac{l\pm 1}{2^L}\frac{\flo{uT}}{T} -\frac{l}{2^L}\frac{\flo{u T}}{T} \frac{u T}{\flo{uT}}\bv \le  \frac{j}{T}+ \bv \frac{1}{2^L}\frac{\flo{u T}}{T} \bv+ \bv\frac{l}{2^L}( \frac{u T}{\flo{uT}}-1)\bv \le \frac{m}{T}+\frac{1}{2^L}+\frac{1}{T}
\]
where it was used that $\sup_{u \in (1/T,1]}| \frac{u T}{\flo{uT}}-1|$. This is of order $O(\frac{m}{T}+\frac{1}{2^L})$. Therefore, an application of \autoref{lem:inddif}  ensures \eqref{erII} goes to zero as $T \to \infty$ and $L\to \infty$. Recall the notation defined in \eqref{eq:zmt} 
\[Z_{m,t}(u,v)=\E\Big[\E_{\tilde{\mb{X}}^{\prime}_{m,0}(v)}\Big[\hti(\tilde{\mb{X}}_{m,t}(u),\tilde{\mb{X}}^{\prime}_{m,0}(v), r)\Big]\Big\vert\G_{t,m}\Big]\]
and denote
\[
g(\mb{X}_{m,t}(v),u,r):=\int_0^u\E_{\tilde{\mb{X}}^\prime_{m,0}(w)}[\hti(\tilde{\mb{X}}_{m,t}(v),\tilde{\mb{X}}^\prime_{m,0}(w), r)] dw.
\]
From the above, we thus find
\begin{align*}
&\lim_{L\to \infty} \lim_{T\to \infty} \sum_{j,j^\prime=0}^{m-1}\frac{\flo{uT}}{T}\frac{1}{\flo{uT}} \sum_{t=1}^{\flo{uT}}\E\Big[  \Gamma_{m,T,t}\Big(\frac{t+j}{T}, \frac{t+j^\prime}{T}\Big) \Big \vert \G_{t-1}\Big]
\\&=
\lim_{L\to \infty} \lim_{T\to \infty}\frac{\flo{uT}}{T}\sum_{j,j^\prime=0}^{m-1}\frac{1}{2^L}\sum_{l=1}^{2^L}  \frac{1}{|\mc{E}_{L,l,\flo{uT}}|}\sum_{t \in \mc{E}_{L,l,\flo{uT}}}
 \E\Big[ \frac{1}{T^2} \sum_{s,s^\prime=1}^{\flo{uT}} 
 P_{t}\big(Z_{m,t+j}(\frac{l}{2^L}u,\frac{s}{T})\big) \, P_{t}\big(Z_{m,t+j^\prime}(\frac{l}{2^L}u,\frac{s^\prime}{T})\big) \bv \G_{t-1}\Big].
\end{align*}
By the ergodic theorem and a Riemann approximation 
\begin{align*}
&\frac{1}{|\mc{E}_{L,l,\flo{uT}}|}\sum_{t \in \mc{E}_{L,l,\flo{uT}}}
 \E\Big[\frac{1}{T^2} \sum_{s,s^\prime=1}^{\flo{uT}} P_{t}\big(Z_{m,t+j}(\frac{l}{2^L}u,\frac{s}{T})\big) \, P_{t}\big(Z_{m,t+j^\prime}(\frac{l}{2^L}u,\frac{s^\prime}{T})\big) \bv \G_{t-1}\Big]
 \\& \overset{p}{\to} \E\Big[u \int_0^1  P_{0}\big(Z_{m,j}(\frac{l}{2^L}u,xu)\big) dx 
 \cdot u\int_0^1 P_{0}\big(Z_{m,j^\prime}(\frac{l}{2^L}u, x^\prime u)\big) d x^\prime\Big]. 
\end{align*}
Hence, a change of variables yields
\begin{align*}
&\lim_{L\to \infty} \lim_{T\to \infty} \sum_{j,j^\prime=0}^{m-1}\E\Big[\frac{\flo{uT}}{T}\frac{1}{\flo{uT}} \sum_{t=1}^{\flo{uT}}\E\Big[  \Gamma_{m,T,t}\Big(\frac{t+j}{T}, \frac{t+j^\prime}{T}\Big) \Big \vert \G_{t-1}\Big]
\\&= \lim_{L\to \infty} \sum_{j,j^\prime=0}^{m-1}\frac{u}{2^L}\sum_{l=1}^{2^L}\E\Big[ \int_0^u  P_{0}\big(Z_{m,j}(\frac{l}{2^L}u,v)\big) dv 
 \cdot \int_0^u P_{0}\big(Z_{m,j^\prime}(\frac{l}{2^L}u,v^\prime)\big) d v^\prime\Big]  
\\& =\sum_{j,j^\prime=0}^{m-1}\int_0^u \E\Big[ \int_0^u  P_{0}\big(Z_{m,j}(\eta,v)\big) dv 
 \cdot \int_0^u P_{0}\big(Z_{m,j^\prime}(\eta,v^\prime)\big) d v^\prime\Big] d\eta. 
\end{align*}
Setting $g(\tilde{\mb{X}}_{j}(\eta),u,r)=\int_0^u \E_{\tilde{\mb{X}}^\prime_0(v)}[\hti(\tilde{\mb{X}}_{j}(\eta), \tilde{\mb{X}}^\prime_{0}(v),r)]dv$ and using \eqref{eq:limmdepapprox}, the dominated convergence theorem allows us to conclude that
\begin{align*}
 \limsup_{m \to \infty} \int_0^u \E \bv\sum_{j=0}^{m-1} \int_0^u  P_{0}\big(Z_{m,j}(\eta,v)\big) dv \bv^2dv = \int_0^u \sum_{\ell \in \znum} \text{Cov}\Big(g(\tilde{\mb{X}}_{\ell}(\eta),u,r),g(\tilde{\mb{X}}^\prime_{0}(\eta),u,r)\Big) d \eta.
\end{align*}
\end{proof}

\section{Proof of Theorem \ref{thm:2parproccon}}

\begin{proof}[Proof of \autoref{thm:2parproccon}]

 We make use of corollary in \cite{DZ08} (see also~\cite[thm. 2.3]{DGS21}), stated here for convenience.
\begin{thm} \label{bivfunctconv}
Let $\xi_T, T\ge 1$ be a $\rnum$-valued stochastic process defined on $[0,1]^2$ whose paths are in $D([0,1]^2)$ almost surely. Then $(\xi_T: T \ge 0)$ converges weakly to $\xi$ in $D([0,1]^2)$ if
\begin{enumerate}[label=(\roman*)] 
\item the fidis of $\xi_T$ converge to the corresponding fidis of a process $\xi$;
\item the process $\xi_T$ can be written as the difference of two coordinate-wise non-decreasing processes $\xi^o_T$ and $\xi^\star_T$;
\item there exists constants $\alpha \ge  \beta >2, c \in (0,\infty)$ such that for all $T$, $\E[\xi_T(0,0)]^\alpha \le c$ and a sequence $\Delta_T \to 0$ for which
\[
\E\big[\xi_T(u,r)-\xi_T(u^\prime,r^\prime)\big]^{\alpha} \le c \big\|(u,r)-(u^\prime,r^\prime)\big\|_{\infty}^{\beta} \quad \text{ whenever } \big\|(u,r)-(u^\prime,r^\prime)\big\|_{\infty} \ge \Delta_T;
\]
\item the process $\xi^\star_T$ satisfies
\begin{align*}
\max_{1\le i_1, i_2\le \flo{\Delta^{-1}_T}}& \bv \xi^\star_T\big(i_1 \Delta_T,i_2\Delta_T\big)-\xi^\star_T\big((i_1-1) \Delta_T,i_2\Delta_T\big) \bv
\\& + \bv \xi^\star_T\big(i_1 \Delta_T,i_2\Delta_T\big)-\xi^\star_T\big(i_1 \Delta_T,(i_2-1)\Delta_T\big) \bv \overset{p}{\to } 0.
\end{align*}
\end{enumerate}
\end{thm}
Specifically, we apply this to $\xi_T = T^{1/2}(\mathscr{S}_T-\E \mathscr{S}_T)$. In the rest of the proof, we assume without loss of generality that $[0,\mathscr{R}]=[0,1]$; otherwise we may simply show the result for
$\tilde{\xi}_{T}(u,r) =\xi_{T}(u,\mathscr{R} r)$. Convergence of the fidis was already established in   \autoref{thm:limfixr} . Observe that
\begin{align*}
    \xi_T(u,r) = T^{-3/2} \sum_{t,s=1}^{\flo{uT}} h(\tilde{\mb{X}}^{n(T)}_{t}, \tilde{\mb{X}}^{n(T)}_{s},r)-T^{-3/2} \sum_{t,s=1}^{\flo{uT}} \E\big[h(\tilde{\mb{X}}^{n(T)}_{t}, \tilde{\mb{X}}^{n(T)}_{s},r)\big]=:\xi^0_T(u,r)-\xi^\star_T(u,r)
\end{align*}
from which it is clear that (ii) holds.
 To verify (iii), we note that 
\begin{align*}
   \sum_{t,s=1}^{\flo{uT}}-\sum_{t,s=1}^{\flo{u^\prime T}} 
   =  2 \Big(\sum_{t=\flo{u^\prime T}+1}^{\flo{uT}}\sum_{s=1}^{t-1}+ \sum_{s=\flo{u^\prime T}+1}^{\flo{uT}}\sum_{t=\flo{u^\prime T}+1}^{s}\Big).
\end{align*}
Therefore,
\begin{align*}
 &   T^{3/2}\Big(\E\bv \xi_{T}(u,r)-\xi_{T}(u^\prime,r)\bv^{2p}\Big)^{1/2p}
\\& \le  \Big\| \sum_{t=\flo{u^\prime T}+1}^{\flo{uT}}\sum_{s=1}^{t-1}\hti(\tilde{\mb{X}}^{n(T)}_{t}, \tilde{\mb{X}}^{n(T)}_{s},r)\Big\|_{\rnum,2p} 
 +  \Big\| \sum_{s=\flo{u^\prime T}+1}^{\flo{uT}}\sum_{t=\flo{u^\prime T}+1}^{s}\hti(\tilde{\mb{X}}^{n(T)}_{t}, \tilde{\mb{X}}^{n(T)}_{s},r)\Big\|_{\rnum,2p}~.
\end{align*}
We focus on the first term as the second term can be treated in similar fashion. A similar argument as in \autoref{lem:errorcontrol} yields
\begin{align*}
     \Big\| \sum_{t=\flo{u^\prime T}+1}^{\flo{uT}}\sum_{s=1}^{t-1}\hti(\tilde{\mb{X}}^{n(T)}_{t}, \tilde{\mb{X}}^{n(T)}_{s},r)\Big\|_{\rnum,2p} &\le \sum_{j=0}^{\infty} \Big\| \sum_{t=\flo{u^\prime T}+1}^{\flo{uT}}P_{t-j}\Big(\sum_{s=1}^{t-1}\hti(\tilde{\mb{X}}^{n(T)}_{t}, \tilde{\mb{X}}^{n(T)}_{s},r)\Big)\Big\|_{\rnum,2p} 
        \\& \le \sum_{j=0}^{\infty} \Big( \sum_{t=\flo{u^\prime T}+1}^{\flo{uT}} \Big(\sum_{k=1}^{t-1}\Big\| P_{t-j}\Big(\hti(\tilde{\mb{X}}^{n(T)}_{t}, \tilde{\mb{X}}^{n(T)}_{t-k},r)\Big)\Big\|_{\rnum,2p}\Big)^2 \Big)^{1/2}
               \\& \le (\flo{u T}-\flo{u^\prime T})^{1/2}T \sup_{k\ge 0}\sum_{j=0}^{\infty} \sup_t\Big\| P_{t-j}\Big(\hti(\tilde{\mb{X}}^{n(T)}_{t}, \tilde{\mb{X}}^{n(T)}_{t-k},r)\Big)\Big\|_{\rnum,2p}
               \\& \le 2 C (u-u^\prime)^{1/2} T^{3/2},  
\end{align*}
where we used \autoref{thm:2parproccon}(i) and that  $(\flo{u T}-\flo{u^\prime T}) \le 2(u-u^\prime)T$. Thus, for some constant $c>0$
 \begin{align*}
&\E\bv \xi_{T}(u,r)-\xi_{T}(u^\prime,r)\bv^{2p}  \le c(u-u^\prime)^p. 
 \end{align*}

Furthermore, Minkowski's inequality yields 
 \begin{align*}
   \|\xi_{T}(u,r)-\xi_{T}(u,r^\prime)\|_{\rnum,2p}  
   \le  \sum_{j=1}^{\infty}&\Big\|T^{-3/2} \sum_{t=1}^{\flo{uT}}\sum_{k=1}^{t-1} P_{t-j}\Big(\hti(\tilde{\mb{X}}^{n(T)}_{t}, \tilde{\mb{X}}^{n(T)}_{t-k},r)-\hti(\tilde{\mb{X}}^{n(T)}_{t}, \tilde{\mb{X}}^{n(T)}_{t-k},r^\prime)\Big)\Big\|_{\rnum,2p}
  \\& +  \Big\|T^{-3/2} \sum_{t=1}^{\flo{uT}}\sum_{k=1}^{t-1} P_{t}\Big(\hti(\tilde{\mb{X}}^{n(T)}_{t}, \tilde{\mb{X}}^{n(T)}_{t-k},r)-\hti(\tilde{\mb{X}}^{n(T)}_{t}, \tilde{\mb{X}}^{n(T)}_{t-k},r^\prime)\Big)\Big\|_{\rnum,2p}. \tageq \label{eq:rrprime}
  \end{align*}
For the first term on the right-hand side, assumption (ii) of  \autoref{thm:2parproccon} yields
 \begin{align*}
   & \sum_{j=1}^{\infty}\Big\|T^{-3/2} \sum_{t=1}^{\flo{uT}}\sum_{k=1}^{t-1} P_{t-j}\Big(\hti(\tilde{\mb{X}}^{n(T)}_{t}, \tilde{\mb{X}}^{n(T)}_{t-k},r)-\hti(\tilde{\mb{X}}^{n(T)}_{t}, \tilde{\mb{X}}^{n(T)}_{t-k},r^\prime)\Big)\Big\|_{\rnum,2p}\\&\le    \sum_{j=1}^{\infty} \sup_t\min\Big(\sum_{l \in \{r,r^\prime \}}\Big\| P_{t-j}\Big(\hti(\tilde{\mb{X}}^{n(T)}_{t}, \tilde{\mb{X}}^{n(T)}_{t-k},l)\Big)\Big\|_{\rnum,2p},
    \Big\|P_{t-j}\Big( \hti(\tilde{\mb{X}}^{n(T)}_{t}, \tilde{\mb{X}}^{n(T)}_{t-k},r) - \hti(\tilde{\mb{X}}^{n(T)}_{t}, \tilde{\mb{X}}^{n(T)}_{t-k},r^\prime)\Big)\Big\|_{\rnum,2p}\Big)
    \\&\le    \sum_{j=1}^{\infty}\sup_t \min\Big( \sum_{l \in \{r,r^\prime \}}\Big\| P_{t-j}\big(\hti(\tilde{\mb{X}}^{n(T)}_{t}, \tilde{\mb{X}}^{n(T)}_{t-k},l)\big)\Big\|_{\rnum,2p} \,,\, 
    C(r-r^\prime)\Big)
    \\&   \le  2   \sum_{j \le (r-r^\prime)^{-1/2}}
    C(r-r^\prime)
    + \sum_{j > (r-r^\prime)^{-1/2}}   \sup_t\sum_{l \in \{r,r^\prime \}}\Big\| P_{t-j}\big(\hti(\tilde{\mb{X}}^{n(T)}_{t}, \tilde{\mb{X}}^{n(T)}_{t-k},l)\big)\Big\|_{\rnum,2p} \,,\,
    \\& \le C(r-r^\prime)^{1/2}+ \sum_{j >(r-r^\prime)^{-1/2}} \frac{j}{j} \sup_t \sum_{l \in \{r,r^\prime \}}\Big\| P_{t-j}\big(\hti(\tilde{\mb{X}}^{n(T)}_{t}, \tilde{\mb{X}}^{n(T)}_{t-k},l)\big)\Big\|_{\rnum,2p}
    \\& \le K(r-r^\prime)^{1/2},
 \end{align*}
 for some constant $K$, where we used assumption (i). 
To consider the second term of \eqref{eq:rrprime}, Rosenthals' inequality yields for $p >1$,
 \begin{align*}
   &\Big\|T^{-3/2} \sum_{t=1}^{\flo{uT}}\sum_{k=1}^{t-1} P_{t}\Big(\hti(\tilde{\mb{X}}^{n(T)}_{t}, \tilde{\mb{X}}^{n(T)}_{t-k},r)-\hti(\tilde{\mb{X}}^{n(T)}_{t}, \tilde{\mb{X}}^{n(T)}_{t-k},r^\prime)\Big)\Big\|^{2p}_{\rnum,2p}
   \\&\le  C T^{-3p}  \Big\{
   \E\Big[\Big(\sum_{t=1}^{\flo{uT}} \E\Big[P^2_t\big(Y_{t}(r,r^\prime)\big)\bv \G_{t-1}\Big]\Big)^p\Big]
+ \sum_{t=1}^{\flo{uT}} \E\bv P_{t}\Big(Y_{t}(r,r^\prime)\Big)\bv^{2p}\Big\}, \tageq \label{eq:rosenthal}
  \end{align*}
  where we set $$Y_{t}(r,r^\prime)= \sum_{k=1}^{t-1}  \big(\hti(\tilde{\mb{X}}^{n(T)}_{t}, \tilde{\mb{X}}^{n(T)}_{t-k},r)-\hti(\tilde{\mb{X}}^{n(T)}_{t}, \tilde{\mb{X}}^{n(T)}_{t-k},r^\prime)\big).$$
For the second term of \eqref{eq:rosenthal} we find with $\Delta_T=1/T$,
\[
C T^{-3p+2p+1}(r-r^\prime)  \le \tilde{C}(r-r^\prime)^{p-1+1} = \tilde{C}(r-r^\prime)^p,
\]
for some constant $\tilde{C}>0$. 
For the first term of \eqref{eq:rosenthal}, we note that almost surely
\begin{align*}
\E[ P^2_t(T^{-1}Y_{t}(r,r^\prime))|\G_{t-1}] 
& = \E[ P_t(T^{-1}Y_{t}(r,r^\prime)) T^{-1}Y_{t}(r,r^\prime)|\G_{t-1}]
\\& \le \E[ T^{-2}Y^2_{t}(r,r^\prime)|\G_{t-1}].
\end{align*}
Therefore,
\begin{align*}
    C T^{-p}  
   \E\Big[\Big(\sum_{t=1}^{\flo{uT}}\E[P^2_t(T^{-1}Y_{t}(r,r^\prime))|\G_{t-1}]\Big)^p\Big]
 &\le C T^{-p}  
   \E\Big[\Big(\sum_{t=1}^{\flo{uT}} \E[T^{-2}Y^2_{t}(r,r^\prime)|\G_{t-1}]\Big)^p\Big].
  \end{align*}
Furthermore, Jensen's inequality and condition (ii) of  \autoref{thm:2parproccon} yield
\begin{align*}
   \E[T^{-2}Y^2_{t}(r,r^\prime)|\G_{t-1}]& = \E[T^{-2}\sum_{k,k^\prime} Y_{k,t,r,r^\prime}Y_{k^\prime,t,r,r^\prime}|\G_{t-1}] \le \frac{1}{T}\sum_{k=1}^{T} \E[Y^2_{k,t,r,r^\prime}|\G_{t-1}]
  \\&  = \frac{1}{T}\sum_{k=1}^{T} \pr\Big(\partial_{\mathscr{B}}\big(\sh(\tilde{\mb{X}}^{n(T)}_{t}),
\sh(\tilde{\mb{X}}^{n(T)}_{t-k})\big)  \in (r,r^\prime]\bv\G_{t-1}\Big)
   =  O(r-r^{\prime}),
\end{align*}
almost surely. For the first term of \eqref{eq:rosenthal}, we therefore obtain
\[
 C T^{-p}  
   \E\Big[\Big(\sum_{t} \E[T^{-2}Y^2_{t}(r,r^\prime)|\G_{t-1}]\Big)^p\Big] \le c(r-r^{\prime})^p, 
\]
Combining terms, we find for \eqref{eq:rrprime}
  \begin{align*}
&\E\bv \xi_{T}(u,r)-\xi_{T}(u,r^\prime)\bv^{2p}  \le c(r-r^\prime)^{p}.  
 \end{align*}
Finally, we consider (iv). Denote
 \[
 \Upsilon_{T,t}(u,r) = T^{-1} \sum_{s=1}^{\flo{uT}} \E[h(\tilde{\mb{X}}^{n(T)}_{t}, \tilde{\mb{X}}^{n(T)}_{s},r)]~.
 \]
 Firstly, from the definition of \eqref{eq:Ymt} and the fact that the terms involved are bounded uniformly by 2,
\begin{align*}
\E &  \max_{1\le i_1,i_2 \le T} \bv \frac{1}{T^{1/2}}\sum_{t=1}^{i_1} \Upsilon_{T,t}(\frac{i_1}{T},\frac{i_2}{T})-\sum_{t=1}^{i_1-1} \Upsilon_{T,t}(\frac{i_1-1}{T},\frac{i_2}{T})\bv\\ & \le \frac{1}{T^{1/2}}  \E \max_{1\le i_1,i_2 \le T}\Big\{\sum_{t=1}^{i_1} 
   \bv\Upsilon_{T,t}(\frac{i_1}{T},\frac{i_2}{T})-\Upsilon_{T,t}(\frac{i_1-1}{T},\frac{i_2}{T})\bv
\\& \phantom{\le \frac{1}{T^{1/2}}  \E \max_{1\le i_1,i_2 \le T}\Big\{}+\bv \sum_{t=1}^{i_1} \Upsilon_{T,t}(\frac{i_1-1}{T},\frac{i_2}{T})-\sum_{t=1}^{i_1-1}\Upsilon_{T,t}(\frac{i_1-1}{T},\frac{i_2}{T})\bv\Big\}
= O(\frac{1}{T^{1/2}}). 
 \tageq \label{eq:xistartime}
\end{align*}
Furthermore, the triangle yields 
\begin{align*}
& \E  \max_{1\le i_1, i_2\le T} \bv \xi^\star_T\big(\frac{i_1}{T},\frac{i_2}{T}\big)-\xi^\star_T\big(\frac{i_1}{T},\frac{i_2-1}{T}\big) \bv \\&\le  2 \frac{1}{T^{3/2}}\E   \max_{1\le i_1,i_2 \le T} \sum_{t=1}^{i_1}\sum_{k=1}^{t-1}  \bv \E\big[h(\tilde{\mb{X}}^{n(T)}_{t}, \tilde{\mb{X}}^{n(T)}_{t-k},\frac{i_2}{T})\big]- \E\big[h(\tilde{\mb{X}}^{n(T)}_{t}, \tilde{\mb{X}}^{n(T)}_{t-k},\frac{i_2-1}{T})\big]\bv.
\end{align*}
For the  summand, the tower property and \eqref{densuni1} yield
\begin{align*}
 \bv \E\big[h(\tilde{\mb{X}}^{n(T)}_{t}, \tilde{\mb{X}}^{n(T)}_{t-k},b)\big]- \E\big[h(\tilde{\mb{X}}^{n(T)}_{t}, \tilde{\mb{X}}^{n(T)}_{t-k},a)\big]\bv \tageq\label{eq:mm3}
&=\E\Big[\pr\Big(\partial_{\mathscr{B}}\big(\sh(\tilde{\mb{X}}^{n(T)}_{t}),
\sh(\tilde{\mb{X}}^{n(T)}_{t-k})\big) \in (a,b]\bv \G_{t-1}\Big)\Big] 
\\&\le C^\prime |b-a|.
\end{align*}
Hence, 
\begin{align*}
&\max_{1\le i_1, i_2\le T} \bv \xi^\star_T\big(\frac{i_1}{T},\frac{i_2}{T}\big)-\xi^\star_T\big(\frac{i_1}{T},\frac{i_2-1}{T}\big)\bv \le 
T^{-3/2} \max_{1\le i_1, i_2\le T} \sum_{t,s=1}^{i_1} C^\prime T^{-1} = O(T^{-1/2}). \tageq \label{eq:xistarrad}
\end{align*}
Thus (iv) is implied by \eqref{eq:xistartime} and \eqref{eq:xistarrad}. 
\end{proof}

\end{document}